\documentclass[11pt]{amsart}
\usepackage{amsmath}
\usepackage{amssymb}
\usepackage{amsthm}
\usepackage{amsfonts}
\usepackage{tikz}
\usetikzlibrary{decorations.markings}
\usepackage{graphicx}
\usepackage{hyperref}
\usepackage{fullpage}
\usepackage{float}
\usepackage[hang]{caption}
\usepackage{xcolor}
\usepackage[all]{xy}
\usetikzlibrary{arrows.meta}
\usepackage{mathdots}
\usepackage{enumerate}

\newcounter{intro}

\newtheorem{thm}{Theorem}[section]
\newtheorem{prop}[thm]{Proposition}
\newtheorem{cor}[thm]{Corollary}
\newtheorem{lem}[thm]{Lemma}
\newtheorem{conj}[intro]{Conjecture}
\newtheorem{mainthm}[intro]{Main Theorem}

\theoremstyle{definition}
\newtheorem{defn}[thm]{Definition}
\newtheorem{rmk}[thm]{Remark}
\newtheorem{exmp}[thm]{Example}

\newcommand{\ww}{\mathfrak{w}}

\newcommand{\Gr}{\mathrm{Gr}}

\newcommand{\SL}{\mathrm{SL}}

\newcommand{\cA}{\mathcal{A}}
\newcommand{\cX}{\mathcal{X}}

\newcommand{\cM}{\mathcal{M}}

\newcommand{\cW}{\mathcal{W}}
\newcommand{\cO}{\mathcal{O}}
\newcommand{\sgn}{\mathrm{sign}}

\newcommand{\inprod}[2]{\left\langle #1, #2\right\rangle}
\newcommand{\doverline}[1]{\overline{\overline{#1}}}
\newcommand{\abs}{\mathrm{abs}}
\newcommand{\cB}{\mathcal{B}}
\newcommand{\fr}{\mathrm{fr}}

\newcommand{\Hom}{\mathrm{Hom}}
\newcommand{\Sk}{\mathrm{Sk}}

\makeatletter
\let\c@equation\c@thm
\makeatother
\numberwithin{equation}{section}

\colorlet{lightblue}{blue!30!white}
\colorlet{lightred}{red!30!white}
\colorlet{lightteal}{teal!30!white}

\title{Weighted Cycles on Weaves}

\author{Daping Weng}

\begin{document}

\begin{abstract} We introduce weighted cycles on weaves of general Dynkin types and define a skew-symmetrizable intersection pairing between weighted cycles. We prove that weighted cycles on a weave form a Laurent polynomial algebra and construct a quantization for this algebra using the skew-symmetric intersection pairing in the simply-laced case. We define merodromies along weighted cycles as functions on the decorated flag moduli space of the weave. We relate weighted cycles with cluster variables in a cluster algebra and prove that mutations of weighted cycles are compatible with mutations of cluster variables.  
\end{abstract}

\maketitle

\setcounter{tocdepth}{2}
\tableofcontents

\section{Introduction}

Weaves were introduced by Casals and Zaslow \cite{CZ} as a graphical tool to describe a family of Legendrian surfaces living inside the $1$-jet space of a base surface. Soon after their introduction, weaves were adapted for the study cluster algebras, with many pieces of literature devoted to the development of weave-based descriptions of cluster structures, including \cite{Hughes,ABL,CGGS,CLSBW}. In particular, weaves were generalized to all Dynkin types in \cite{CGGLSS} and played an important role in the construction of cluster structures on braid varieties. 

Under the generalization in \textit{loc. cit.}, the original version of weaves by Casals and Zaslow, which can also be referred to as \emph{Legendrian weaves}, would be categorized as Dynkin type A. In \cite{CW}, Casals and the author gave a topological interpretation of cluster structures described by weaves of Dynkin type A, which can be summarized by the following table.
    \begin{table}[H]
        \centering
        \begin{tabular}{|c|c|}\hline
             \textbf{Cluster Algebras} & \textbf{Legendrian Weaves}  \\ \hline
             Cluster seed &Exact Lagrangian surface $S$ of a weave   \\ \hline 
            Quiver vertices &$\begin{array}{c} \text{Special collection of 1-cycles}\\
            \text{on $S$ (also called Y-cycles)}\end{array}$\\ \hline
            Quiver arrows & $\begin{array}{c} \text{Intersection pairing}\\
            \text{between Y-cycles}\end{array}$\\ \hline
            Cluster variables & $\begin{array}{c} \text{Merodromies along relative}\\
            \text{1-cycles dual to Y-cycles}\end{array}$\\ \hline
            Cluster mutation & Polterovich surgery on $S$  \\ \hline
        \end{tabular}
        \caption{Parallels between Legendrian weaves and their cluster algebras.}
        \label{tab:my_label}
    \end{table}

Although Y-cycles together with their pairings and mutations have been generalized to weaves of general Dynkin types in \cite{CGGLSS}, the rest of the parallels have been absent for general Dynkin types. In this article, we introduce \emph{weighted cycles} on the base surface to fill this gap for general Dynkin types. To put it simply, we define a weighted chain to be a relative 1-chain on the base surface that intersects the weave generically, with the data of a weight (of the corresponding simply-connected Lie group) attached to each segment in the complement of the intersection with the weave; we also introduce homotopies of weighted chains and a condition on how to glue them together to form a weighted cycle. 

We define a commutative algebra $\cW(\ww)$ on the space of formal linear combinations of weighted cycles called the \emph{weighted cycle algebra}, where the product is defined by stacking weighted cycles on top of each other diagrammatically. The weighted cycle algebra $\cW(\ww)$ is a topological incarnation of the Laurent polynomial ring of a cluster seed, as the following theorem indicates.

\begin{mainthm}[Theorem \ref{thm: weighted cycle algebra is a torus}] The weighted cycle algebra $\cW(\ww)$ is a Laurent polynomial algebra of rank $\tau+r(\beta-1)$, where $\tau$ is the number of trivalent vertices in $\ww$, $r$ is the rank of the Dynkin type, and $\beta$ is the number of boundary base points.
\end{mainthm}

Moreover, we define a homotopy-invariant skew-symmetrizable intersection pairing between weighted cycles, which allows us to introduce a quantization $\mathbb{W}(\ww)$ of the weighted cycle algebra in the simply-laced case. We also describe weighted cycle representatives for Y-cycles, and prove the following.

\begin{mainthm}[Theorem \ref{thm: recovering pairing between Y-cycles}] The intersection pairing between weighted cycle representatives of Y-cycles recovers the skew-symmetric pairing between Y-cycles.
\end{mainthm}

A weave also defines a flag moduli space on the base surface, for which we associate a flag (of the same Dynkin type) to each face of the weave and impose a relative position condition on each pair of flags on adjacent faces according to the color of the weave edge. In \cite{CW}, Casals and the author defined a nowhere-vanishing function called \emph{merodromy} for each relative cycle that intersects a Legendrian (Dynkin type A) weave generically. Simply speaking, the flag data give rise to a rank 1 local system on the weave surface; thus, given two non-zero vectors at the two ends of a relative 1-cycle, one can parallel transport the vector at the source to the other endpoint (target) of the relative cycle using the rank 1 local system, and then take the ratio with the vector sitting at the target; this ratio is the merodromy of the relative 1-cycle.

Although a topological interpretation in terms rank 1 local systems on surfaces is no longer available for other Dynkin types, we manage to define a similar nowhere-vanishing function for each weighted cycle on the flag moduli space for weaves of a general Dynkin type. This result generalizes the aforementioned merodromy construction from \cite{CW}. By choosing a suitable collection of weighted cycles, we can represent cluster variables as merodromies, as the following theorem implies.

\begin{mainthm}[Theorem \ref{thm: mutation}] Merodromies of the positive weighted cycles transform according to the cluster mutation formula under a mutation of Y-cycle.
\end{mainthm}

Furthermore, we observe that weighted cycles share many similarities with \emph{webs} \cite{Kuperberg}, which are generators of skein algebras: for example, weighted cycles and webs are both attaching weight data to a network of oriented curves, they both admit quantizations, and they both can give rise to elements in cluster algebras. Yet there is one key difference: the multiplication between weighted cycles is $q$-commutative, whereas the multiplication between webs obeys the skein relation. We believe that weighted cycles should be viewed as localizations of webs, and we further conjecture the following. In particular, the case $G=\SL_2$ is in parallel with \cite{Mul}.

\begin{conj} Let $\Sk(S,G)$ denote the quantum skein algebra associated with a simply-connected Lie group $G$ on a surface $S$. Let $\ww$ be a $G$-weave arising from an ideal triangulation on $S$. Then there exists a quantum algebra homomorphism $\Sk(S,G)\rightarrow \mathbb{W}(\ww)$, mapping each web $W$ to a linear combination of weighted cycles that share the same topological support as $W$.
\end{conj}

\begin{figure}[H]
    \centering
    \begin{tikzpicture}[scale=1.5]
        \foreach \i in {0,1,2}
        {
        \draw [very thick] (90+\i*120:2) -- (-30+\i*120:2);
        \draw [decoration={markings, mark=at position 0.5 with {\arrow{>}}}, postaction={decorate}] (90+\i*120:2) -- (0,0);
        \node at (70+\i*120:0.7) [] {$\omega_2$};
        }
    \end{tikzpicture}\hspace{2cm}
        \begin{tikzpicture}[scale=1.5]
    \foreach \i in {0,1,2}
    {
    \draw [very thick] (90+120*\i:2) -- (-30+120*\i:2);
    \draw [blue] (30+120*\i:1) -- (0,0);
    \draw [red] (90+120*\i:1) -- (0,0);
    \draw [red] (110+120*\i:1.34) -- (90+120*\i:1) -- (70+120*\i:1.34);
    }
    \draw [dashed, decoration={markings, mark=at position 0.5 with {\arrow{>}}}, postaction={decorate}] (90:2) -- node [left] {$-\omega_1$} (0.2,0.5);
    \draw [dashed, decoration={markings, mark=at position 0.5 with {\arrow{>}}}, postaction={decorate}] (-30:2) node [above left] {$-\omega_1 \quad$} -- (0.2,0.5) node [right] {$\quad \omega_1-\omega_2$};
    \draw [dashed, decoration={markings, mark=at position 0.5 with {\arrow{>}}}, postaction={decorate}] (-150:2) node [above right] {$\quad -\omega_1$} -- (0.2,0.5);
    \node at (0.1,0.2) [] {$\omega_2$};
    \node at (0.2,0.5) [] {$\bullet$};
    \end{tikzpicture}
    \caption{Left: an $\SL_3$-web. Right: the corresponding weighted cycle.}
\end{figure}
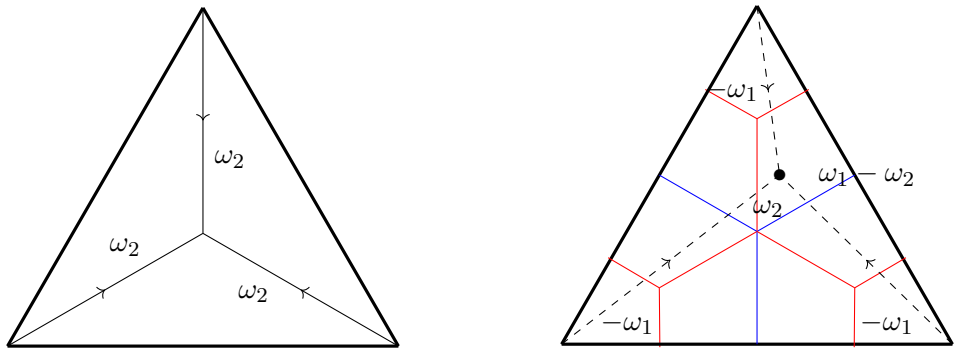

The article is structured as follows: Section \ref{sec2} reviews some basic background on decorated flags and weaves; Section \ref{sec3} gives the definitions of weighted cycles and related concepts, and proves our main theorems in the simply-laced case; Section \ref{sec4} revisits some well-known constructions in cluster algebras and showcases how to describe them using weighted cycles; Section \ref{sec5} goes over the construction of (co)weighted cycles and their skew-symmetrizable intersection pairing in the non-simply-laced case.

\subsection*{Acknowledgements} The author would like to thank Roger Casals, Honghao Gao, James Hughes, Thang T. Q. Le, Lenhard Ng, Linhui Shen, Zhe Sun, and Eric Zaslow for great inspiration and cheerful discussions through the development of this project.

\section{Preliminaries}\label{sec2}

\subsection{Weaves and Y-cycles} Weaves were first introduced in \cite{CZ} and they were originally devised as combinatorial tools to describe Legendrian surfaces in 1-jet spaces. Weaves were later generalized to all Dynkin types in \cite{CGGLSS}, and weaves of Dynkin type A were the ones associated with Legendrian surfaces. In this subsection, we briefly review some basics about weaves of simply-laced Dynkin types, and for simplicity, we mostly only consider weaves on a disk (with Subsection \ref{subsec: quantum group} as an exception).

Recall that simply-laced Dynkin diagrams are classified into two infinite families (types A and D) with three exceptions ($\mathrm{E}_6$, $\mathrm{E}_7$, and $\mathrm{E}_8$). They are of the following forms.
\begin{figure}[H]
    \centering
    \begin{tikzpicture}
        \foreach \i in {1,2,4,5}{
        \node (\i) at (\i,0) [] {$\bullet$};
        }
        \node (3) at (3,0) [] {$\cdots$};
        \draw (1) -- (2) -- (3) -- (4) -- (5);
        \node at (1) [below] {$1$};
        \node at (2) [below] {$2$};
        \node at (4) [below] {$n-1$};
        \node at (5) [below] {$n$};
        \node at (3,-1) [] {$\mathrm{A}_n$};
    \end{tikzpicture}\hspace{1cm}
    \begin{tikzpicture}
        \foreach \i in {1,2,4}{
        \node (\i) at (\i,0) [] {$\bullet$};
        }
        \node (3) at (3,0) [] {$\cdots$};
        \node (5) at (5,0.5) [] {$\bullet$};
        \node (6) at (5,-0.5) [] {$\bullet$};
        \draw (1) -- (2) -- (3) -- (4) -- (5);
        \draw (4) -- (6);
        \node at (1) [below] {$1$};
        \node at (2) [below] {$2$};
        \node at (4) [below] {$n-2$};
        \node at (5) [above] {$n-1$};
        \node at (6) [below] {$n$};
        \node at (3,-1) [] {$\mathrm{D}_n$};
    \end{tikzpicture}\hspace{1cm}
    \begin{tikzpicture}
        \foreach \i in {1,...,5}
        {
        \node (\i) at (\i,0) [] {$\bullet$};
        }
        \node (6) at (3,-1) [] {$\bullet$};
        \draw (1) -- (2) -- (3) -- (4) -- (5);
        \draw (3) -- (6);
        \node at (1) [above] {$1$};
        \node at (2) [above] {$2$};
        \node at (3) [above] {$3$};
        \node at (6) [right] {$4$};
        \node at (4) [above] {$5$};
        \node at (5) [above] {$6$};
        \node at (3,-1.5) [] {$\mathrm{E}_6$};
    \end{tikzpicture}\\\begin{tikzpicture}
        \foreach \i in {1,...,6}
        {
        \node (\i) at (\i,0) [] {$\bullet$};
        }
        \node (7) at (3,-1) [] {$\bullet$};
        \draw (1) -- (2) -- (3) -- (4) -- (5) -- (6);
        \draw (3) -- (7);
        \node at (1) [above] {$1$};
        \node at (2) [above] {$2$};
        \node at (3) [above] {$3$};
        \node at (7) [right] {$4$};
        \node at (4) [above] {$5$};
        \node at (5) [above] {$6$};
        \node at (6) [above] {$7$};
        \node at (3,-1.5) [] {$\mathrm{E}_7$};
    \end{tikzpicture}\hspace{2cm}
    \begin{tikzpicture}
        \foreach \i in {1,...,7}
        {
        \node (\i) at (\i,0) [] {$\bullet$};
        }
        \node (8) at (3,-1) [] {$\bullet$};
        \draw (1) -- (2) -- (3) -- (4) -- (5) -- (6) --(7);
        \draw (3) -- (8);
        \node at (1) [above] {$1$};
        \node at (2) [above] {$2$};
        \node at (3) [above] {$3$};
        \node at (8) [right] {$4$};
        \node at (4) [above] {$5$};
        \node at (5) [above] {$6$};
        \node at (6) [above] {$7$};
        \node at (7) [above] {$8$};
        \node at (3,-1.5) [] {$\mathrm{E}_8$};
    \end{tikzpicture}
    \caption{Simply-laced Dynkin diagrams.}
\end{figure}
Each Dynkin diagram defines a Coxeter group $W$, which is generated by \emph{simple reflections} $s_i$ that are in bijection with the vertices in the Dynkin diagram. These simple reflections satisfy the following relations:
\begin{itemize}
    \item $s_i^2=e$,
    \item $s_is_j=s_js_i$ if $i\neq j$ are two non-adjacent vertices in the Dynkin diagram,
    \item $s_is_js_i=s_js_is_j$ if $i\neq j$ are two adjacent vertices in the Dynkin diagram.
\end{itemize}
The last two relations are also known as the \emph{braid relations}.

There is a faithful action of $W$ on the Euclidean space $\mathbb{R}^n$ by reflections over hyperplanes, with $n$ being the number of vertices in the Dynkin diagram. To define such an action, we first fix a basis $\{\alpha_i\}$ such that  
\begin{itemize}
    \item $(\alpha_i,\alpha_i)=2$,
    \item $(\alpha_i,\alpha_j)=0$ if $i\neq j$ are two non-adjacent vertices in the Dynkin diagram,
    \item $(\alpha_i,\alpha_j)=-1$ if $i\neq j$ are two adjacent vertices in the Dynkin diagram.
\end{itemize}
Then the action of the simple reflection $s_i$ is defined to be the reflection over the hyperplane orthogonal to $\alpha_i$; in formula, the action of $s_i$ is
\[
s_i.v:=v-(\alpha_i,v)\alpha_i.
\]
These basis vectors $\alpha_i$'s are called \emph{simple roots}.

\begin{defn} A \emph{weave} of simply-laced Dynkin types (ADE) is a planar graph embedded in the disk with edges labeled by the simple reflections of that Dynkin type, such that each vertex is one of the three types listed in Figure \ref{fig: weave vertices}. A weave edge is said to be \emph{external} if it is incident to the boundary of the disk; otherwise it is said to be \emph{internal}.
\end{defn}

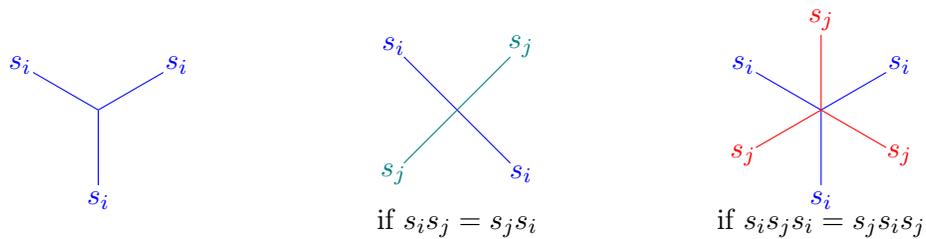
\begin{figure}[H]
    \centering
    \begin{tikzpicture}[baseline=0]
    \draw [blue] (0,-1) node [below] {$s_i$} -- (0,1) node [above] {$s_i$};
    \node [blue] at (0,0) [] {$\bullet$};
    \end{tikzpicture} \hspace{2cm}
    \begin{tikzpicture}[baseline=0]
        \draw[blue] (150:1) -- (0,0);
        \draw[blue] (30:1) -- (0,0);
        \draw[blue] (0,0) -- (0,-1);
        \foreach \i in {1,2,3}
        {
        \node [blue] at (30+120*\i:1.2) [] {$s_i$};
        }
        \end{tikzpicture}
\hspace{2cm}
    \begin{tikzpicture}[baseline=0]
        \draw[blue] (135:1) -- (0,0);
        \draw[blue] (0,0) -- (-45:1);
        \draw[teal] (45:1) -- (0,0);
        \draw[teal] (0,0) -- (-135:1);
        \foreach \i in {1,2}
        {
        \node [blue] at (135+180*\i:1.2) [] {$s_i$};
        \node [teal] at (45+180*\i:1.2) [] {$s_j$};
        }
        \node at (0,-1.5) [] {if $s_is_j=s_js_i$};
    \end{tikzpicture}        \hspace{2cm}\begin{tikzpicture}[baseline=0]
        \draw[blue] (150:1) -- (0,0);
        \draw[blue] (30:1) -- (0,0);
        \draw[red] (0,1) -- (0,0);
        \draw[blue] (0,0) -- (0,-1);
        \draw[red] (0,0) -- (-30:1);
        \draw[red] (0,0) -- (-150:1);
        \foreach \i in {1,2,3}
        {
        \node [blue] at (30+120*\i:1.2) [] {$s_i$};
        \node [red] at (90+120*\i:1.2) [] {$s_j$};
        }
        \node at (0,-1.5) [] {if $s_is_js_i=s_js_is_j$};
    \end{tikzpicture}
    \caption{Allowable vertices in a weave.}
    \label{fig: weave vertices}
\end{figure}

\begin{rmk} The bivalent weave vertex is a new vertex type that we are introducing in this article. It serves the purpose of resolving conflicting orientations on weave edges, which are needed to define merodromies later in Subsection \ref{subsec: 3.6}.
\end{rmk}

\begin{defn} In addition to the definition of weaves above, we would like to introduce \emph{boundary base points}: they are a finite collection of points along the boundary of the disk away from any external weave edges. We require each weave to have at least one boundary base point. The connected components of the complement of boundary base points are called \emph{boundary intervals}.
\end{defn}

\begin{defn}\label{defn: Y-cycle} Let $\ww$ be a weave and let $E(\ww)$ be the set of weave edges in $\ww$. A \emph{Y-cycle} is a map $\gamma:E(\ww)\rightarrow \mathbb{Z}_{\geq 0}$ satisfying the following three conditions:
\begin{itemize}
\item Between the two weave edges $a$ and $b$ incident to a bivalent weave vertex, $\gamma(a)=\gamma(b)$.
    \item Among the three weave edges $a,b,c$ incident to a trivalent weave vertex, the minimum of $\gamma(a)$, $\gamma(b)$, and $\gamma(c)$ is achieved at least twice.
    \item Among the four weave edges $a,b,c,d$ incident to a tetravalent weave vertex (in a cyclic order), $\gamma(a)=\gamma(c)$ and $\gamma(b)=\gamma(d)$.
    \item Among the six weave edges $a,b,c,d,e,f$ incident to a hexavalent weave vertex (in a cyclic order), $\gamma(a)-\gamma(d)=\gamma(e)-\gamma(b)=\gamma(c)-\gamma(f)$. 
\end{itemize}
The \emph{support} of a Y-cycle is the subgraph of $\ww$ spanned by the weave edges $e$ with $\gamma(e)>0$. The number $\gamma(e)$ is called the \emph{mulitplicity} of $\gamma$ at $e$. A Y-cycle is \emph{unfrozen} if $\gamma(e)=0$ for all external weave edges $e$; a Y-cycle is \emph{frozen} if it is not unfrozen. A set of Y-cycles is said to be \emph{linearly independent} if they are linearly independent as functions on $E(\ww)$.
\end{defn}

\begin{defn}\label{defn: pairing between Y-cycles} On a weave $\ww$, let $Y(\ww)$ be the set of Y-cycles, $V(\ww)$ be the set of vertices, and $X(\ww)$ be the set of external weave edges. Let $\frac{1}{2}\mathbb{Z}=\{0,\pm \frac{1}{2},\pm 1,\pm\frac{3}{2},\dots\}$ denote the set of half-integers. The \emph{intersection pairing} between Y-cycles is a map $\{\cdot, \cdot\}:Y(\ww)\times Y(\ww)\rightarrow\frac{1}{2}\mathbb{Z}$ defined by
\[
\{\gamma,\gamma'\}:=\sum_{v\in V(\ww)}\{\gamma,\gamma'\}_v+\sum_{e,e'\in X(\ww)}\gamma(e)\gamma'(e')\{e,e'\}.
\]
The terms in the first summation are defined by
\[
\{\gamma,\gamma'\}_v=\left\{\begin{array}{ll}
\det \begin{pmatrix}
    1 & 1 & 1 \\ \gamma(a) & \gamma(b) &\gamma(c)\\ \gamma'(a) & \gamma'(b) &\gamma'(c)
\end{pmatrix} & \text{if $\begin{tikzpicture}[baseline=0]
    \draw[blue] (150:1) -- node [above] {$b$} (0,0);
        \draw[blue] (30:1) -- node [above] {$a$} (0,0);
        \draw[blue] (0,0) -- node [left] {$c$} (0,-1);
        \node at (0,0) [below right] {$v$};
\end{tikzpicture}$} \\
& \\ 
\dfrac{1}{2}\left(\det \begin{pmatrix}
    1 & 1 & 1 \\ \gamma(a) & \gamma(c) &\gamma(e)\\ \gamma'(a) & \gamma'(c) &\gamma'(e)
\end{pmatrix}+\det \begin{pmatrix}
    1 & 1 & 1 \\ \gamma(b) & \gamma(d) &\gamma(f)\\ \gamma'(b) & \gamma'(d) &\gamma'(f)
\end{pmatrix}\right) & \text{if $\begin{tikzpicture}[baseline=0]
\draw[blue] (150:1)node [above] {$c$}  -- (0,0);
        \draw[blue] (30:1) node [above] {$a$}  -- (0,0);
        \draw[red] (0,1) node [right] {$b$} --  (0,0);
        \draw[blue] (0,0) -- (0,-1)node [right] {$e$} ;
        \draw[red] (0,0) --  (-30:1)node [below] {$f$};
        \draw[red] (0,0) --  (-150:1)node [below] {$d$};
        \node at (0,0) [left] {$v$};
\end{tikzpicture}$}\\
& \\
0 & \text{otherwise}.
\end{array}\right.
\]
The bracket $\{e,e'\}$ in the second summation is skew-symmetric and is $0$ unless both $e$ and $e'$ are on the same boundary interval. When they are on the same boundary interval, let us assume without loss of generality that $e'$ preceeds $e$ in the clockwise direction, separated by external weave edges of colors $s_{i_1},s_{i_2},\dots, s_{i_{k-1}}$ in the clockwise direction (as in Figure \ref{fig: intersection between frozen Y-cycles}). Suppose $e$ is of color $s_{i_0}$ and $e'$ is of color $s_{i_k}$. We then define 
\[
\{e,e'\}=\frac{1}{2}(\alpha_{i_0}, s_{i_1}\cdots s_{i_{k-1}}.\alpha_{i_k}),
\]
where $\alpha_{i_0}$ and $\alpha_{i_k}$ are simple roots.
\end{defn}
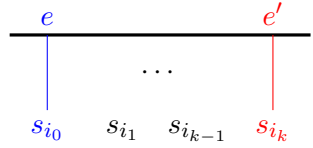
\begin{figure}[H]
    \centering
    \begin{tikzpicture}
        \draw [very thick] (-0.5,0) -- (3.5,0);
        \draw [blue] (0,0) node [above] {$e$} -- (0,-1) node [below] {$s_{i_0}$};
        \draw [red] (3,0) node [above] {$e'$} -- (3,-1) node [below] {$s_{i_k}$};
        \node at (1,-1) [below] {$s_{i_1}$};
        \node at (2,-1) [below] {$s_{i_{k-1}}$};
        \node at (1.5,-0.5) [] {$\cdots$};
    \end{tikzpicture}
    \caption{Intersection pairing between frozen Y-cycles.}
    \label{fig: intersection between frozen Y-cycles}
\end{figure}

\begin{lem} At a hexavalent weave vertex, 
\begin{align*}
\det \begin{pmatrix}
    1 & 1 & 1 \\ \gamma(a) & \gamma(c) &\gamma(e)\\ \gamma'(a) & \gamma'(c) &\gamma'(e)
\end{pmatrix}=&\det \begin{pmatrix}
    1 & 1 & 1 \\ \gamma(b) & \gamma(d) &\gamma(f)\\ \gamma'(b) & \gamma'(d) &\gamma'(f)
\end{pmatrix}\\
=&\frac{1}{2}\left(\det \begin{pmatrix}1 & 1 & 1 \\ \gamma(c) & \gamma(b) & \gamma(a) \\ \gamma'(c) & \gamma'(b) & \gamma'(a)\end{pmatrix} - \det \begin{pmatrix}1 & 1 & 1 \\ \gamma(d) & \gamma(e) & \gamma(f) \\ \gamma'(d) & \gamma'(e) & \gamma'(f)\end{pmatrix}\right).
\end{align*}
Thus, the intersection pairing between Y-cycles at a hexavalent weave vertex can be replaced by either of the determinants in this lemma, and either of them is equal to the pairing given in \cite{CGGLSS}.
\end{lem}
\begin{proof} On the one hand, let $u=\gamma(a)-\gamma(d)=\gamma(e)-\gamma(b)=\gamma(c)-\gamma(f)$ and let $u'$ be defined similarly for $\gamma'$. Then 
\[
   \det \begin{pmatrix}
    1 & 1 & 1 \\ \gamma(b) & \gamma(d) &\gamma(f)\\ \gamma'(b) & \gamma'(d) &\gamma'(f)
\end{pmatrix}=\det \begin{pmatrix}
    1 & 1 & 1 \\ \gamma(b)+u & \gamma(d)+u &\gamma(f)+u\\ \gamma'(b)+u' & \gamma'(d)+u' &\gamma'(f)+u'\end{pmatrix}
    =\det \begin{pmatrix}
    1 & 1 & 1 \\ \gamma(a) & \gamma(c) &\gamma(e)\\ \gamma'(a) & \gamma'(c) &\gamma'(e)
\end{pmatrix}.
\]
On the other hand, note that
\begin{align*}&\det \begin{pmatrix}1 & 1 & 1 \\ \gamma(c) & \gamma(b) & \gamma(a) \\ \gamma'(c) & \gamma'(b) & \gamma'(a)\end{pmatrix} - \det \begin{pmatrix}1 & 1 & 1 \\ \gamma(d) & \gamma(e) & \gamma(f) \\ \gamma'(d) & \gamma'(e) & \gamma'(f)\end{pmatrix}\\
=&\gamma(c')(\gamma(a)-\gamma(b))+\gamma'(b)(\gamma(c)-\gamma(a))+\gamma'(a)(\gamma(b)-\gamma(c))\\
&-\gamma'(d)(\gamma(f)-\gamma(e))-\gamma'(e)(\gamma(d)-\gamma(f))-\gamma'(f)(\gamma(e)-\gamma(d))\\
=&\gamma(c') \gamma(a)+\gamma(b)(\gamma'(a)-\gamma'(c))+\gamma'(b)(\gamma(c)-\gamma(a))-\gamma'(a)\gamma(c)\\
&-\gamma'(d)\gamma(f)+\gamma(e)(\gamma'(d)-\gamma'(f))-\gamma'(e)(\gamma(d)-\gamma(f))+\gamma'(f)\gamma(d)\\
=&\gamma(c') \gamma(a)+\gamma(b)(\gamma'(f)-\gamma'(d))+\gamma'(b)(\gamma(f)-\gamma(d))-\gamma'(a)\gamma(c)\\
&-\gamma'(d)\gamma(f)+\gamma(e)(\gamma'(a)-\gamma'(c))-\gamma'(e)(\gamma(a)-\gamma(c))+\gamma'(f)\gamma(d)\\
=&\det \begin{pmatrix}
    1 & 1 & 1 \\ \gamma(a) & \gamma(c) &\gamma(e)\\ \gamma'(a) & \gamma'(c) &\gamma'(e)
\end{pmatrix}+\det \begin{pmatrix}
    1 & 1 & 1 \\ \gamma(b) & \gamma(d) &\gamma(f)\\ \gamma'(b) & \gamma'(d) &\gamma'(f)
\end{pmatrix}.\qedhere
\end{align*}
\end{proof}

\begin{defn}\label{defn: weave equivalences} There are certain moves between weaves that are considered \emph{equivalences} (Figure \ref{fig: weave equivalences}). Weaves that are related by weave equivalences are said to be \emph{equivalent}. Note that there is a natural bijection between Y-cycles on equivalent weaves: the correspondence is dictated by the multiplicities on the surrounding weave edges together with the conditions on multiplicities at the weave vertices.
\end{defn}
\begin{figure}[H]
    \centering
    \begin{tikzpicture}[baseline=0,scale=0.7]
        \draw [blue] (-1,-1) -- (-0.5,0) -- (-1,1);
        \draw [red] (-1,0) -- (-0.5,0);
        \draw [red] (-0.5,0) to [out=60,in=180] (0,0.5) to [out=0,in=120] (0.5,0) to [out=-120,in=0] (0,-0.5) to [out=180,in=-60] (-0.5,0);
        \draw [blue] (-0.5,0) -- (0.5,0);
        \draw [red] (0.5,0) -- (1,0);
        \draw [blue] (1,-1) -- (0.5,0) -- (1,1);
    \end{tikzpicture} \quad $\overset{\text{I}}{\longleftrightarrow}$ \quad \begin{tikzpicture}[baseline=0, scale=0.7]
    \draw [blue] (-1,-0.8) -- (1,-0.8);
    \draw [red] (-1,0) -- (1,0);
    \draw [blue] (-1,0.8) -- (1,0.8);
    \end{tikzpicture} \hspace{1cm}
    \begin{tikzpicture}[baseline=0,scale=0.7]
    \draw [blue] (-1,-1) -- (-0.5,0) -- (-1,1);
    \draw [blue] (-0.5,0) -- (0.5,0);
    \draw [red] (0,1) -- (0.5,0) -- (0,-1);
    \draw [blue] (1,1) -- (0.5,0) -- (1,-1);
    \draw [red] (0.5,0) -- (1,0);
    \end{tikzpicture} \quad $\overset{\text{II}}{\longleftrightarrow}$ \begin{tikzpicture}[baseline=0,scale=0.7]
    \draw [blue] (-1,-1) -- (0,-0.5) -- (1,-1);
    \draw [blue] (-1,1) -- (0,0.5) -- (1,1);
    \draw [red] (0,1) -- (0,0.5) to [out=-150,in=90] (-0.5,0) to [out=-90,in=150] (0,-0.5) -- (0,-1);
    \draw [red] (0,0.5) -- (0.5,0) -- (0,-0.5);
    \draw [red] (0.5,0) -- (1,0);
    \draw [blue] (0,0.5) -- (0,-0.5);
    \end{tikzpicture}\hspace{1cm} 
    \begin{tikzpicture}[baseline=0,scale=0.7]
    \foreach \i in {0,...,3}
    {
    \draw [red] (\i*90:1) -- (\i*90:0.5);
    \draw [red] (\i*90+90:0.5) -- (\i*90:0.5);
    }
    \foreach \i in {-1,1}
    {
    \draw [blue] (-1,0.7*\i) -- (-0.5,0);
    \draw [blue] (1,0.7*\i) -- (0.5,0);
    \draw [teal] (0.7*\i,1) -- (0,0.5);
    \draw [teal] (0.7*\i,-1) -- (0,-0.5);
    }
    \draw [blue] (-0.5,0) -- (0.5,0);
    \draw [teal] (0,0.5) -- (0,-0.5);
    \end{tikzpicture} \quad $\overset{\text{III}}{\longleftrightarrow}$ \quad \begin{tikzpicture}[baseline=0,scale=0.7]
    \foreach \i in {0,...,3}
    {
    \draw [red] (\i*90:1) -- (\i*90:0.5);
    \draw [red] (\i*90+90:0.5) -- (\i*90:0.5);
    }
    \foreach \i in {-1,1}
    {
    \draw [blue] (\i,-0.7) -- (0,-0.5);
    \draw [blue] (\i,0.7) -- (0,0.5);
    \draw [teal] (0.7,\i) -- (0.5,0);
    \draw [teal] (-0.7,\i) -- (-0.5,0);
    }
    \draw [teal] (-0.5,0) -- (0.5,0);
    \draw [blue] (0,0.5) -- (0,-0.5);
    \end{tikzpicture}\\
    \vspace{1cm}
    \begin{tikzpicture}[baseline=0,scale=0.7]
    \draw[blue] (-1,1) to [out=-45,in=180] (0,-0.5) to [out=0,in=-135] (1,1);
    \draw [teal] (-1,-1)  to [out=45,in=180] (0,0.5) to [out=0,in=135] (1,-1);
    \end{tikzpicture} \quad $\overset{\text{IV}}{\longleftrightarrow}$ \quad 
    \begin{tikzpicture} [baseline=0,scale=0.7]
    \draw [blue] (-1,0.7) -- (1,0.7);
    \draw [teal] (-1,-0.7) -- (1,-0.7);
    \end{tikzpicture}\hspace{1cm}
    \begin{tikzpicture}[baseline=0,scale=0.7]
    \draw [blue] (-1,1) -- (-0.2,0) -- (-1,-1);
    \draw [blue] (-0.2,0) -- (1,0);
    \draw [teal] (0,1) to [out=-90,in=90] (0.5,0) to [out=-90,in=90] (0,-1);
    \end{tikzpicture} \quad $\overset{\text{V}}{\longleftrightarrow}$ \quad 
    \begin{tikzpicture}[baseline=0,scale=0.7]
    \draw [blue] (-1,1) -- (0.2,0) -- (-1,-1);
    \draw [blue] (0.2,0) -- (1,0);
    \draw [teal] (0,1) to [out=-90,in=90] (-0.5,0) to [out=-90,in=90] (0,-1);
    \end{tikzpicture}\hspace{1cm}
    \begin{tikzpicture}[baseline=0,scale=1.2]
        \draw [blue] (0,0) -- (1,0);
    \end{tikzpicture} \quad $\overset{\text{VI}}{\longleftrightarrow}$ \quad \begin{tikzpicture}[baseline=0,scale=1.2]
        \draw [blue] (0,0) -- (1,0);
        \node [blue] at (0.5,0) [] {$\bullet$};
    \end{tikzpicture}
    \caption{Weave equivalences.}
    \label{fig: weave equivalences}
\end{figure}
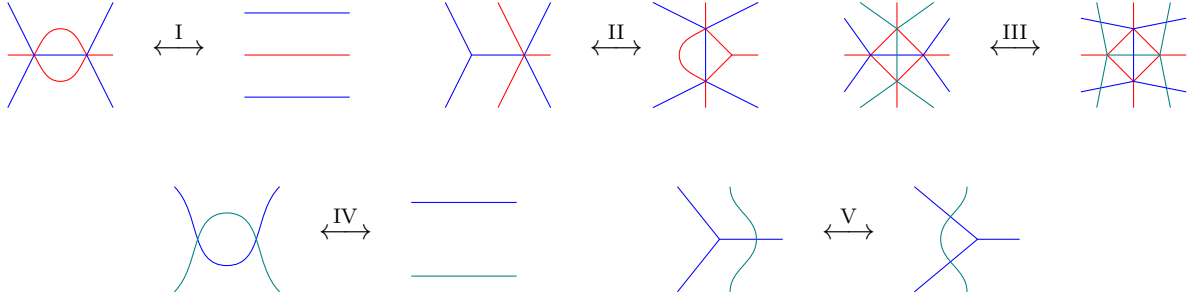

\begin{defn} An unfrozen Y-cycle $\gamma$ is called a \emph{short I-cycle} if there is a unique internal weave edge $e$ such that $\gamma(e)\neq 0$. Given a weave $\ww$ and a short I-cycle $\gamma$ on $\ww$, we can perform a \emph{mutation} at $\gamma$ to produce a new weave $\ww'$. In the new weave $\ww'$, the short I-cycle $\gamma$ is replaced by a new short I-cycle $\gamma'$. Note that two weaves differ by a mutation are not equivalent to each other.
\begin{figure}[H]
    \centering
    \begin{tikzpicture}[baseline=0,scale=0.8]
        \draw [blue] (-1,-1) -- (-0.5,0) -- (-1,1);
        \draw [blue] (1,1) -- (0.5,0) -- (1,-1);
        \draw [blue] (-0.5,0) -- node [above] {$\gamma$} (0.5,0);
        \node at (0,-1.5) [] {$\ww$};
    \end{tikzpicture} \hspace{1cm} $\longleftrightarrow$ \hspace{1cm}
        \begin{tikzpicture}[baseline=0,scale=0.8]
        \draw [blue] (-1,-1) -- (0,-0.5) -- (1,-1);
        \draw [blue] (1,1) -- (0,0.5) -- (-1,1);
        \draw [blue] (0,-0.5) -- node [right] {$\gamma'$} (0,0.5);
        \node at (0,-1.5) [] {$\ww'$};
    \end{tikzpicture}
    \caption{Mutation at a short I-cycle}
    \label{fig:mutation}
\end{figure}
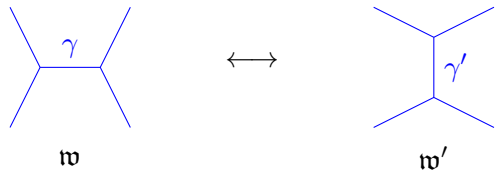
\end{defn}

\subsection{Decorated Flags, Distances, and Flag Moduli Spaces}\label{subsec: 2.2}

Let us recall a few definitions and constructions for decorated flags in this subsection.

Let $G$ be a simply-connected simple linear algebraic group. For simplicity, we assume that $G$ is of simply-laced Dynkin type (ADE) here; the non-simply-laced case will be covered in Section \ref{sec5}. Let us fix a Borel subgroup $B\subset G$ and let $N:=[B,B]$ be its maximal unipotent subgroup. The \emph{flag variety} $\cB$ can be identified with the quotient space $G/B$, whose elements are hence called \emph{flags} or \emph{undecorated flags}.

The \emph{decorated flag variety} is defined to be the quotient space $\cA:=G/N$, whose elements are called \emph{decorated flags}. There is a natural projection map $\pi:\cA\rightarrow \cB$. Picking a decorated flag that maps to a particular flag is called choosing a \emph{decoration}. If we fix a maximal torus $T\subset B$, the space of decorations over a fixed flag is isomorphic to $T$, as a result of the Levi decomposition $B=TN$.

From the fixed maximal torus $T$ we get a Weyl group $W:=N_GT/T$, which happens to be the Coxeter group defined by the Dynkin type (this is how simple linear algebraic groups are classified). The Weyl group $W$ acts on $T$ and hence also on its weight lattice $X^*(T):=\Hom(T,\mathbb{G}_m)$. The induced action of $W$ on the linearization $X^*(T)\underset{\mathbb{Z}}{\otimes}\mathbb{R}$ coincides with the faithful action of $W$ on the Euclidean space $\mathbb{R}^n$. In particular, the choice of Borel subgroup $B\supset T$ gives rise to a basis $\{\omega_i\}$ of $X^*(T)$ called the \emph{fundamental weights}, which is dual to the basis of simple roots $\{\alpha_i\}$ in the sense that $(\alpha_i,\omega_j)=\delta_{ij}$.

A \emph{reduced word} for a Weyl group element $w$ is a product with the fewest Coxeter generators that multiply to $w$. The number of Coxeter generators in a reduced word of $w$ is called the \emph{length} of $w$. Any two reduced words of the same Weyl group element $w$ can be transformed into each other by using only the braid relations.

By fixing a collection of Chevalley generators of $G$, we obtain for each simple root $\alpha_i$ a group homomorphism $\varphi_i:\SL_2\rightarrow G$. For each Coxeter generator $s_i$ of the Weyl group, we define two lifts 
\[
\overline{s}_i=\varphi_i\begin{pmatrix} 0 & -1 \\ 1 & 0 \end{pmatrix}  \quad \text{and} \quad \doverline{s}_i=\varphi_i\begin{pmatrix} 0 & 1 \\ -1 & 0 \end{pmatrix}.
\]
Note that both lifts satisfy the braid relations, and therefore we can use any reduced word of $w$ to define lifts $\overline{w}$ and $\doverline{w}$ for a Weyl group element $w$.

The \emph{Bruhat decomposition} of $G$ states that $G=\bigsqcup_{w\in W} BwB$. It can be further refined to 
\[
G=\bigsqcup_{w\in W} N\overline{w}TN=\bigsqcup_{w\in W}N\doverline{w}TN.
\]

\begin{defn} For a pair of flags $xB$ and $yB$, the \emph{Tits distance} $w(xB,yB)$ is the unique Weyl group element $w$ such that $x^{-1}y\in BwB$.
\end{defn}

Goncharov and Shen introduced an $h$ distance for a pair decorated flags in \cite{GS1}. We will use two versions of their $h$ distances in this article.

\begin{defn}\label{defn: h distance} For a pair of decorated flags $xN$ and $yN$, the \emph{$h_+$ distance} $h_+(xN,yN)$ is defined to be the unique element $h\in T$ such that $x^{-1}y\in N\overline{w}hN$, and the \emph{$h_-$ distance} $h_-(xN,yN)$ is defined to be the unique element $h\in T$ such that $x^{-1}y\in N\doverline{w}hN$.
\end{defn}

\begin{lem}\label{lem: invariant under diagonal G-action} Tits distance between flags and both versions of $h$ distances between decorated flags are invariant under the respective diagonal $G$-actions.
\end{lem}
\begin{proof} It follows from the fact that these distances only depend on the product $x^{-1}y$, which is invariant under the diagonal $G$-actions $g.(xB,yB):=(gxB,gyB)$ and $g.(xN,yN):=(gxN,gyN)$.
\end{proof}

Recall that the \emph{coroot lattice} $X_*(T):=\Hom(\mathbb{G}_m,T)$ is the lattice dual to $X^*(T)$. The \emph{simple coroots} $\{\alpha_i^\vee\}$ is the basis of $X_*(T)$ dual to $\{\omega_i\}$. We will use exponentiation notation $p^{\alpha^\vee}$ to denote the image of $p\in \mathbb{G}_m$ under the an element $\alpha^\vee$ in the coroot lattice; in particular, for a coroot $\alpha_i^\vee$, $p^{\alpha_i^\vee}=\varphi_i\begin{pmatrix} p & 0 \\ 0 & p^{-1}\end{pmatrix}$ for the group homomorphism $\varphi_i:\SL_2\rightarrow G$.

\begin{lem}\label{lem: signs of h-distances} Suppose $x^{-1}y\in Bs_iB$. Then $h_\pm(xN,yN)=(-1)^{\alpha_i^\vee}h_\mp(xN,yN)$.
\end{lem}
\begin{proof} It follows from the fact that $\overline{s}_i=\doverline{s}_i\cdot (-1)^{\alpha_i^\vee}$.
\end{proof}

Recall from \cite{CZ,CW, CGGLSS} that for a weave $\ww$ on a disk, the \emph{flag moduli space} $\cM(\ww)$ is the moduli space of flag configurations on $\ww$, i.e., we associate a flag (i.e., an element of $G/B$) for each face of $\ww$, such that if two faces are separated by a weave edge of color $s_i$, their associated flags are in relative position $s_i$; then we quotient these flag data by the global action of $G$. 

\begin{defn} \label{defn: framed moduli space} With a choice of orientation on each external weave edge of $\ww$, we can further define a \emph{framed flag moduli space} $\cM_\fr(\ww)$: in addition to associating a flag with each face of $\ww$, we also associate a decorated flag (i.e., an element in $G/N$) with each intersection of a face $F$ and a boundary interval $C$, such that the following conditions are satisfied:
\begin{itemize}
    \item if $F\cap C$ is associated with the decorated flag $xN$, then the flag at $F$ is $xB$;
    \item if $F$ and $F'$ are two boundary faces separated by a weave edge of color $s_i$ that is a source (resp. sink) (left (resp. right) picture in Figure \ref{fig: framed moduli space}), and $xN$ and $x'N$ are the decorated flags associated with the boundaries of $F$ and $F'$, respectively, then $x^{-1}x'\in N\overline{s}_iN$ (resp. $x^{-1}x'\in N\doverline{s}_iN$).
\end{itemize}
\begin{figure}[H]
    \centering
    \begin{tikzpicture}[scale=0.7]
        \draw [very thick] (-2,0) -- (2,0);
        \draw [blue, decoration={markings, mark=at position 0.5 with {\arrow{>}}}, postaction={decorate}] (0,0) -- (0,-2) node [below] {$s_i$};
        \node at(1,0) [above] {$x'N$};
        \node at (-1,0) [above] {$xN$};
        \node at (1,-1) [] {$x'B$};
        \node at (-1,-1) [] {$xB$};
    \end{tikzpicture}\hspace{2cm}
    \begin{tikzpicture}[scale=0.7]
        \draw [very thick] (-2,0) -- (2,0);
        \draw [blue, decoration={markings, mark=at position 0.5 with {\arrow{<}}}, postaction={decorate}] (0,0) -- (0,-2) node [below] {$s_i$};
        \node at(1,0) [above] {$x'N$};
        \node at (-1,0) [above] {$xN$};
        \node at (1,-1) [] {$x'B$};
        \node at (-1,-1) [] {$xB$};
    \end{tikzpicture}
    \caption{Decorated flags along the boundary of a weave.}
    \label{fig: framed moduli space}
\end{figure}
In particular, if a boundary face $F$ contains a boundary base point, then the two decorated flags associated with the two adjacent boundary intervals must share the same underlying undecorated flag. Similar to $\cM(\ww)$, we also need to quotient out by the global action of $G$. Note that there is a natural forgetful map 
\[
\cM_\fr(\ww)\longrightarrow \cM(\ww).
\]
\end{defn}

\begin{rmk} In general, $\cM(\ww)$ and $\cM_\fr(\ww)$ are Artin stacks.\end{rmk}

\section{Weighted Cycles and Their Merodromies}\label{sec3}

\subsection{Definition of Weighted Cycles} 
The chains and cycles in this article should be understood as $1$-chains and $1$-cycles; since we will not discuss any chains or cycles of other dimensions, for simplicity, we omit the prefix $1$.

\begin{defn} A \emph{weighted chain} on a weave $\ww$ is an oriented path $\eta$ on the plane that avoids all boundary base points, with no self-intersection, and intersecting $\ww$ transversely, together with the assignment of a weight to each connected component in $\eta\setminus\ww$, such that the two weights adjacent to any edge colored by $s_i$ are of the form $\mu$ and $s_i.\mu$. We also adopt the convention that a weighted chain with a $0$ weight is the same as removing that weighted chain.
\end{defn}

\begin{defn} We also consider weighted chains up to \emph{homotopies}, which is a combination of local path homotopies and the following list of moves (see also Figure \ref{fig: homotopies}):
\begin{enumerate}
    \item[(0)] Pulling or pushing a U-turn through a weave edge.
    \item[(1)] Reversing the orientation and changing the weights to their opposite.
    \item[(2)] Trivial if it does not intersect any weave edges and homotopic to a small loop or a boundary interval.
    \item[(3)] Homotoping through a bivalent weave vertex up to a sign.
    \item[(4)] Homotoping through a trivalent weave vertex of color $s_i$ if $s_i.\mu=\mu$.
    \item[(5)] Homotoping through a tetravalent weave vertex.
    \item[(6)] Homotoping through a hexavalent weave vertex.
\end{enumerate}

\begin{figure}[H]
    \centering
    \begin{tabular}{cccccc}\hline
    \multicolumn{3}{|c|}{
    \begin{tikzpicture}[baseline=0]
    \draw [blue] (0,-1) node [below] {$s_i$} -- (0,1);
    \draw [dashed, ->] (-1,-0.7) node [below] {$\mu$} to [out=0,in=-90] (1,0) node [right] {$s_i.\mu$} to [out=90,in=0] (-1,0.7) node [above] {$\mu$};
    \end{tikzpicture} \quad = \quad \begin{tikzpicture}[baseline = 0]
    \draw [blue] (0,-1) node [below] {$s_i$} -- (0,1);
    \draw [dashed, ->] (-1,-0.7) -- node [left] {$\mu$} (-1,0.7);   
    \node at (1,0) [] {\quad};
    \end{tikzpicture}}&
    \multicolumn{3}{c|}{
    \begin{tikzpicture}[baseline=0]
    \draw[blue] (0,-1) node [below] {$s_i$} -- (0,1);
    \draw [dashed,->] (-1,0) node [above] {$\mu$} -- (1,0) node [above] {$s_i.\mu$};
    \end{tikzpicture} \quad =\quad \begin{tikzpicture}[baseline=0]
    \draw[blue] (0,-1) node [below] {$s_i$} -- (0,1);
    \draw [dashed,<-] (-1,0) node [above] {$-\mu$} -- (1,0) node [above] {$-s_i.\mu$};
    \end{tikzpicture}} \\
    \multicolumn{3}{|c|}{(0)} & \multicolumn{3}{c|}{(1)}  \\[5pt] \hline
    \multicolumn{3}{|c|}{
 \begin{tikzpicture}[baseline=0]
        \draw [dashed,decoration={markings, mark=at position 0.5 with {\arrow{<}}}, postaction={decorate}] (0,0) circle [radius=0.7];
        \node at (0.7,0) [right] {$\mu$};
    \end{tikzpicture}\ $=$ \ \begin{tikzpicture}[baseline=0]
        \draw [very thick] (-1,0.25) -- (1,0.25);
        \draw [dashed,decoration={markings, mark=at position 0.5 with {\arrow{<}}}, postaction={decorate}] (-0.5,0.25) arc (-180:0:0.5);
        \node at (0,-0.25) [below] {$\mu$};
    \end{tikzpicture} \ $=1$} &
    \multicolumn{3}{c|}{
    \begin{tikzpicture}[baseline=0]
    \draw [blue] (0,-1) node [below] {$s_i$} -- (0,1) node [above] {$s_i$};
    \node [blue] at (0,0) [] {$\bullet$};
    \draw [dashed,decoration={markings, mark=at position 0.5 with {\arrow{<}}}, postaction={decorate}] (0,0) circle [radius=0.7];
    \node at (0.7,0) [right] {$\mu \ = (-1)^{\inprod{\alpha_i^\vee}{\mu}}$};
    \node at (-0.7,0) [left] {$s_i.\mu$};
    \end{tikzpicture}
    } \\ 
     \multicolumn{3}{|c|}{(2)} & \multicolumn{3}{c|}{(3)} \\[5pt] \hline
    \multicolumn{2}{|c|}{
    \begin{tikzpicture}[baseline=0] 
    \foreach \i in {0,1,2} 
    {
        \draw [blue] (0,0) -- (-90+\i*120:1);
        \node [blue] at (-90+\i*120:1.2) [] {$s_i$};
    }
    \draw [dashed,decoration={markings, mark=at position 0.5 with {\arrow{<}}}, postaction={decorate}] (0,0) circle [radius=0.7];
    \node at (0.7,0) [right] {$\mu$};
    \end{tikzpicture}\hspace{-1cm} \begin{tikzpicture}[baseline=0]
        \node at (0,0) [] {$=1$};
        \node at (-0.5,-1.7) [] {if $s_i.\mu=\mu$};
    \end{tikzpicture}} 
   & \multicolumn{2}{c|}{
    \begin{tikzpicture}[baseline=0,scale=0.9]
    \draw [blue] (-0.9,-0.9) node [below left] {$s_i$} -- (0.9,0.9) node [above right] {$s_i$};
    \draw [teal] (-0.9,0.9) node [above left] {$s_j$} -- (0.9,-0.9) node [below right] {$s_j$};
    \draw [dashed,decoration={markings, mark=at position 0.5 with {\arrow{<}}}, postaction={decorate}] (0,0) circle [radius=0.7];
    \node at (0.7,0) [right] {$\mu$};
    \node at (0,0.7) [above] {$s_i.\mu$};
    \node at (-0.7,0) [left] {$s_js_i.\mu$};
    \node at (0,-0.7) [below] {$s_j.\mu$};
    \end{tikzpicture}\hspace{-1.5cm}
    \begin{tikzpicture}[baseline=0]
        \node at (0,0) [] {$=1$};
        \node at (-0.5,-1.7) [] {if $s_is_j=s_js_i$};
    \end{tikzpicture}} & 
    \multicolumn{2}{c|}{
    \begin{tikzpicture}[baseline=0]
    \foreach \i in {0,1,2}
    {
        \draw [blue] (0,0) -- (-90+\i*120:1);
        \node [blue] at (-90+\i*120:1.2) [] {$s_i$};
        \draw [red] (0,0) -- (90 +\i*120:1);
        \node [red] at (90+\i*120:1.2) [] {$s_j$};
    }
    \draw [dashed,decoration={markings, mark=at position 0.5 with {\arrow{<}}}, postaction={decorate}] (0,0) circle [radius=0.7];
    \node at (1,0) [] {$\mu$};
    \node at (60:1) [] {$s_i.\mu$};
    \node at (120:1) [] {$s_js_i.\mu$};
    \node at (-1.5,0) [] {$s_is_js_i.\mu$};
    \node at (-60:1) [] {$s_j.\mu$};
    \node at (-120:1) [] {$s_is_j.\mu$};
    \end{tikzpicture}\hspace{-2cm}
    \begin{tikzpicture}[baseline=0]
        \node at (0,0) [] {$=1$};
    \node at (-1,-1.7) [] {if $s_is_js_i=s_js_is_j$};
    \end{tikzpicture}} \\ 
    \multicolumn{2}{|c|}{(4)} & \multicolumn{2}{c|}{(5)} & \multicolumn{2}{c|}{(6)}  \\[5pt] \hline
    \end{tabular}
    \caption{Homotopies in the simply-laced case. }
    \label{fig: homotopies}
\end{figure}
\end{defn}

\begin{rmk} Note that in these homotopy moves, trivial homotopy is set to $1$ instead of $0$; this is because we are treating the weighted cycles multiplicatively instead of additively. This is particularly needed to make sense of the sign in homotopy move (3). Also, please note that there is no ambiguity in the sign factor in homotopy move (3) since $\inprod{\alpha_i^\vee}{s_i.\mu}=\inprod{s_i.\alpha_i^\vee}{\mu}=-\inprod{\alpha_i^\vee}{\mu}$.
\end{rmk}

\begin{defn} A \emph{weighted (relative) cycle} on a weave $\ww$ is a collection of non-intersecting weighted chains with endpoints that are one of the two cases below:
\begin{itemize}
    \item on the boundary of the disk away from the boundary base points;
    \item a generic interior point of the disk such that, after possibly reversing orientations (homotopy move (1)) to make it a sink for all incident weighted chains, the sum of all nearby weights is $0$. We call this the \emph{balancing condition} on the interior endpoint (Figure \ref{fig: interior endpoint}).
\end{itemize}
A weighted cycle is called a \emph{weighted absolute cycle} if the weighted cycle homotopy class has a representative with no endpoints on the boundary of the disk.
\end{defn}

\begin{figure}[H]
    \centering
    \begin{tikzpicture}[baseline=0]
    \foreach \i in {1,2,3,4}
    { \draw [dashed, decoration={markings, mark=at position 0.5 with {\arrow{>}}}, postaction={decorate}] (135-90*\i:1) -- (0,0);
    \node at (135-90*\i:1.2) [] {$\mu_\i$};
    }
    \node at (0,0) [] {$\bullet$};
    \node at (0,-1.5) [] {$\mu_1+\mu_2+\mu_3+\mu_4=0$};
    \end{tikzpicture}
    \caption{An interior point where multiple weighted chains end.}
    \label{fig: interior endpoint}
\end{figure}

\begin{figure}[H]
    \centering
    \begin{tikzpicture}
        \draw [very thick] (0,0) circle [radius=2];
        \node at (2,0) [] {$\bullet$};
        \draw [blue] (45:2) -- (0.75,0) -- (-45:2);
        \draw [blue] (135:2) -- (-0.75,0) -- (-135:2);
        \draw [blue] (-0.75,0) -- (0.75,0);
        \draw [dashed,decoration={markings, mark=at position 0.5 with {\arrow{>}}}, postaction={decorate}] (0,2) -- (0,-1);
        \draw [dashed,decoration={markings, mark=at position 0.5 with {\arrow{>}}}, postaction={decorate}] (0,-2) -- (0,-1);
        \node at (0,-1) [] {$\bullet$};
        \draw [dashed,decoration={markings, mark=at position 0.5 with {\arrow{<}}}, postaction={decorate}] (20:2) to [bend left] (0,-1);
        \node at (0,1) [left] {$\omega_1$};
        \node at (0,-0.5) [left] {$-\omega_1$};
        \node at (0,-1.5) [left] {$2\omega_1$};
        \node at (0.8,-1) [] {$\omega_1$};
        \node at (1.25,0) [above] {$-\omega_1$};
    \end{tikzpicture} \hspace{2cm}
    \begin{tikzpicture}
        \draw [very thick] (0,0) circle [radius=2];
        \node at (2,0) [] {$\bullet$};
        \draw [blue] (45:2) -- (0.75,0) -- (-45:2);
        \draw [blue] (135:2) -- (-0.75,0) -- (-135:2);
        \draw [blue] (-0.75,0) -- (0.75,0);
        \draw [dashed,decoration={markings, mark=at position 0.3 with {\arrow{<}}}, postaction={decorate}] (0.8,-0.25) arc (-90:90:0.25) -- (-0.8,0.25) arc (90:270:0.25) -- (0.8,-0.25);
        \node at (0,0.25) [above] {$\omega_1$};
        \node at (0,-0.25) [below] {$\omega_1$};
        \node at (-1,0) [left] {$-\omega_1$};
        \node at (1,0) [right] {$-\omega_1$};
         \end{tikzpicture}
    \caption{Left: a weighted cycle on an $\SL_2$-weave. Right: a weighted absolute cycle on an $\SL_2$-weave.}
\end{figure}

Recall that the boundary of the disk is decorated with base points, which cut the boundary of the disk into boundary intervals. In addition to path homotopies, we also allow
\begin{enumerate}
    \item[(7)] homotopies that move boundary endpoints of weighted cycles along boundary intervals. 
\end{enumerate} 
We also add the following additional homotopy moves for interior endpoints of weighted chains.
\begin{enumerate}
    \item[(8)] Adding/removing an interior endpoint.
    \item[(9)] Pushing an interior endpoint off the boundary of the disk. 
    \item[(10)] Expand/contract adjacent interior endpoints.
    \item[(11)] Split/combine weighted chains between interior endpoints.
    \item[(12)] Moving an interior endpoint of weighted chains through a weave edge.
    
\end{enumerate}
\begin{figure}[H]
    \centering
    \begin{tabular}{|c|c|}\hline
    \begin{tikzpicture}[baseline=0]
        \draw [very thick] (-1.5,1) -- (1.5,1);
        \draw [blue] (0,-1) -- (0,1) node [above] {$s_i$};
        \draw [dashed, ->] (0.8,1) -- node [left]{$\mu$} (0.8,-1);
    \end{tikzpicture}=\begin{tikzpicture}[baseline=0]
        \draw [very thick] (-1.5,1) -- (1.5,1);
        \draw [blue] (0,-1) -- (0,1) node [above] {$s_i$};
        \draw [dashed, ->] (-0.8,1)  to [out=-90,in=90] (0.8,-1);
        \node at (-0.8,0.1) []{$s_i.\mu$};
        \node at (0.8,-0.1) [] {$\mu$};
    \end{tikzpicture}  &
    \begin{tikzpicture}[baseline=0]
        \draw[dashed,decoration={markings, mark=at position 0.5 with {\arrow{>}}}, postaction={decorate}] (0,-1) -- node [right] {$\mu$} (0,1);
    \end{tikzpicture}\quad $=$ \quad 
    \begin{tikzpicture}[baseline=0]
        \draw[dashed,decoration={markings, mark=at position 0.5 with {\arrow{>}}}, postaction={decorate}] (0,-1) node [above left] {$\mu$} -- (0,0);
        \draw [dashed, decoration={markings, mark=at position 0.5 with {\arrow{>}}}, postaction={decorate}](0,0) -- (0,1) node [below left] {$\mu$} ;
        \node at (0,0) [] {$\bullet$};
    \end{tikzpicture} \\
    (7) & (8) \\[5pt] \hline
    \begin{tikzpicture}[baseline=0,scale=0.8]
        \foreach \i in {0,...,3}
        {
        \draw [dashed, decoration={markings, mark=at position 0.5 with {\arrow{>}}}, postaction={decorate}] (45+90*\i:1.5) -- (0,0);
        }
        \node at (135:1.5) [above left] {$\mu_1$};
        \node at (-135:1.5) [below left] {$\mu_k$};
        \node at (45:1.5) [above left] {$\mu_m$};
        \node at (-45:1.5) [below left] {$\mu_{k+1}$};
        \node at (-1,0) [] {$\vdots$};
        \node at (0.5,0) [] {$\vdots$};
        \node at (0,0) [] {$\bullet$};
        \draw [very thick] (1.1,-1.6) -- (1.1,1.6);
    \end{tikzpicture} \quad  = \quad 
    \begin{tikzpicture}[baseline=0,scale=0.8]
        \draw [very thick] (1,-1.6) -- (1,1.6);
        \draw [dashed, decoration={markings, mark=at position 0.5 with {\arrow{>}}}, postaction={decorate}] (-1.5,0.7) node [below right] {$\mu_1$ }-- (1,0.7);
        \draw [dashed, decoration={markings, mark=at position 0.5 with {\arrow{>}}}, postaction={decorate}] (-1.5,-0.7) node [below right] {$\mu_k$ }-- (1,-0.7);
        \node at (0,0) [] {$\vdots$};
    \end{tikzpicture} &
    \begin{tikzpicture}[baseline=0,scale=0.7]
        \foreach \i in {0,...,3}
        {
        \draw [dashed, decoration={markings, mark=at position 0.5 with {\arrow{>}}}, postaction={decorate}] (45+90*\i:1.5) -- (0,0);
        }
        \node at (135:1.5) [above left] {$\mu_1$};
        \node at (-135:1.5) [below left] {$\mu_k$};
        \node at (45:1.5) [above right] {$\mu_m$};
        \node at (-45:1.5) [below right] {$\mu_{k+1}$};
        \node at (-1,0) [] {$\vdots$};
        \node at (1,0) [] {$\vdots$};
        \node at (0,0) [] {$\bullet$};
    \end{tikzpicture} =\hspace{-2cm}
    \begin{tikzpicture}[baseline=0,scale=0.7]
        \foreach \i in {-1,1}
        {
        \draw [dashed, decoration={markings, mark=at position 0.5 with {\arrow{>}}}, postaction={decorate}] (135*\i:1.5) -- (-0.5,0);
        \draw [dashed, decoration={markings, mark=at position 0.5 with {\arrow{>}}}, postaction={decorate}] (45*\i:1.5) -- (0.5,0);
        \node at (\i*0.5,0) [] {$\bullet$};
        }
        \draw [dashed, decoration={markings, mark=at position 0.5 with {\arrow{>}}}, postaction={decorate}] (-0.5,0) -- node [above] {$\nu$} (0.5,0);
        \node at (135:1.5) [above left] {$\mu_1$};
        \node at (-135:1.5) [below left] {$\mu_k$};
        \node at (45:1.5) [above right] {$\mu_m$};
        \node at (-45:1.5) [below right] {$\mu_{k+1}$};
        \node at (-1,0) [] {$\vdots$};
        \node at (1,0) [] {$\vdots$};
        \node at (-1,-2) [] {$\nu=\mu_1+\cdots +\mu_k-\mu_{k+1}-\cdots-\mu_m$};
    \end{tikzpicture}\\
    (9) &  (10) \\[5pt] \hline
    \begin{tikzpicture}[baseline=0]
    \draw [dashed, decoration={markings, mark=at position 0.5 with {\arrow{>}}}, postaction={decorate}] (0,-1) -- node [left] {$\mu_1+\mu_2$}(0,1);
    \node at (0,-1) [] {$\bullet$};
    \node at (0,1) [] {$\bullet$};
    \end{tikzpicture} \quad =\quad 
    \begin{tikzpicture}[baseline=0]
        \draw [dashed, decoration={markings, mark=at position 0.5 with {\arrow{>}}}, postaction={decorate}] (0,-1) to [bend left] node [left] {$\mu_1$} (0,1);
        \draw [dashed, decoration={markings, mark=at position 0.5 with {\arrow{>}}}, postaction={decorate}] (0,-1) to [bend right] node [right] {$\mu_2$} (0,1);
        \node at (0,-1) [] {$\bullet$};
    \node at (0,1) [] {$\bullet$};
    \end{tikzpicture} &
    \begin{tikzpicture}[baseline=0,scale=0.8]
        \foreach \i in {1,2,3}
        {
        \draw [dashed, decoration={markings, mark=at position 0.5 with {\arrow{>}}}, postaction={decorate}] (-60-120*\i:1.5) -- (0,0);
        \node at (30-15*\i:0.7) [] {$\cdot$};
        }
        \node at (-1.7,0) [left] {$s_i.\mu_1$};
        \node at (-0.5,0) [below] {$\mu_1$};
        \node at (60:1.7) [] {$\mu_2$};
        \node at (-60:1.7) [] {$\mu_k$};
        \draw [blue] (-1,-1.2) node [below] {$s_i$} -- (-1,1.2);
        \node at (0,0) [] {$\bullet$};
    \end{tikzpicture}\quad = \quad \begin{tikzpicture}[baseline=0,scale=0.8]
        \foreach \i in {1,2,3}
        {
        \draw [dashed, decoration={markings, mark=at position 0.5 with {\arrow{>}}}, postaction={decorate}] (-60-120*\i:1.5) -- (-0.7,0);
        \node at (30-15*\i:0.7) [] {$\cdot$};
        }
        \node at (-1.7,0) [left] {$s_i.\mu_1$};
        \node at (60:1.7) [] {$\mu_2$};
        \node at (-60:1.7) [] {$\mu_k$};
        \draw [blue] (0.2,-1.2) node [below] {$s_i$} -- (0.2,1.2);
        \node at (0,0.7) [above left] {$s_i.\mu_2$};
        \node at (0,-0.7) [below left] {$s_i.\mu_k$};
        \node at (-0.7,0) [] {$\bullet$};
    \end{tikzpicture} \\ 
    (11)  & (12) \\[5pt] \hline\end{tabular}
    \caption{Homotopies involving interior endpoints.}
    \label{fig: interior endpoint through a weave edge}
\end{figure}

\begin{rmk} There is a special case of homotopy move (10), where the intermediate weight $\nu=0$; in this case, we delete the weighted chain in the middle and the move will effectively split the interior endpoint in two. 
\end{rmk}

\subsection{Intersection Pairing between Weighted Cycles} In this subsection we define an intersection number between weighted cycles. 

\begin{defn} When two weighted chains $\eta_1$ and $\eta_2$ intersect at $p$, we define the \emph{sign} of the intersection $p$, $\sgn_p(\eta_1,\eta_2)$, as in Figure \ref{fig: weighted chain intersection sign}. If $\eta_1$ and $\eta_2$ do not intersect, we define $\sgn_p(\eta_1,\eta_2)=0$.
    \begin{figure}[H]
    \centering
    \begin{tikzpicture}[baseline=0]
    \draw[dashed,<-] (0,-1) node [below left] {$\eta_2$} -- (0,1);
    \draw [dashed,->] (-1,0) node [below left] {$\eta_1$} -- (1,0);
    \node at (0,-2) [] {$\sgn_p(\eta_1,\eta_2)=1$};
    \end{tikzpicture} \hspace{2cm}     \begin{tikzpicture}[baseline=0]
    \draw[dashed,->] (0,-1) node [below left] {$\eta_2$}-- (0,1);
    \draw [dashed,->] (-1,0) node [below left] {$\eta_1$} -- (1,0);
    \node at (0,-2) [] {$\sgn_p(\eta_1,\eta_2)=-1$};
    \end{tikzpicture} 
    \caption{Signs of the intersection between two weighted chains.}
    \label{fig: weighted chain intersection sign}
\end{figure}
Moreover, if $\eta_1$ and $\eta_2$ have weights $\mu_1$ and $\mu_2$ at $p$, we define the \emph{local intersection number} between $\eta_1$ and $\eta_2$ at $p$ to be
\[
\{\eta_1,\eta_2\}_p:=\sgn_p(\eta_1,\eta_2)\cdot (\mu_1,\mu_2),
\]
where $(\mu_1,\mu_2)$ is the inner product between weights.
\end{defn}

\begin{defn}\label{defn: pairing along boundary interval} For a boundary endpoint $p$ of a weighted chain, we define $\sgn(p)=-1$ if it is a source and $\sgn(p)=1$ if it is a sink. 
For a boundary interval $C$ and two boundary endpoints $p_1$ and $p_2$ of weighted chains, we define $\sgn_C(p_1,p_2)=1$ if $p_1$ precedes $p_2$ in the clockwise direction, $\sgn_C(p_1,p_2)=-1$ if $p_1$ precedes $p_2$ in the counterclockwise direction, and $\sgn_C(p_1,p_2)=0$ if at least one of $p_1$ and $p_2$ is not on $C$. Lastly, suppose $\mu_1$ and $\mu_2$ are weights at the boundary endpoints $p_1$ and $p_2$ on the same boundary interval $C$, separated by external weave edges of colors $s_{i_1}, s_{i_2},\dots, s_{i_k}$ in the clockwise direction (as in Figure \ref{fig: boundary intersection pairing}); we define the inner product between $p_1$ and $p_2$ to be
\[
I_C(\mu_1,\mu_2):=\left\{\begin{array}{ll}
    (\mu_1, s_{i_1}\cdots s_{i_k}.\mu_2) & \text{if $\mu_2$ precedes $\mu_1$ in the clockwise direction,} \\
    (\mu_2,s_{i_1}\cdots s_{i_k}.\mu_1) & \text{if $\mu_1$ precedes $\mu_2$ in the clockwise direction.} 
\end{array}\right.
\]
\begin{figure}[H]
    \centering
    \begin{tikzpicture}
        \draw [very thick] (-2,0) --(2,0) node [right] {$C$};
        \draw [dashed] (-1.5,0) node [above] {$p_1$} -- node [left] {$\mu_1$}(-1.5,-1.5) ;
        \draw [dashed] (1.5,0) node [above] {$p_2$} -- node [right] {$\mu_2$}(1.5,-1.5) ;
        \draw [blue] (-0.5,0) -- (-0.5,-1.5) node [below] {$s_{i_1}$};
        \draw [red] (0.5,0) -- (0.5,-1.5) node [below] {$s_{i_k}$};
        \node at (0,-0.8)[ ]{$\cdots$};
    \end{tikzpicture}
    \caption{Inner product between boundary endpoints.}
    \label{fig: boundary intersection pairing}
\end{figure}
We define the \emph{local intersection number} between weighted chains $\eta_1$ and $\eta_2$ at a boundary interval $C$ to be
\[
\{\eta_1,\eta_2\}_C:=\frac{1}{2}\sum_{p_1\in \partial \eta_1}\sum_{p_2\in \partial \eta_2} \sgn(p_1)\sgn(p_2)\sgn_C(p_1,p_2)I_C(\mu_1,\mu_2).
\]
\end{defn}

\begin{defn}\label{defn: total intersection number between weighted cycles} We define the \emph{(total) intersection number} between two weighted chains $\eta_1$ and $\eta_2$ to be
\[
\{\eta_1,\eta_2\}:=\sum_{p\in \eta_1\cap \eta_2}\{\eta_1,\eta_2\}_p+\sum_C \{\eta_1,\eta_2\}_C.
\]
It is not hard to see that $\{\cdot, \cdot\}$ is skew-symmetric. We extend this intersection number linearly to an intersection number between weighted cycles.
\end{defn}

\begin{prop} The intersection number between weighted cycles is invariant under homotopies.
\end{prop}
\begin{proof} Since the intersection number is defined locally and the intersections are assumed to be in generic position, we can always isolate the intersections between weighted cycles away from most of the homotopy moves on the list. Thus, it suffices to only consider the following. 
\begin{enumerate}[(i)]
    \item Path homotopies that create/cancel pairs of intersections between weighted cycles.
    \item Homotopy move (7) that moves boundary endpoints of weighted chains past each other.
    \item Homotopy move (7) that moves a boundary endpoint of a weighted cycle through an external weave edge.
    \item Homotopy move (1), which reverses the orientation of a weighted cycle and changes the weights to their opposite.
    \item Homotopy move (9), which pushes an interior endpoint off the boundary of the disk.
    \item Pushing/pulling an intersection between weighted cycles through a weave edge.
    \item Pushing/pulling intersections between weighted cycles through an interior endpoint.
\end{enumerate}  
\begin{figure}[H]
    \centering
    \begin{tabular}{ccccccc}\hline
    \multicolumn{2}{|c|}{
    \begin{tikzpicture}[baseline=0,scale=0.7]
        \draw [dashed] (-0.5,-1) node [below] {$\mu_1$} -- (-0.5,1);
        \draw [dashed] (0.5,-1) node [below] {$\mu_2$} -- (0.5,1);
    \end{tikzpicture} $\longleftrightarrow$ 
    \begin{tikzpicture}[baseline=0,scale=0.7]
        \draw [dashed] (-0.5,-1) node [below] {$\mu_1$} to [out=45, in=-90] (0.2,0) to [out=90,in=-45] (-0.5,1);
        \draw [dashed] (0.5,-1) node [below] {$\mu_2$} to [out=135,in=-90] (-0.2,0) to [out=90,in=-135] (0.5,1);
    \end{tikzpicture}}&
    \multicolumn{3}{c|}{
    \begin{tikzpicture}[baseline=0,scale=0.7]
    \draw [very thick] (-0.8,1) -- (0.8,1);
        \draw [dashed, decoration={markings, mark=at position 0.5 with {\arrow{<}}}, postaction={decorate}] (-0.5,-1) node [below] {$\mu_1$} -- (-0.5,1) node [above] {$p_1$};
        \draw [dashed, decoration={markings, mark=at position 0.5 with {\arrow{<}}}, postaction={decorate}] (0.5,-1) node [below] {$\mu_2$} -- (0.5,1) node [above] {$p_2$};
    \end{tikzpicture}$\longleftrightarrow$ 
    \begin{tikzpicture}[baseline=0,scale=0.7]
    \draw [very thick] (-0.8,1) -- (0.8,1);
        \draw [dashed, decoration={markings, mark=at position 0.7 with {\arrow{<}}}, postaction={decorate}] (-0.5,-1) node [below] {$\mu_1$} to [out=90, in=-90] (0.5,1) node [above] {$p_1$};
        \node at (0,0) [below] {$p$};
        \draw [dashed, decoration={markings, mark=at position 0.7 with {\arrow{<}}}, postaction={decorate}] (0.5,-1) node [below] {$\mu_2$} to [out=90,in=-90] (-0.5,1) node [above] {$p_2$};
    \end{tikzpicture}}&
    \multicolumn{2}{c|}{
    \begin{tikzpicture}[baseline=0,scale=0.7]
        \draw [very thick] (-1,1) -- (2,1);
        \draw [dashed] (-0.5,1) node [above] {$p_1$} -- node [left] {$\mu_1$}(-0.5,-1);
        \draw [blue] (0,1) -- (0,-1) node [below] {$s_{i_1}$};
        \draw [red] (1,1) -- (1,-1) node [below] {$s_{i_k}$};
        \draw [dashed] (1.5,1) node [above] {$p_2$} -- node [right] {$\mu_2$} (1.5,-1);
        \node at (0.5,0) [] {$\cdots$};
    \end{tikzpicture}\quad $\longleftrightarrow$\quad
    \begin{tikzpicture}[baseline=0,scale=0.7]
        \draw [very thick] (-1,1) -- (2,1);
        \draw [dashed] (0.5,1) node [above] {$p_1$} to [out=-90, in=90] (-0.5,0) node [left] {$\mu_1$} -- (-0.5,-1);
        \draw [blue] (0,1) -- (0,-1) node [below] {$s_{i_1}$};
        \draw [red] (1,1) -- (1,-1) node [below] {$s_{i_k}$};
        \draw [dashed] (1.5,1) node [above] {$p_2$} -- node [right] {$\mu_2$} (1.5,-1);
        \node at (0.5,0) [] {$\cdots$};
    \end{tikzpicture}} \\
    \multicolumn{2}{|c|}{(i)} & 
    \multicolumn{3}{c|}{(ii)} & 
    \multicolumn{2}{c|}{(iii)} \\[5pt]\hline
    \multicolumn{4}{|c|}{
    \begin{tikzpicture}[baseline=0,scale=0.8]
    \draw[dashed,<-] (0,-1) node [below] {$\mu_1$} -- (0,1);
    \draw [dashed,->] (-1,0) node [above] {$\mu_2$} -- (1,0);
    \end{tikzpicture} \quad $\longleftrightarrow$   \quad  \begin{tikzpicture}[baseline=0, scale=0.8]
    \draw[dashed,->] (0,-1) node [below] {$-\mu_1$}-- (0,1);
    \draw [dashed,->] (-1,0) node [above] {$\mu_2$} -- (1,0);
    \end{tikzpicture}} &
    \multicolumn{3}{c|}{
    \begin{tikzpicture}[baseline=0, scale=0.8]
    \draw [very thick] (-1,1) -- (3,1);
    \foreach \i in {0,...,3}
    {
        \draw [dashed, decoration={markings, mark=at position 0.5 with {\arrow{>}}}, postaction={decorate}] (45+90*\i:1.3) -- (0,0);
    }
    \node at (0,0) [] {$\bullet$};
    \node at (0,-0.5) [] {$\cdots$};
    \node at (0,0.5) [] {$\cdots$};
    \node at (-1,1) [above] {$p_1$};
    \node at (0,1) [above] {$\cdots$};
    \node at (1,1) [above] {$p_k$};
    \node at (-1,0.5) [] {$\mu_1$};
    \node at (1,0.5) [] {$\mu_k$};
    \node at (-1,-0.5) [] {$\mu_{k+1}$};
    \node at (1,-0.5) [] {$\mu_m$};
    \draw [dashed,->] (2,1) node [above] {$q$} -- node [right] {$\nu$} (2,-1);
    \end{tikzpicture}\quad $\longleftrightarrow$\quad\begin{tikzpicture}[baseline=0, scale=0.8]
    \draw [very thick] (-1,1) -- (3,1);
    \foreach \i in {-1,1}
    {
        \draw [dashed, decoration={markings, mark=at position 0.5 with {\arrow{>}}}, postaction={decorate}] (\i*0.8,-1) -- (\i*0.8,1);
    }
    \node at (0,0) [] {$\cdots$};
    \node at (-1,1) [above] {$p_{k+1}$};
    \node at (0,1) [above] {$\cdots$};
    \node at (1,1) [above] {$p_m$};
    \node at (-1.2,-0.5) [] {$\mu_{k+1}$};
    \node at (1.2,-0.5) [] {$\mu_m$};
    \draw [dashed,->] (2,1) node [above] {$q$} -- node [right] {$\nu$} (2,-1);
    \end{tikzpicture} } \\
    \multicolumn{4}{|c|}{(iv)} & 
    \multicolumn{3}{c|}{(v)} \\[5pt] \hline
    \multicolumn{4}{|c|}{
    \begin{tikzpicture}[baseline=0,scale=0.8]
        \draw[dashed] (-1,-1) node [below] {$\mu_1$} -- (1,1) node [above] {$s_i.\mu_1$};
        \draw [dashed] (-1,1) node [above] {$\mu_2$} -- (1,-1)node [below] {$s_i.\mu_2$};
        \draw [blue] (-0.5,-1.5) node [below] {$s_i$} -- (-0.5,1.5);
    \end{tikzpicture}$\longleftrightarrow$  \begin{tikzpicture}[baseline=0,scale=0.8]
        \draw[dashed] (-1,-1) node [below] {$\mu_1$} -- (1,1) node [above] {$s_i.\mu_1$};
        \draw [dashed] (-1,1) node [above] {$\mu_2$} -- (1,-1)node [below] {$s_i.\mu_2$};
        \draw [blue] (0.5,-1.5) node [below] {$s_i$} -- (0.5,1.5);
    \end{tikzpicture}}&
    \multicolumn{3}{c|}{
    \begin{tikzpicture}[baseline=0,scale=0.7]
        \foreach \i in {0,...,3}
        {
        \draw [dashed, decoration={markings, mark=at position 0.5 with {\arrow{>}}}, postaction={decorate}] (45+90*\i:1.5) -- (0,0);
        }
        \node at (135:1.5) [above left] {$\mu_1$};
        \node at (-135:1.5) [below left] {$\mu_k$};
        \node at (45:1.5) [above right] {$\mu_m$};
        \node at (-45:1.5) [below right] {$\mu_{k+1}$};
        \draw [dashed,->] (-0.4,-1.5) -- (-0.4,1.5) node [above] {$\nu$};
        \node at (-1,0) [] {$\vdots$};
        \node at (1,0) [] {$\vdots$};
        \node at (0,0) [] {$\bullet$};
\end{tikzpicture}$\longleftrightarrow$\begin{tikzpicture}[baseline=0,scale=0.7]
        \foreach \i in {0,...,3}
        {
        \draw [dashed, decoration={markings, mark=at position 0.5 with {\arrow{>}}}, postaction={decorate}] (45+90*\i:1.5) -- (0,0);
        }
        \node at (135:1.5) [above left] {$\mu_1$};
        \node at (-135:1.5) [below left] {$\mu_k$};
        \node at (45:1.5) [above right] {$\mu_m$};
        \node at (-45:1.5) [below right] {$\mu_{k+1}$};
        \draw [dashed,->] (0.4,-1.5) -- (0.4,1.5) node [above] {$\nu$};
        \node at (-1,0) [] {$\vdots$};
        \node at (1,0) [] {$\vdots$};
        \node at (0,0) [] {$\bullet$};
\end{tikzpicture}} \\
\multicolumn{4}{|c|}{(vi)} & \multicolumn{3}{c|}{(vii)} \\[5pt]\hline
\end{tabular}
    \caption{Invariance of intersection numbers under homotopies of weighted cycles.}
    \label{fig: invariant of intersection numbers between weighted cycles}
\end{figure}

For (i), regardless of the orientations on the weighted chains, the new pair of intersection points have opposite signs. Thus, their combined contribution is $(\mu_1,\mu_2)-(\mu_1,\mu_2)=0$.

For (ii), the LHS has a local contribution $\{p_1,p_2\}_C=-\frac{1}{2}(\mu_1,\mu_2)$, and the RHS has a local contribution $\{p_1,p_2\}_C+\{\eta_1,\eta_2\}_p=\frac{1}{2}(\mu_1,\mu_2)-(\mu_1,\mu_2)$, which is equal to the LHS.

For (iii), the inner product $I_C(p_1,p_2)$ changes from $(\mu_1,s_{i_1}\cdots s_{i_k}.\mu_2)$ to $(s_{i_1}.\mu_1,s_{i_2}\cdots s_{i_k}.\mu_2)$, which is equal to $(\mu_1,s_{i_1}\cdots s_{i_k}.\mu_2)$ by the invariance of inner product between weights under the Weyl group action. 

For (iv), the sign of the intersection flips, which cancels with the new minus sign on the weight, so the local contribution does not change.

For (v), the combined contribution $\sum_{i=1}^k\{q,p_i\}$ is equal to $\sum_{i=1}^k(\mu_i,\nu)=\left(\sum_{i=1}^k\mu_i,\nu\right)$, and the combined contribution $\sum_{i=k+1}^m\{q,p_i\}$ is equal to $-\sum_{i=k+1}^m(\mu_i,\nu)=\left(-\sum_{i=k+1}^m\mu_i,\nu\right)=\left(\sum_{i=1}^k\mu_i,\nu\right)$, so the combined contribution does not change.

For (vi), regardless of the orientations on the weighted chains, the sign of the intersection does not change, and the pairing between weights becomes $(s_i.\mu_1,s_i.\mu_2)$, which is equal to $(\mu_1,\mu_2)$ since the inner product between weights is invariant under the Weyl group action.

For (vii), since $\mu_{k+1}+\cdots +\mu_m=-\mu_1-\cdots-\mu_k$, we have $(\mu_1,\nu)+\cdots+(\mu_k,\nu)=-(\mu_{k+1},\nu)-\cdots -(\mu_m,\nu)$; but then the signs of intersections between $\mu_{k+1}, \dots, \mu_m$ and $\nu$ are opposite to those between $\mu_1, \dots, \mu_k$ and $\nu$. Thus the local contribution does not change either.
\end{proof}

\subsection{Weighted Cycle Algebra and Quantization}

\begin{defn} Given a weave $\ww$, we define the \emph{weighted (relative) cycle algebra} $\mathcal{W}(\ww)$ to be the commutative algebra over $\mathbb{C}$ whose elements are formal linear combinations of homotopy classes of weighted relative cycles on $\ww$, with a multiplication defined by first superimposing weighted cycles on top of each other generically and then replacing each crossing with an interior endpoint of weighted chains.
\begin{figure}[H]
    \centering
    \begin{tikzpicture}[baseline=0]
        \draw [dashed,->] (-1,-1) -- (1,1);
        \draw [dashed] (1,-1) -- (0.3,-0.3);
        \draw [dashed, ->] (-0.3,0.3) -- (-1,1);
        \node at (0.5,0.5) [below right] {$\mu_1$};
        \node at (-0.5,0.5) [below left] {$\mu_2$};
    \end{tikzpicture} \quad $=$ \quad 
    \begin{tikzpicture}[baseline=0]
        \draw [dashed,->] (-1,-1) -- (1,1) ;
        \draw [dashed, ->] (1,-1) -- (-1,1);
        \node at (0.5,0.5) [below right] {$\mu_1$};
        \node at (-0.5,0.5) [below left] {$\mu_2$};
        \node at (0,0) [] {$\bullet$};
    \end{tikzpicture}
    \caption{Multiplication in weighted cycle algebras.}
\end{figure}
\end{defn}

Note that the empty weighted chain (e.g., the $0$'s in Figure \ref{fig: homotopies}) is the multiplicative identity in $\mathcal{W}(\ww)$.

\begin{defn} Recall that a weighted cycle is absolute if there is a representative whose weighted chains do not end at the boundary of $\ww$. Since the product of two weighted absolute chains is still absolute, weighted absolute chains form a subalgebra inside $\cW(\ww)$, which we call the \emph{weighted absolute cycle algebra} and denote by $\cW_\abs(\ww)$
\end{defn}

\begin{thm}\label{thm: weighted cycle algebra is a torus} Suppose $\ww$ is a weave for a simply connected semisimple Lie group $G$ of rank $r$. Let $\beta$ be the number of boundary base points on the disk and let $\tau$ be the number of trivalent weave vertices in $\ww$. The weighted cycle algebra $\mathcal{W}(\ww)$ is a Laurent polynomial ring whose spectrum is an algebraic torus of dimension $\tau+r(\beta-1)$.
\end{thm}
\begin{proof} It suffices to show that the weighted cycles form a $\mathbb{Z}$-lattice of rank $\tau+r(\beta-1)$. First, we claim that each time we add a boundary base point, we add a summand of $\mathbb{Z}^r$. To see this, let $\eta$ be a small oriented path surrounding only the new base point. For each fundamental weight $\omega_i$ with $1\leq i\leq r$ we define a weighted cycle $\eta_i$ to be $\eta$ together with the weight $\omega_i$. These are the basis for the new $\mathbb{Z}^r$ summand.

The statement is now reduced to showing that if there is only one boundary base point, the rank of the lattice of weighted cycles is $\tau$. We observe that we can remove all interior endpoints for weighted chains inside a weighted cycle by using homotopy moves. To see this, note that we can use homotopy move (1) and (8) to remove all bivalent interior endpoints of weighted chains, and then use homotopy move (10) to make all remaining interior endpoints trivalent. We then do an induction on the number of interior endpoints. There is nothing to show for the base case with no interior endpoints. Inductively, pick a trivalent interior endpoint and draw a curve $C$ that goes from this interior endpoint to the boundary of the disk, intersecting only the interiors of the weighted chains inside the weighted cycle and the interiors of weave edges. By using homotopy move (8), we put a new bivalent interior endpoint at each intersection between $C$ and the weighted cycle. By using homotopy moves (9) and (12), we can push these bivalent interior endpoints off the boundary of the disk one-by-one, which creates a path to push the chosen trivalent interior endpoint off the boundary of the disk as well.

\begin{figure}[H]
\centering
\begin{tikzpicture}[baseline=-20,scale=0.8]
\foreach \i in {0,1,2}
{
    \draw [dashed, decoration={markings, mark=at position 0.5 with {\arrow{>}}}, postaction={decorate}] (90*\i:1.5) -- (0,0);
}
\node at (0,0) [] {$\bullet$};
\foreach \i in {1,2}
{
\draw [dashed, decoration={markings, mark=at position 0.6 with {\arrow{>}}}, postaction={decorate}] (-1.5,-0.75*\i) -- (1.5,-0.75*\i);
}
\draw [very thick] (-1.5,-2.25) -- (1.5,-2.25);
\draw [red,decoration={markings, mark=at position 0.5 with {\arrow{>}}}, postaction={decorate}] (0,0) -- node [right] {$C$} (0,-2.25);
\end{tikzpicture} \quad \quad $\longrightarrow$ \quad \quad
\begin{tikzpicture}[baseline=-20,scale=0.8]
    \draw [dashed, decoration={markings, mark=at position 0.5 with {\arrow{>}}}, postaction={decorate}] (-2.5,0) to [out=0,in=120] (0,-2);
    \draw [dashed, decoration={markings, mark=at position 0.5 with {\arrow{>}}}, postaction={decorate}] (2.5,0) to [out=180,in=60] (0,-2);
    \draw [dashed, decoration={markings, mark=at position 0.5 with {\arrow{>}}}, postaction={decorate}] (0,1.5) -- (0,-2);
\node at (0,-2) [] {$\bullet$};
\foreach \i in {1,2}
{
\draw [dashed, decoration={markings, mark=at position 0.6 with {\arrow{>}}}, postaction={decorate}] (-2.5,-0.75*\i) arc (90:0:0.75*3-0.75*\i);
\draw [dashed, decoration={markings, mark=at position 0.6 with {\arrow{<}}}, postaction={decorate}] (2.5,-0.75*\i) arc (90:180:0.75*3-0.75*\i);
}
\draw [very thick] (-2.5,-2.25) -- (2.5,-2.25);
\end{tikzpicture}
\caption{Removing a trivalent interior endpoint.}
\label{fig: reduction of trivalent endpoints}
\end{figure}

With the assumption that there are no interior endpoints for weighted chains, each weighted chain in a weighted cycle must be an arc that goes from some boundary interval to another boundary interval (may possibly be the same one), or forms a closed circle. By adding a generic bivalent interior endpoint and using a move similar to Figure \ref{fig: reduction of trivalent endpoints}, we can open up any closed circle into an arc as well. Thus, without loss of generality, we may assume that all weighted chains are pairwise non-intersecting arcs. Let us do an induction on the number of trivalent weave vertices $\tau$. For the base case where $\tau=0$, all weighted chains can be homotoped to the unique boundary interval using homotopy moves (3), (5), and (6), and they can be further shown to be trivial by moving the boundary endpoints together via homotopy move (7) and then applying the homotopy move (2).

Now inductively, since all weighted chains inside the weighted cycle are pairwise non-intersecting arcs, every connected component in the complement of the weighted cycle must contain some boundary interval of the disk. Let us fix one trivalent weave vertex $v$ of color $s_i$ inside one of these connected components. We choose a simple curve $\gamma$ inside this connected component, going from one point on the boundary to another point on the boundary, intersecting the weave at generic positions and isolating the weave vertex $v$ from the rest of the weave (see Figure \ref{fig: isolating a trivalent vertex}). We may choose $\gamma$ so that it also intersects a weave edge incident to $v$ at a point $p$. Note that $\gamma$ cuts the disk into two disks, with all weighted chains contained inside the disk without the trivalent weave vertex $v$. On the one hand, the smaller disk without the weave vertex $v$ has one fewer trivalent weave vertex, and hence by induction, the lattice of weighted cycles on this smaller disk is of rank $\tau-1$. On the other hand, by using homotopy move (4), most weighted chain with a support that is homotopic to $\gamma$ can be further homotoped to the boundary of the original disk through the other smaller disk containing $v$. For an obstructed weighted chain, let $\omega$ be its weight near the point $p$. Since $\{\omega_j\}$ is a basis of the weight lattice, we can write $\omega=\sum_jc_j\omega_j$. By using homotopy move (11), we can split this weighted chain into multiple weighted chains with the same support and with weights $c_j\omega_j$ near $p$. Because $s_i.\omega_j=\omega_j$ for all $j\neq i$, all these newly split weighted chains can be homotoped through the smaller disk except the one with the weight $c_i\omega_i$ and $c_i\neq 0$. This adds $1$ to the rank of the lattice of weighted cycles and hence the total rank is $\tau$. The induction is now complete.
\end{proof}

\begin{figure}[H]
    \centering
    \begin{tikzpicture}
        \draw (0,0)  circle [radius=1.5];
        \draw (45:1.5) arc (135:225:1.5) node [below right] {$\gamma$};
        \foreach \i in {0,1,2}
        {
        \draw [blue] (1,0) -- ++(60+120*\i:0.4);
        }
        \node [blue] at (1,0) [right] {$v$};
        \node at (0.6,0) [left] {$p$};
    \end{tikzpicture}
    \caption{Isolating a single trivalent weave vertex.}
    \label{fig: isolating a trivalent vertex}
\end{figure}

Since the weighted cycles are equipped with an intersection pairing, we can further quantize the weighted cycle algebra $\mathcal{W}(\ww)$ by imposing the following relation on weighted cycles:
\[
\eta_1\eta_2=q^{\{\eta_1,\eta_2\}}(\eta_1\#\eta_2),
\]
where $(\eta_1\#\eta_2)$ denotes the weighted cycle that is the classical product between $\eta_1$ and $\eta_2$. We denote the quantum cycle algebra by $\mathbb{W}(\ww)$.

It is not hard to see that we can recover $\mathcal{W}(\ww)$ from $\mathbb{W}(\ww)$ by setting $q=1$.

\subsection{Weighted Cycle Representatives for Y-Cycles}\label{subsec: 3.4}

We can construct weighted cycle representatives of Y-cycles as follows.
\begin{itemize}
    \item Let $\gamma$ be a Y-cycle. For each weave edge $e$ in $\gamma$, we draw a relative chain on each side of $e$ and orient them such that $e$ is on the right side of the relative chain.
    \item If $e$ has color $s_i$, we assign the weight $\gamma(e)\omega_i$ to each of the two relative chains.
    \item Near a bivalent weave vertex, we connect the nearby weighted chains as in the top left picture of Figure \ref{fig: Y-cycle representatives}.
    \item Near a trivalent weave vertex, we connect the nearby weighted chains as in the top right picture of Figure \ref{fig: Y-cycle representatives}, where 
    \[
    \mu_j:=|a_{j-1}-a_{j+1}|\omega_i-[a_{j-1}-a_{j+1}]_+\alpha_i
    \]
    with the indices taken modulo $3$. (Recall that $[n]_+:=\max\{n,0\}$, and $\alpha_i$ is the $i$th simple root.)
    \item Near a tetravalent weave vertex, we connect the nearby weighted chains as in the bottom left picture of Figure \ref{fig: Y-cycle representatives}.
    \item Near a hexavalent weave vertex, we connect the nearby weighted chains as in the bottom right picture of Figure \ref{fig: Y-cycle representatives}. The weight $\mu_i$ is determined by the other two weights via the balancing condition. Note that the interior endpoint in the center should be perturbed slightly to meet the generic position requirement.
    \begin{figure}[H]
        \centering
        \begin{tikzpicture}[scale=1.2,baseline=0]
        \draw [blue] (0,-1.5) node [below] {$s_i$}-- (0,1.5) node [above] {$s_i$};
        \node [blue] at (0,0) [] {$\bullet$};
            \draw [dashed,decoration={markings, mark=at position 0.5 with {\arrow{>}}}, postaction={decorate}] (-0.5,0) --node [left] {$a\omega_i$} (-0.5,1.5);
            \draw [dashed,decoration={markings, mark=at position 0.5 with {\arrow{>}}}, postaction={decorate}] (-0.5,-1.5) --node [left] {$a\omega_i$} (-0.5,0);
            \node at (-0.5,0) [] {$\bullet$};
            \draw [dashed,decoration={markings, mark=at position 0.5 with {\arrow{>}}}, postaction={decorate}] (0.5,1.5) --node [right] {$a\omega_i$} (0.5,0);
            \draw [dashed,decoration={markings, mark=at position 0.5 with {\arrow{>}}}, postaction={decorate}] (0.5,0) --node [right] {$a\omega_i$} (0.5,-1.5);
            \node at (0.5,0) [] {$\bullet$};
        \end{tikzpicture}
        \hspace{2cm}
        \begin{tikzpicture}[scale=1.5, baseline=0]
            \foreach \i in {1,2,3}
            {
            \draw [blue] (\i*120:1.5) -- (0,0);
            \draw [dashed,decoration={markings, mark=at position 0.5 with {\arrow{>}}}, postaction={decorate}] (-\i*120:1.5)++(-\i*120-90:0.45) -- (-\i*120-60:0.5);
            \draw [dashed,decoration={markings, mark=at position 0.5 with {\arrow{<}}}, postaction={decorate}] (-\i*120:1.5)++(-\i*120+90:0.45) -- (-\i*120+60:0.5);
            \draw [dashed,decoration={markings, mark=at position 0.7 with {\arrow{>}}}, postaction={decorate}] (\i*120+60:0.5) -- (\i*120-60:0.5);
            \node at (120*\i+60:0.5) [] {$\bullet$};
            \node at (120*\i+90-120:1.3) [] {$a_\i\omega_i$};
            \node at (120*\i+150-120:1.3) [] {$a_\i\omega_i$};
            \node at (120*\i+145-120:0.5) [] {$\mu_\i$};
            \node at (120*\i+100-120:0.6) [] {$s_i.\mu_\i$};
            \node [blue] at (5+120*\i:1.5) [] {$s_i$};
            }
        \end{tikzpicture}\\
        \begin{tikzpicture}[baseline=0,scale=2]
        \draw [blue] (-1,1) node [above left] {$s_i$}-- (1,-1);
        \draw [teal] (-1,-1) node [below left] {$s_j$} -- (1,1);
        \draw [dashed,decoration={markings, mark=at position 0.5 with {\arrow{>}}}, postaction={decorate}] (-1,-0.7) node[above] {$b\omega_j$} -- (-0.3,0);
        \draw [dashed,decoration={markings, mark=at position 0.5 with {\arrow{>}}}, postaction={decorate}] (0,0.3) -- (0.7,1) node [left] {$b\omega_j$};
        \draw [dashed,decoration={markings, mark=at position 0.5 with {\arrow{>}}}, postaction={decorate}] (1,0.7) node [below] {$b\omega_j$} -- (0.3,0);
        \draw [dashed,decoration={markings, mark=at position 0.5 with {\arrow{>}}}, postaction={decorate}] (0,-0.3) -- (-0.7,-1) node [right] {$b\omega_j$};
        \draw [dashed,decoration={markings, mark=at position 0.5 with {\arrow{>}}}, postaction={decorate}] (-0.7,1) node [right] {$a\omega_i$}-- (0,0.3);
        \draw [dashed,decoration={markings, mark=at position 0.5 with {\arrow{>}}}, postaction={decorate}] (0.3,0) -- (1,-0.7) node [above] {$a\omega_i$};
        \draw [dashed,decoration={markings, mark=at position 0.5 with {\arrow{>}}}, postaction={decorate}] (0.7,-1)node [left] {$a\omega_i$}-- (0,-0.3);
        \draw [dashed,decoration={markings, mark=at position 0.5 with {\arrow{>}}}, postaction={decorate}] (-0.3,0) -- (-1,0.7) node [below] {$a\omega_i$};
        \foreach \i in {0,...,3}
        {
        \node at (90*\i:0.3) [] {$\bullet$};
        \draw [dashed,decoration={markings, mark=at position 0.7 with {\arrow{>}}}, postaction={decorate}] (90-\i*90:0.3) -- (-\i*90:0.3);
        }
        \foreach \i in {0,1} {
        \node at (135-\i*180:0.5) [] {$b\omega_j$};
        \node at (45-\i*180:0.5) [] {$a\omega_i$};
        }
        \end{tikzpicture}\hspace{1cm}
        \begin{tikzpicture}[baseline=0,scale =1.3]
        \foreach \i in {1,2,3}
        {
        \draw [blue] (0,0) -- (120*\i:1.5);
        \draw [red] (0,0) -- (60+120*\i:1.5);
        }
        \foreach \i in {1,...,6}
        {
        \draw [dashed,decoration={markings, mark=at position 0.5 with {\arrow{<}}}, postaction={decorate}] (60*\i:1.5) ++ (60*\i+60:0.6) -- (30+60*\i:1);
        \draw [dashed,decoration={markings, mark=at position 0.5 with {\arrow{>}}}, postaction={decorate}](60*\i+60:1.5) ++ (60*\i:0.6) -- (30+60*\i:1);
        \draw [dashed,decoration={markings, mark=at position 0.5 with {\arrow{>}}}, postaction={decorate}] (30+60*\i:1) -- (0,0);
        \node at (30+60*\i:1) [] {$\bullet$};
        }
        \foreach \i in {1,...,6}
        {
        \node at (75+60*\i:0.7) [] {$\mu_\i$};
        }
        \node at (0,0) [] {$\bullet$};
        \foreach \i in {1,3,5}
        {
        \node at (60*\i+60-20:1.9) [] {$a_\i\omega_i$};                \node at (60*\i+60+20:1.9) [] {$a_\i\omega_i$};
        \node [blue] at (60*\i+60:1.7) [] {$s_i$};
        }
        \foreach \i in {2,4,6}
        {
        \node at (60*\i+60-20:1.9) [] {$a_\i\omega_j$};                \node at (60*\i+60+20:1.9) [] {$a_\i\omega_j$};
        \node [red] at (60*\i+60:1.7) [] {$s_j$};
        }
        \end{tikzpicture}
        \caption{Weighted cycle representatives for Y-cycles.}
        \label{fig: Y-cycle representatives}
    \end{figure}
\end{itemize}

\begin{prop} The balancing condition is satisfied at all interior endpoints of weighted chains inside a weighted cycle representative of a Y-cycle.
\end{prop}
\begin{proof}The balancing condition on the two interior endpoints near a bivalent weave vertex is obvious.

Next, let us consider an interior endpoint near a trivalent weave vertex, say the left one in the left picture of Figure \ref{fig: Y-cycle representatives}. The total weights incident to that endpoint is
\begin{align*}
&a_2\omega_i+s_i.\mu_2-a_1\omega_i-\mu_1\\
=&a_2\omega_i+|a_1-a_3|\omega_i-|a_1-a_3|\alpha_i+[a_1-a_3]_+\alpha_i-a_1\omega_i-|a_3-a_2|\omega_i+[a_3-a_2]_+\alpha_i\\
=&(a_2+|a_1-a_3|-a_1-|a_3-a_2|)\omega_i+(-|a_1-a_3|+[a_1-a_3]_++[a_3-a_2]_+)\alpha_i.
\end{align*}
Depending on the relations among the multiplicities $a_1,a_2$, and $a_3$, we have three cases to consider.
\begin{enumerate}[(i)]
    \item $a_1=a_2\leq a_3$: in this case, the total incident weight is
    \[
    (a_2+(a_3-a_1)-a_1-(a_3-a_2))\omega_i+(-(a_3-a_1)+(a_3-a_2))\alpha_i=0.
    \]
    \item $a_1=a_3\leq a_2$: in this case, the total incident weight is
    \[
    (a_2-a_1-(a_2-a_3))\omega_i=0.
    \]
    \item $a_2=a_3\leq a_1$: in this case, the total incident weight is
    \[
    (a_2+(a_1-a_3)-a_1)\omega_i+(-(a_1-a_3)+(a_1-a_3))\alpha_i=0.
    \]
\end{enumerate}
The balancing condition on the other two endpoints can be proved by symmetric arguments.

The balancing condition on the four interior endpoints near a tetravalent weave vertex is obvious since there is an exact match of incoming and outgoing weighted chains at each interior endpoint.

Lastly, the only non-obvious balancing condition near a hexavalent weave vertex is on the center interior endpoint. By perturbing it upward slightly, we get the following picture.
\begin{figure}[H]
    \centering
    \begin{tikzpicture}[scale=0.8]
        \foreach \i in {1,2,3}
        {
        \draw [blue] (\i*120:1.5) -- (0,0);
        \draw [red] (\i*120+60:1.5) -- (0,0);
        }
        \foreach \i in {1,...,6}
        {
        \draw [dashed,decoration={markings, mark=at position 0.5 with {\arrow{>}}}, postaction={decorate}] (90+60*\i:1.5) -- (0.1,0.4);
        \node at (90+60*\i:1.7)[] {$\mu_\i$};
        }
        \node at (0.1,0.4) [] {$\bullet$};
    \end{tikzpicture}
    \caption{Perturbing the center interior endpoint at a hexavalent vertex.}
\end{figure}
Without loss of generality, let us assume that the blue weave edges are colored by $s_1$ and the red weave edges are colored by $s_2$; by construction, $\mu_i=a_{i+1}\omega_2-a_i\omega_1$ for $i=1,3,5$, and $\mu_i=a_{i+1}\omega_1-a_i\omega_2$ for $i=2,4,6$ (indices modulo $6$). Then the sum of the weights incident to the center interior endpoint is
\begin{align*}
    &s_1.\mu_1+s_1s_2.\mu_2+s_2s_1s_2.\mu_3+s_2s_1.\mu_4+s_2.\mu_5+\mu_6\\
    =&a_2\omega_2-a_1(\omega_1-\alpha_1)+a_3(\omega_1-\alpha_1)-a_2(\omega_2-\alpha_1-\alpha_2)+a_4(\omega_2-\alpha_1-\alpha_2)-a_3(\omega_1-\alpha_1-\alpha_2)\\
    &a_5(\omega_1-\alpha_1-\alpha_2)-a_4(\omega_2-\alpha_2)-+a_6(\omega_2-\alpha_2)-a_5\omega_1+a_1\omega_1-a_6\omega_2-\\
    =&(a_1+a_2-a_4-a_5)\alpha_1+(a_2+a_3-a_5-a_6)\alpha_2.
\end{align*}
Both coefficients vanish due to the condition that $a_1-a_4=a_3-a_6=a_5-a_2$.
\end{proof}

Furthermore, we can recover the intersection pairing between Y-cycles from the intersection pairing between their weighted cycle representatives. The following lemma is useful when computing the intersection pairing between complicated weighted cycles.

\begin{lem}\label{lem: boundary intersection} Let $C$ be a simple closed curve on the disk intersecting the weave $\ww$ at generic positions. Suppose $\ww$ does not have any trivalent vertices in the region enclosed by $C$. Let $\eta_1$ and $\eta_2$ be two weighted cycles and let $p$ be a point on $C$ that is away from $\eta_1$, $\eta_2$, and any weave edges in $\ww$. Then the contribution to $\{\eta_1,\eta_2\}$ from the region enclosed by $C$ is equal to $\{\eta'_1,\eta'_2\}_{C\setminus\{p\}}$, where $\eta'_i$ is the truncation of $\eta_i$ in an outward tubular neighborhood of $C$.
\end{lem}
\begin{proof} Note that the region enclosed by $C$ is also a disk. Since the restriction of $\ww$ to this smaller disk contains no trivalent weave vertices, and there is only one boundary base point $p$, by Theorem \ref{thm: weighted cycle algebra is a torus}, all weighted cycles inside this smaller disk are trivial. Let $\eta''_i$ be the truncation of $\eta_i$ inside this smaller disk. It follows from the discussion above that $\{\eta''_1,\eta''_2\}=0$. But the intersection pairing $\{\eta''_1,\eta''_2\}$ is the sum of the local contribution to the original intersection pairing $\{\eta_1,\eta_2\}$ together with the intersection pairing between $\eta''_1$ and $\eta''_2$ along the boundary interval $C\setminus\{p\}$. Thus, we can conclude that the local contribution to the original intersction pairing $\{\eta_1,\eta_2\}$ is $-\{\eta''_1,\eta''_2\}_{C\setminus\{p\}}=\{\eta'_1,\eta'_2\}_{C\setminus\{p\}}$.
\end{proof}

\begin{figure}[H]
    \centering
    \begin{tikzpicture}[baseline=0]
        \draw  [dashed,decoration={markings, mark=at position 0.7 with {\arrow{>}}}, postaction={decorate}] (-1,1) node [above left] {$\eta_1$}--(1,-1); 
        \draw [purple,dashed,decoration={markings, mark=at position 0.7 with {\arrow{>}}}, postaction={decorate}] (1,1) node [above right] {$\eta_2$} -- (-1,-1);
        \node [purple] at (0.3,0.3) [right] {$\mu_2$};
        \node at (-0.3,0.3) [left] {$\mu_1$};
        \draw (0,0) circle [radius=1];
        \node at (0,1) [] {$\bullet$};
        \node at (0,1) [above] {$p$};
        \node at (0,-1) [below] {$C$};
    \end{tikzpicture}\hspace{2cm}
    \begin{tikzpicture}[baseline=0]
    \draw (-2,-1) -- (2,-1) node [right] {$C\setminus \{p\}$};
    \draw [dashed,decoration={markings, mark=at position 0.5 with {\arrow{>}}}, postaction={decorate}] (1.5,1) -- node [right] {$\mu_1$} (1.5,-1);
    \draw [purple,dashed,decoration={markings, mark=at position 0.5 with {\arrow{<}}}, postaction={decorate}] (0.5,1) -- node [right] {$\mu_2$} (0.5,-1);
    \draw [dashed,decoration={markings, mark=at position 0.5 with {\arrow{<}}}, postaction={decorate}] (-0.5,1) -- node [right] {$\mu_1$} (-0.5,-1);
    \draw [purple,dashed,decoration={markings, mark=at position 0.5 with {\arrow{>}}}, postaction={decorate}] (-1.5,1) -- node [left] {$\mu_2$} (-1.5,-1);
    \end{tikzpicture}
    \caption{Demonstration of Lemma \ref{lem: boundary intersection}: the local contribution in the left picture is $(\mu_1,\mu_2)$, which is equal to the boundary interval contribution in the right picture.}
\end{figure}

\begin{thm}\label{thm: recovering pairing between Y-cycles} Let $\eta_1$ and $\eta_2$ be the weighted cycle representatives of Y-cycles $\gamma_1$ and $\gamma_2$, respectively. Then the intersection pairing $\{\gamma_1,\gamma_2\}$ is equal to the intersection pairing $\{\eta_1,\eta_2\}$.
\end{thm}
\begin{proof} We can draw the weighted cycle representative $\eta_2$ in a tubular neighborhood closer to the weave edges that $\eta_1$ so that $\eta_1$ and $\eta_2$ only intersect near the weave vertices. Then it suffices to check the intersection pairings between $\eta_1$ and $\eta_2$ at each weave vertex and on each boundary interval and then compare the results with the intersection pairing contribution to $\{\gamma_1,\gamma_2\}$ from weave vertices and external edges.

There are no intersections between $\eta_1$ and $\eta_2$ occurring near a bivalent weave vertex.

At a trivalent weave vertex, by construction, all intersections must occur between the $\eta_2$ weighted chains parallel to weave edges and the $\eta_1$ weighted chains perpendicular to weave edges. Let $a_1,a_2,a_3$ be the multiplicities for $\gamma_1$ and let $a'_1,a'_2,a'_3$ be the multiplicities for $\gamma_2$. Then at each weave edge, we will have an intersection pairing contribution of the form
\[
(a'_j\omega_i,s_i.\mu_j)-(a'_j\omega_i,\mu_j)=(a'_j\omega_i,s_i.\mu_j-\mu_j)=a'_j\left(2[a_{j-1}-a_{j+1}]_+-|a_{j-1}-a_{j+1}|\right)=a'_j(a_{j-1}-a_{j+1}).
\]
Thus, the total contribution locally near the trivalent weave vertex is
\[
\sum_{j=1}^3a'_j(a_{j-1}-a_{j+1})=\det\begin{pmatrix} 1 & 1 & 1 \\ a_1 & a_2 & a_3\\ a'_1 & a'_2 & a'_3\end{pmatrix}.
\]

At a tetravalent weave vertex, it is not hard to see that all intersections between $\eta_1$ and $\eta_2$ come in canceling pairs and therefore the total contribution locally near a tetravalent weave vertex is $0$.

At a hexavalent weave vertex, we first draw a small circle $C$ cutting out the hexavalent weave vertex and put a base point $p$ at a generic position along $C$. Then we apply Lemma \ref{lem: boundary intersection} and turn the problem to an intersection pairing computation along the interval $C\setminus \{p\}$. Without loss of generality, we may assume that the truncations of the weighted cycles look like the following.
\begin{figure}[H]
    \centering
    \begin{tikzpicture}[scale=1.5]
    \foreach \i in {1,...,6}
    {
    \draw [dashed,decoration={markings, mark=at position 0.5 with {\arrow{<}}}, postaction={decorate}] (-\i*2+0.4,0) -- (-\i*2+0.4,1);
    \draw [dashed,decoration={markings, mark=at position 0.5 with {\arrow{<}}}, postaction={decorate}](-\i*2-0.4,1) -- (-\i*2-0.4,0);
    \draw [purple,dashed,decoration={markings, mark=at position 0.5 with {\arrow{<}}}, postaction={decorate}] (-\i*2+0.2,0) -- (-\i*2+0.2,1);
    \draw [purple,dashed,decoration={markings, mark=at position 0.5 with {\arrow{<}}}, postaction={decorate}](-\i*2-0.2,1) -- (-\i*2-0.2,0);
    }
    \draw (-12.5,0) -- (-1.5,0);
    \foreach \i in {1,3,5}
    {
    \node at (-\i*2+0.4,1) [above] {$a_\i\omega_i$};
    \node [purple] at (-\i*2-0.2,1) [above] {$a'_\i\omega_i$};
    \draw [blue] (-\i*2,1) -- (-\i*2,0) node [below] {$s_i$};
    }
    \foreach \i in {2,4,6}
    {
    \node at (-\i*2+0.4,1) [above] {$a_\i\omega_j$};
    \node [purple] at (-\i*2-0.2,1) [above] {$a'_\i\omega_j$};
    \draw [red] (-\i*2,1) -- (-\i*2,0) node [below] {$s_j$};
    }
    \end{tikzpicture}
    \caption{Truncation of weighted cycles near a hexavalent weave vertex; the black weighted chains are $\eta'_1$ and the purple weighted chains are $\eta'_2$.}
\end{figure}
\noindent Note that within each of the six collections of weighted chains near weave edges, the pairing between the chains in $\eta'_1$ and $\eta'_2$ is $0$. Therefore we only need to pay attention to how $\eta'_1$ in each collection is paired with $\eta'_2$ in every other collection. This allows us to do some homotopies to simplify: for example, to pair $\eta'_1$ in the $1$st collection with $\eta'_2$ in every other collection, we may simplify the configuration to the following, which yields a contribution of
\[
-\frac{1}{2}a_1(a'_2-a'_3-2a'_4-a'_5+a'_6)=a_1a'_1.
\]
\begin{figure}[H]
    \centering
    \begin{tikzpicture}[xscale=1.5]
        \draw [dashed,decoration={markings, mark=at position 0.5 with {\arrow{<}}}, postaction={decorate}] (-1.2,1.5) node [above] {$a_1\alpha_i$} -- (-1.2,0);
        \foreach \i in {2,...,6}
        {
        \draw [purple,dashed,decoration={markings, mark=at position 0.5 with {\arrow{<}}}, postaction={decorate}] (-\i-0.2,1.5) -- (-\i-0.2,0);
        }
        \draw (-6.5,0) -- (-0.5,0);
        \foreach \i in {1,3,5}
        {
        \draw [blue] (-\i,1.5) -- (-\i,0) node [below] {$s_i$};
        \draw [red] (-\i-1,1.5) -- (-\i-1,0) node [below] {$s_j$};
        }
        \node [purple] at (-3.2,1.5) [above] {$a'_3\alpha_i$};
        \node [purple] at (-5.2,1.5) [above] {$a'_5\alpha_i$};
        \foreach \i in {2,4,6}
        {
        \node [purple] at (-\i-0.2,1.5) [above] {$a'_\i\alpha_j$};
        }
    \end{tikzpicture}
    \caption{Simplifying the weighted chains.}
\end{figure}
\noindent By similar computations, we will get
\begin{align*}
\{\eta'_1,\eta'_2\}_{C\setminus\{p\}}=&a_1a'_1+a_2(a'_4+a'_5)+a_3(a'_2-a'_4+a'_6)-a_4(a'_1-a'_3+a'_5)-a_5(a'_2+a'_3)-a_6a'_6\\
=&a_1(a'_3-a'_5)-a_3(a'_1-a'_5)+a_5(a'_1-a'_3)\\
=&\det\begin{pmatrix} 1 & 1 & 1\\
a_1 & a_3 & a_5 \\
a'_1 & a'_3 & a'_5\end{pmatrix}.
\end{align*}

For two external weave edges $e$ and $e'$ incident to the same boundary interval, we first perform the following homotopy move (7) to the two weighted cycle representatives.
\begin{figure}[H]
    \centering
     \begin{tikzpicture}[baseline=-30,scale=1.3]
        \draw [very thick] (-0.5,0) -- (3.5,0);
        \draw [blue] (0,0) node [above] {$e$} -- (0,-1) node [below] {$s_{i_0}$};
        \draw [red] (3,0) node [above] {$e'$} -- (3,-1) node [below] {$s_{i_k}$};
        \draw [dashed,decoration={markings, mark=at position 0.5 with {\arrow{>}}}, postaction={decorate}] (0.3,0) -- node [right] {$\omega_{i_0}$} (0.3,-1);
        \draw [dashed,decoration={markings, mark=at position 0.5 with {\arrow{<}}}, postaction={decorate}] (-0.3,0) -- node [left] {$\omega_{i_0}$} (-0.3,-1);
        \draw [dashed,decoration={markings, mark=at position 0.5 with {\arrow{>}}}, postaction={decorate}] (3.3,0) -- node [right] {$\omega_{i_k}$} (3.3,-1);
        \draw [dashed,decoration={markings, mark=at position 0.5 with {\arrow{<}}}, postaction={decorate}] (3-0.3,0) -- node [left] {$\omega_{i_k}$} (3-0.3,-1);
        \node at (1,-1) [below] {$s_{i_1}$};
        \node at (2,-1) [below] {$s_{i_{k-1}}$};
        \node at (1.5,-0.5) [] {$\cdots$};
        \node at (0,-1.5) [] {$\eta_e$};
        \node at (3,-1.5) [] {$\eta_{e'}$};
    \end{tikzpicture} $=$ 
    \begin{tikzpicture}[baseline=-30,scale=1.3]
        \draw [very thick] (-0.5,0) -- (3.5,0);
        \draw [blue] (0,0) node [above] {$e$} -- (0,-1) node [below] {$s_{i_0}$};
        \draw [red] (3,0) node [above] {$e'$} -- (3,-1) node [below] {$s_{i_k}$};
        \draw [dashed,decoration={markings, mark=at position 0.5 with {\arrow{<}}}, postaction={decorate}] (0.3,0) -- node [right] {$\alpha_{i_0}$} (0.3,-1);
        \draw [dashed,decoration={markings, mark=at position 0.5 with {\arrow{>}}}, postaction={decorate}] (3-0.3,0) -- node [left] {$\alpha_{i_k}$} (3-0.3,-1);
        \node at (1,-1) [below] {$s_{i_1}$};
        \node at (2,-1) [below] {$s_{i_{k-1}}$};
        \node at (1.5,-0.5) [] {$\cdots$};
        \node at (0.3,-1.5) [] {$\eta_e$};
        \node at (3-0.3,-1.5) [] {$\eta_{e'}$};
    \end{tikzpicture}
    \caption{Homotoping weighted cycle representatives along a boundary interval.}
\end{figure}
Then by definition, the intersection pairing
\[
\{\eta_e,\eta_{e'}\}=\frac{1}{2}(\alpha_{i_0},s_{i_1}\cdots s_{i_{k-1}}\alpha_{i_k}),
\]
which agrees with the definition of $\{e,e'\}$ in Definition \ref{defn: pairing between Y-cycles}.
\end{proof}

Note that the weighted cycle representatives not only recover the intersection pairing between Y-cycles, they also allow us to define an intersection pairing between Y-cycles and weighted cycles. 

\begin{defn}\label{defn: intersection number} Let $\gamma$ be a Y-cycle and let $\eta_\gamma$ be its weighted cycle representative. Then the intersection number between $\gamma$ and an arbitrary weighted cycle $\eta$ is defined to be
\[
\inprod{\gamma}{\eta}:=\{\eta_\gamma,\eta\}.
\]
\end{defn}

Since the intersection pairing between weighted cycles are invariant under homotopies of weighted cycles, we can hence deduce the following corollary.

\begin{cor} The intersection number between Y-cycles and weighted cycles is invariant under homotopies of weighted cycles.
\end{cor}

\subsection{Merodromies}\label{subsec: 3.6}
Recall that each weave defines a flag moduli space $\cM(\ww)$, each point of which corresponds to a configuration of flags associated with faces of the weave $\ww$ (Subsection \ref{subsec: 2.2}). In this subsection, we define a map called \emph{merodromies}, which allows us to view weighted cycles on $\ww$ as functions on the flag moduli space $\cM(\ww)$.

In order to define the merodromies of weighted cycles, we need to first orient the weave edges, which corresponds to the choice of one of the $h$ distances (Definition \ref{defn: h distance}). In addition, we want to make sure that the weave edges are oriented in a compatible fashion so that the merodromies are homotopy-invariant.

\begin{defn}\label{defn: compatible orientations} A choice of orientations on all edges of a weave is said to be \emph{compatible} if the following are satisfied:
\begin{itemize}
\item each bivalent weave vertex is either a source or a sink;
\item there are two incoming edges and one outgoing edge at each trivalent weave vertex;
\item there are two adjacent incoming edges and two adjacent outgoing edges at each tetravalent weave vertex;
\item there are three adjacent incoming edges and three adjacent outgoing edges at each hexavalent weave vertex. 
\end{itemize}
\end{defn}
\begin{figure}[H]
    \centering
    \begin{tikzpicture}[baseline=0]
        \draw [blue, decoration={markings, mark=at position 0.5 with {\arrow{>}}}, postaction={decorate}] (-1,0.5) -- (0,0.5);
        \draw [blue, decoration={markings, mark=at position 0.5 with {\arrow{>}}}, postaction={decorate}] (1,0.5) -- (0,0.5);
        \node [blue] at (0,0.5) [] {$\bullet$};
        \draw [blue, decoration={markings, mark=at position 0.5 with {\arrow{<}}}, postaction={decorate}] (-1,-0.5) -- (0,-0.5);
        \draw [blue, decoration={markings, mark=at position 0.5 with {\arrow{<}}}, postaction={decorate}] (1,-0.5) -- (0,-0.5);
        \node [blue] at (0,-0.5) [] {$\bullet$};
    \end{tikzpicture}
    \hspace{2cm}
    \begin{tikzpicture}[baseline=0]
        \draw[blue, decoration={markings, mark=at position 0.5 with {\arrow{>}}}, postaction={decorate}] (150:1) -- (0,0);
        \draw[blue, decoration={markings, mark=at position 0.5 with {\arrow{>}}}, postaction={decorate}] (30:1) -- (0,0);
        \draw[blue, decoration={markings, mark=at position 0.5 with {\arrow{>}}}, postaction={decorate}] (0,0) -- (0,-1);
        \end{tikzpicture}
\hspace{2cm}
    \begin{tikzpicture}[baseline=0]
        \draw[blue, decoration={markings, mark=at position 0.5 with {\arrow{>}}}, postaction={decorate}] (135:1) -- (0,0);
        \draw[blue, decoration={markings, mark=at position 0.5 with {\arrow{>}}}, postaction={decorate}] (0,0) -- (-45:1);
        \draw[teal, decoration={markings, mark=at position 0.5 with {\arrow{>}}}, postaction={decorate}] (45:1) -- (0,0);
        \draw[teal, decoration={markings, mark=at position 0.5 with {\arrow{>}}}, postaction={decorate}] (0,0) -- (-135:1);
    \end{tikzpicture}        \hspace{2cm}\begin{tikzpicture}[baseline=0]
        \draw[blue, decoration={markings, mark=at position 0.5 with {\arrow{>}}}, postaction={decorate}] (150:1) -- (0,0);
        \draw[blue, decoration={markings, mark=at position 0.5 with {\arrow{>}}}, postaction={decorate}] (30:1) -- (0,0);
        \draw[red, decoration={markings, mark=at position 0.5 with {\arrow{>}}}, postaction={decorate}] (0,1) -- (0,0);
        \draw[blue, decoration={markings, mark=at position 0.5 with {\arrow{>}}}, postaction={decorate}] (0,0) -- (0,-1);
        \draw[red, decoration={markings, mark=at position 0.5 with {\arrow{>}}}, postaction={decorate}] (0,0) -- (-30:1);
        \draw[red, decoration={markings, mark=at position 0.5 with {\arrow{>}}}, postaction={decorate}] (0,0) -- (-150:1);
    \end{tikzpicture}
    \caption{Compatible orientation on weave edges.}
    \label{fig: orientation}
\end{figure}

\begin{defn} Once a compatible orientation is chosen, we define the \emph{sign} of an intersection between a weighted chain $\eta$ and a weave edge as follows; note that the sign of an intersection is \textbf{NOT} the same as the intersection number defined in Definition \ref{defn: intersection number}.
\begin{figure}[H]
    \centering
    \begin{tikzpicture}[baseline=0]
    \draw[blue,<-] (0,-1) -- (0,1);
    \draw [dashed,->] (-1,0) -- (1,0);
    \node at (0,-1.5) [] {$+$};
    \end{tikzpicture} \hspace{2cm}     \begin{tikzpicture}[baseline=0]
    \draw[blue,->] (0,-1) -- (0,1);
    \draw [dashed,->] (-1,0) -- (1,0);
    \node at (0,-1.5) [] {$-$};
    \end{tikzpicture} 
    \caption{Signs of the intersection between a weighted chain and a weave edge.}
    \label{fig:intersection sign}
\end{figure}
\end{defn}

\begin{defn} Given a configuration of flags in $\cM_\fr(\ww)$, we first fix a decoration on each flag, making them decorated flags. Let $\eta$ be a weighted chain. The \emph{initial decorated flag} (resp. \emph{terminal decorated flag}) of $\eta$ is the decorated flag along the boundary if the source (resp. \emph{target}) of $\eta$ is on the boundary of the disk, and is otherwise the decorated flag associated with the face containing the interior endpoint.

If $\eta$ does not intersect any weave edges, then $\eta$ must be contained inside some face, and hence the underlying undecorated flag of its initial and terminal decorated flags are the same. This implies that the terminal decoration is equal to $t$ times the initial decoration for some $t\in T$, and we define the \emph{merodromy} of $\eta$ to be
\[
M^\eta:=t^\mu,
\]
where $\mu$ is the weight associated with $\eta$.

If $\eta$ does intersect weave edges, we break $\eta$ down into segments cut out by the weave edges. We associate the initial decorated flag with the segment containing the source of $\eta$ and the terminal decorated flag with the segment containing the target of $\eta$, and associate the decorated flag of the local face for each of the remaining segments. Now suppose $p$ is an intersection point between a weighted chain $\eta$ with a weave edge $e$ of color $s_i$ such that $s_i.\mu$ and $\mu$ are the weights attached to $\eta$ before and after $\eta$ crossing $e$, and $xN$ and $yN$ are the decorated flags associated with the segments before and after $\eta$ crossing $e$; then the \emph{local merodromy} $M_p$ is defined to be
\[
M_p:=h_\pm(xN,yN)^\mu,
\]
where the subscript of $h$ depends on the sign of the intersection at $p$. The \emph{merodromy} along $\eta$ is then defined to be
\[
M^\eta:=\prod_p M_p
\]
where $p$ runs through all intersection points between $\eta$ and weave edges.

We extend merodromies for weighted cycles as follows: if a weighted cycle $\eta$ consists of weighted chains $\eta_1,\dots, \eta_k$, then its \emph{merodromy} is
\[
M^\eta:=\prod_{i=1}^k M^{\eta_i}.
\]
\end{defn}

\begin{rmk} Note that by Lemma \ref{lem: invariant under diagonal G-action}, the $h$ distances are invariant under the diagonal $G$-action; thus, (local) merodromies are not affected by the diagonal $G$-action. We will use this fact to simplify our proofs. 
\end{rmk}

\begin{rmk} Note that the interior endpoints of weighted chains do not contribute anything to merodromies. Also, by convention, the empty weighted cycle has a merodromy of $1$.
\end{rmk}

It may seem that the merodromy of a weighted cycle depends on the choice of decorations on flags. In the next few propositions, we will prove that merodromies are in fact independent of such choice and is also invariant under homotopies of weighted cycles.

\begin{prop}\label{prop: reduced word} Suppose $\eta$ is a weighted chain passing through a collection of weave edges $e_1, e_2,\dots, e_l$ with colors $s_{i_1},s_{i_2},\dots, s_{i_l}$ such that $s_{i_1}s_{i_2}\cdots s_{i_l}$ form a reduced word for a Weyl group element $w$. Suppose all intersections are positive (resp. negative). Let $x_0N, x_1N,\dots, x_lN$ be the decorated flags associated with the segments in $\eta$. Then $M^\eta=h_+(x_0N,x_lN)^\mu$ (resp. $M^\eta=h_-(x_0N,x_lN)^\mu$ where $\mu$ is the weight attached to the end of $\eta$.
\end{prop}
\begin{proof} Without loss of generality let us assume that all intersections are positive. Suppose $x_{j-1}^{-1}x_j\in N\overline{s}_{i_j}h_jN$ for $j=1,2,\dots, l$. Then on the one hand, by definition, 
\[
M^\eta=\prod_{j=1}^lh_j^{w_{> j}.\mu}=\left(\prod_{j=1}^lw_{>j}^{-1}(h_j)\right)^\mu,
\]
where $w_{> j}:=s_{i_{j+1}}s_{i_{j+2}}\cdots s_{i_l}$. On the other hand, 
\[
x_0^{-1}x_l=(x_0^{-1}x_1)(x_1^{-1}x_2)\cdots (x_{l-1}^{-1}x_l)\in (N\overline{s}_{i_1}h_1N)(N\overline{s}_{i_2}h_2N)\cdots (N\overline{s}_{i_l}h_lN)=N\overline{w}\prod_{j=1}^lw_{>j}^{-1}(h_j)N.
\]
Thus $h_+(x_0N,x_lN)=\prod_{j=1}^lw_{>j}^{-1}(h_j)$ as well. 
\end{proof}

\begin{prop}\label{prop: independent of decoration} If $F$ is a face where a weighted chain $\eta$ passes through in the middle, i.e., it is neither the first face nor the last face, then merodromy $M^\eta$ only depends on the undecorated flag associated with $F$ and is independent of the decoration.
\end{prop}
\begin{proof} Suppose the two weave edges before and after the face $F$ are of colors $s_i$ and $s_j$ and suppose $\eta$ passes through faces with with decorated flags $x_0N$, $x_1N$, and $x_2N$ (see Figure \ref{fig: intermediate decoration}). Changing the decoration at $F$ is the same as changing the decorated flag $x_1N$ to $x_1tN$ for some $t\in T$. Under such a change, we have $h_\pm(x_0N, x_1tN)=th_\pm(x_0N,x_1N)$ and $h_\pm(x_1tN, x_2N)=s_j(t^{-1})h_\pm(x_1N,x_2N)$, where the signs in the subscript depend on the orientations of the weave edges. Suppose the weight associated with the middle portion of $\eta$ is $\mu$. Then the local contribution to the merodromy $M^\eta$ is 
\begin{align*}
h_\pm(x_0N,x_1tN)^\mu\cdot h_\pm(x_1tN,x_2N)^{s_j.\mu}=&(th_\pm(x_0N,x_1N))^\mu\cdot (s_j(t^{-1})h_\pm(x_1N,x_2N))^{s_j.\mu}\\
=& t^\mu 
\cdot h_\pm(x_0N,x_1N)^\mu\cdot t^{-\mu}\cdot h_\pm(x_1N,x_2N)^{s_j.\mu}\\
=& h_\pm(x_0N,x_1N)^\mu\cdot h_\pm(x_1N,x_2N)^{s_j.\mu}.
\end{align*}
This computation shows that the local contribution before and after the change of decoration are equal, and hence $M^\eta$ is invariant under the change of decorations on the intermediate faces.
\end{proof}
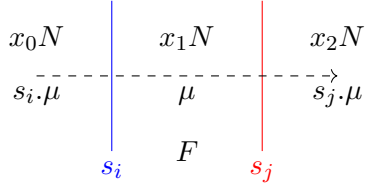
\begin{figure}[H]
    \centering
    \begin{tikzpicture}
        \draw [blue] (-1,-1) node [below] {$s_i$} -- (-1,1);
        \draw [red] (1,-1) node [below] {$s_j$} -- (1,1);
        \node at (0,0.5) [] {$x_1N$};
        \node at (2,0.5) [] {$x_2N$};
        \node at (-2,0.5) [] {$x_0N$};
        \node at (0,-1) [] {$F$};
        \draw [->, dashed] (-2,0) node [below] {$s_i.\mu$} -- node [below] {$\mu$} (2,0) node [below] {$s_j.\mu$};
    \end{tikzpicture}
    \caption{A weighted chain passing through a face.}
    \label{fig: intermediate decoration}
\end{figure}

\begin{prop} Suppose $F$ is a face that contains an interior endpoint $p$ for weighted chains $\eta_1,\eta_2,\dots, \eta_k$, each is oriented towards $p$ with weights $\mu_1,\mu_2,\dots, \mu_k$, respectively, such that $\mu_1+\mu_2+\cdots +\mu_k=0$. Then the total merodromy $\prod_{i=1}^k M^{\eta_k}$ only depends on the undecorated flag associated with $F$ and is independent of the decoration.
\end{prop}
\begin{proof} Without loss of generality we may assume that the decorated flag associated with the face $F$ is $N$. Now if we change it to $tN$ for some $t\in T$, the combined change to the total merodromy is 
\[
\prod_{i=1}^kt^{\mu_i}=t^{\sum_{i=1}^k\mu_i}=t^0=1.\qedhere
\]
\end{proof}

\begin{prop}\label{prop: merodromy invariant under homotopies} Merodromies are invariant under homotopies of weighted cycles. 
\end{prop}
\begin{proof} Since all contributions to merodromies are local, we just need to go through the list of homotopy moves (see Figures \ref{fig: homotopies} and \ref{fig: interior endpoint through a weave edge}) and prove the invariance.

Homotopy move (0) can be achieved by first adding an interior endpoint (homotopy move (8)), then pushing this interior endpoint through the weave edge (homotopy move (12)), and then deleting the interior endpoint (homotopy move (8)) again. Thus, its homotopy invariance follows from the homotopy invariance of other moves.

For the homotopy move (1), let us assume without loss of generality that the weave edge is oriented downward. Let us suppose the decorated flags are $xN$ on the left and $yN$ on the right with $h_+(xN,yN)=h$, i.e., $x^{-1}y\in N\overline{s}_ihN$. Then the merodromy on the LHS is $h^{s_i.\mu}$. Note that 
\[
y^{-1}x\in Nh^{-1}\overline{s}_i^{-1}N=Nh^{-1}\doverline{s}_iN=N\doverline{s}_i(s_i(h^{-1}))N.
\]
Thus, the merodromy on the RHS is
$(s_i(h^{-1}))^{-\mu}=h^{s_i.\mu}$ as well.

For the homotopy move (2), the initial decorated flag and terminal decorated flag of the weighted chain are the same and hence $M^\eta=1$.

For the homotopy move (3), we observe that it is almost the composition of homotopy moves (0) and (2), except that one of the two merodromies is computed with the opposite $h$ distance (i.e., $h_\mp$ instead of $h_\pm$). It then follows from Lemma \ref{lem: signs of h-distances} that this homotopy move results in a sign $((-1)^{\alpha_i^\vee})^\mu=(-1)^{\inprod{\alpha_i^\vee}{\mu}}$.

For the homotopy move (4), since the picture has a rotational symmetry, we may assume without loss of generality that all the weave edges are oriented downward. Furthermore, we may first use the diagonal $G$-action to move the decorated flags so that the decorated flag associated with the left face is $N$; we then apply Proposition \ref{prop: independent of decoration} to assume that the decorated flag associated with the top face is of the form $x_i(p)\overline{s}_iN$; under this setup, the merodromy across the left weave edge is $1$. The decorated flag associated with the right face is of the form $x_i(q)\overline{s}_ihN$ for some $h\in T$ and $q\neq p$, and as a result, the merodromy across the bottom weave edge is $h_-(x_i(q)\overline{s}_ihN, N)^\mu=s_i(h^{-1})^{\mu}=h^{-\mu}$ (note that $s_i.\mu=\mu$ by assumption). The remaining merodromy across the right weave edge is $h_+(x_i(p)\overline{s}_iN, x_i(q)\overline{s}_ihN)^{\mu}$. By applying the identity
\[
\overline{s}_i^{-1}x_i(a)=x_i(-a^{-1})(-a)^{-\alpha_i^\vee}y_i(a^{-1}),
\]
we note that
\[
N\overline{s}_i^{-1}x_i(p)^{-1}x_i(q)\overline{s}_ihN=N\overline{s}_i^{-1}x_i(q-p)\overline{s}_ihN=N(p-q)^{-\alpha_i^\vee}\overline{s}_ihN=N\overline{s}_i(p-q)^{\alpha_i^\vee}hN.
\]
Since $s_i.\mu=\mu$, we know that $\inprod{\alpha_i^\vee}{\mu}=0$. Thus, we can conclude that the merodromy across the right weave edge is
\[
((p-q)^{\alpha_i^\vee}h)^{\mu}=h^\mu.
\]
Thus, the total merodromy is $1\cdot h^{-\mu}\cdot h^\mu=1$.

For the homotopy moves (5) and (6), we can add bivalent interior endpoints on the left-most and the right-most points of the circle to break the circle into two semicircular arcs. Note that by definition, reversing the orientation on a weighted chain without changing the weights results in a merodromy that is the inverse of the original. Thus, the invariance of merodromies for homotopy moves (5) and (6) is equivalent to showing that the merodromy across the upper semicircular arc is equal to the merodromy across the lower semicircular arc, which follows straight from Proposition \ref{prop: reduced word}.

For the homotopy move (7), it suffices to notice that the merodromy from the intersection on the RHS is $1$ due to the condition on the decorated flags along the boundary (Definition \ref{defn: framed moduli space}).

Homotopy moves (8)-(11) do not involve any non-trivial contribution to merodromies and there is nothing to show.

Lastly, for the homotopy move (12), let us assume without loss of generality that the weave edge is oriented downward; the other orientation can be verified similarly with the $h$ distance $h_\pm$ replaced by $h_\mp$. Let $xN$ be the decorated flag on the left of the weave edge and let $yN$ be the decorated flag on the right of the weave edge. Suppose $x^{-1}y\in N\overline{s}_ihN$. Then on the one hand, the merodromy on the LHS is $h^{\mu_1}$. On the other hand, since $y^{-1}x\in Nh^{-1}\overline{s}_i^{-1}N=N\doverline{s}_is_i(h^{-1})N$, the merodromy on the RHS is
\[
(s_i(h^{-1}))^{s_i.\mu_2+\cdots+s_i.\mu_k}=h^{-\mu_2-\cdots-\mu_k}=h^{\mu_1}. \qedhere
\]
\end{proof}

As a result of Proposition \ref{prop: merodromy invariant under homotopies}, we can view merodromies as an algebra homomorphism
\[
M^\bullet:\cW(\ww)\longrightarrow \cO(\cM_\fr(\ww)).
\]

Recall that within the weighted cycle algebra, there is a subalgebra $\cW_\abs(\ww)$ generated by weighted absolute cycles. From the projection map $\cM_\fr(\ww)\rightarrow \cM(\ww)$, we also get a subalgebra $\cO(\cM(\ww))\subset \cO(\cM_\fr(\ww))$. 

\begin{cor}\label{cor: abs weighted cycle algebra merodromy} The merodromy homomorphism $M^\bullet:\cW(\ww)\rightarrow \cO(\cM_\fr(\ww))$ restricts to a homomorphism
\[
M^\bullet:\cW_\abs(\ww)\longrightarrow \cO(\cM(\ww)).
\]
\end{cor}
\begin{proof} Since weighted absolute cycles do not have endpoints along the boundary of the disk, the framing data of decorated flags along the boundary of the disk do not affect the merodromies along weighted absolute cycles. This shows that the merodromy homomorphism does restrict to a homomorphism from $\cW_\abs(\ww)$ to $\cO(\cM(\ww))$.
\end{proof}

Note that by construction, merodromies along weighted cycles depend on the choice of compatible orientation on the weave edges. Next, let us discuss how different choices of compatible orientations may affect merodromies.

\begin{defn} A \emph{dissecting path} on a weave $\ww$ is a simple curve consisting of consecutive weave edges such that
\begin{itemize}
    \item at a tetravalent weave vertex on its path, the two weave edges are the $i$th and the $(i+2)$nd edges (modulo $4$) with respect to any cyclic orientation on the incident weave edges;
    \item at a hexavalent weave vertex on its path, the two weave edges are the $i$th and the $(i+3)$rd edges (modulo $6$) with respect to any cyclic orientation on the incident weave edges;
    \item it is maximal with respect to inclusion.
\end{itemize}
Note that a dissecting path either forms a closed loop or travels across the disk, and hence it cuts the disk into two connected components.
\end{defn}

\begin{lem} Any two choices of compatible orientations on the same weave differ by reversing orientations on a collection of dissecting paths.
\end{lem}
\begin{proof} Let us do an induction on the number $n$ of weave edges along which the two choices of compatible orientations disagree. There is nothing to show for the base case $n=0$. Now let $e$ be a weave edge along which the two choices disagree. Then at each of the endpoints of $e$, there must be another edge where the two choices disagree; in particular, if the endpoint is a tetravalent or a hexavalent weave vertex, the weave edge opposite to $e$ must have disagreeing orientations as well. Thus we can extend $e$ into a path consisting of consecutive weave edges. Eventually this process has to stop, either when the path closes itself up into a loop or goes to the boundary of the disk, and thus giving us a dissecting path. Now reversing one choice of compatible orientations along this dissecting path gives us another choice of compatible orientations with fewer disagreeing edges compared to the other choice, and the proof is complete by induction.
\end{proof}

\begin{prop} Let us fix a choice of compatible orientations on $\ww$. Let $\eta$ be a weighted cycle with merodromy $M_1^\eta$. Let $\gamma$ be a dissecting path and let $M_2^\eta$ be the merodromy along $\eta$ with respect to the choice of compatible orientations obtained by reversing the edges in $\gamma$. Then $M_2^\eta=(-1)^{\inprod{\gamma}{\eta}}M_1^\eta$.
\end{prop}
\begin{proof} By Lemma \ref{lem: signs of h-distances}, the two $h$-distances between two decorated flags separated by a weave edge of color $s_i$ differ by $(-1)^{\alpha_i^\vee}$. Therefore we may conclude that $M_2^\eta=M_1^\eta\prod_k (-1)^{\inprod{\alpha_{i_k}^\vee}{\mu_k}}$, where $k$ runs through all intersection points between $\eta$ and $\gamma$, $\mu_k$ is the weight on $\eta$ right after the intersection point $k$, and $s_{i_k}$ is the color of the weave edge at the intersection point $k$. The proposition then follows from the fact that $\sum_k \inprod{\alpha_{i_k}^\vee}{\mu_k}$ is precisely the intersection number $\inprod{\gamma}{\eta}$.
\end{proof}

\begin{exmp} Below is an example of two choices of compatible orientations on the same $\SL_3$-weave together with the same weighted absolute cycle whose merodromies would differ by a factor of $-1$.
\begin{figure}[H]
    \centering
    \begin{tikzpicture}
        \draw [blue,decoration={markings, mark=at position 0.5 with {\arrow{>}}}, postaction={decorate}] (-1.5,0) -- (0,0);
        \draw [blue,decoration={markings, mark=at position 0.3 with {\arrow{>}}}, postaction={decorate}] (-1.5,0) ++ (120:1) -- (-1.5,0);
        \draw [blue,decoration={markings, mark=at position 0.3 with {\arrow{>}}}, postaction={decorate}] (-1.5,0) ++ (-120:1) -- (-1.5,0);
        \draw [red,decoration={markings, mark=at position 0.3 with {\arrow{>}}}, postaction={decorate}] (120:1) -- (0,0);
        \draw [red,decoration={markings, mark=at position 0.3 with {\arrow{>}}}, postaction={decorate}] (-120:1) -- (0,0);
        \draw [blue,decoration={markings, mark=at position 0.8 with {\arrow{>}}}, postaction={decorate}] (0,0) -- (60:1);
        \draw [blue,decoration={markings, mark=at position 0.8 with {\arrow{>}}}, postaction={decorate}] (0,0) -- (-60:1);
        \draw [red,decoration={markings, mark=at position 0.5 with {\arrow{>}}}, postaction={decorate}] (0,0) -- (1.5,0);
        \draw [red,decoration={markings, mark=at position 0.3 with {\arrow{>}}}, postaction={decorate}] (1.5,0) ++ (60:1) -- (1.5,0);
        \draw [red,decoration={markings, mark=at position 0.8 with {\arrow{>}}}, postaction={decorate}] (1.5,0) -- ++ (-60:1);
        \draw [dashed, ->, decoration={markings, mark=at position 0.5 with {\arrow{>}}}, postaction={decorate}] (-1,0.5) -- (1.5,0.5) arc (90:-90:0.5) -- (-1.5,-0.5) arc (-90:-270:0.5) -- (-1,0.5);
        \node at (-2,0) [left] {$s_1.\omega_1$};
        \node at (-1,0.5) [above] {$\omega_1$};
        \node at (0,0.5)[above] {$\omega_1$};
        \node at (1,0.5) [above] {$s_1.\omega_1$};
        \node at (2,0) [right] {$w_0.\omega_1$};
        \node at (1,-0.5) [below] {$s_1.\omega_1$};
        \node at (0,-0.5) [below] {$\omega_1$};
        \node at (-1,-0.5) [below] {$\omega_1$};
    \end{tikzpicture} \hspace{1cm}
    \begin{tikzpicture}
        \draw [blue,decoration={markings, mark=at position 0.5 with {\arrow{>}}}, postaction={decorate}] (-1.5,0) -- (0,0);
        \draw [blue,decoration={markings, mark=at position 0.3 with {\arrow{>}}}, postaction={decorate}] (-1.5,0) ++ (120:1) -- (-1.5,0);
        \draw [blue,decoration={markings, mark=at position 0.3 with {\arrow{>}}}, postaction={decorate}] (-1.5,0) ++ (-120:1) -- (-1.5,0);
        \draw [red,decoration={markings, mark=at position 0.2 with {\arrow{<}}}, postaction={decorate}] (120:1) -- (0,0);
        \draw [red,decoration={markings, mark=at position 0.3 with {\arrow{>}}}, postaction={decorate}] (-120:1) -- (0,0);
        \draw [blue,decoration={markings, mark=at position 0.8 with {\arrow{>}}}, postaction={decorate}] (0,0) -- (60:1);
        \draw [blue,decoration={markings, mark=at position 0.8 with {\arrow{<}}}, postaction={decorate}] (0,0) -- (-60:1);
        \draw [red,decoration={markings, mark=at position 0.5 with {\arrow{>}}}, postaction={decorate}] (0,0) -- (1.5,0);
        \draw [red,decoration={markings, mark=at position 0.3 with {\arrow{>}}}, postaction={decorate}] (1.5,0) ++ (60:1) -- (1.5,0);
        \draw [red,decoration={markings, mark=at position 0.8 with {\arrow{>}}}, postaction={decorate}] (1.5,0) -- ++ (-60:1);
        \draw [dashed, ->, decoration={markings, mark=at position 0.5 with {\arrow{>}}}, postaction={decorate}] (-1,0.5) -- (1.5,0.5) arc (90:-90:0.5) -- (-1.5,-0.5) arc (-90:-270:0.5) -- (-1,0.5);
        \node at (-2,0) [left] {$s_1.\omega_1$};
        \node at (-1,0.5) [above] {$\omega_1$};
        \node at (0,0.5)[above] {$\omega_1$};
        \node at (1,0.5) [above] {$s_1.\omega_1$};
        \node at (2,0) [right] {$w_0.\omega_1$};
        \node at (1,-0.5) [below] {$s_1.\omega_1$};
        \node at (0,-0.5) [below] {$\omega_1$};
        \node at (-1,-0.5) [below] {$\omega_1$};
    \end{tikzpicture}
    \caption{Two choices of compatible orientations such that the merodromies along the same weighted cycle differ by a sign.}
\end{figure}
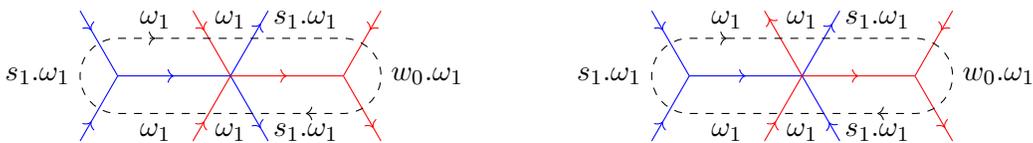
\end{exmp}

\begin{rmk}Let us discuss the significance of introducing bivalent weave vertex for computing merodromies. First, let us observe that without bivalent weave vertices, there exist weaves that do not admit any choices of compatible orientations (e.g., the left picture in Figure \ref{fig: no compatible orientations}). In addition, although weave equivalences are local moves (Definition \ref{defn: weave equivalences}), there is no guarantee that one can adjust a compatible choice of orientations locally to obtain a compatible choice of orientations after a weave equivalence (e.g., the right picture in Figure \ref{fig: no compatible orientations}). 

\begin{figure}[H]
    \centering
    \begin{tikzpicture}[baseline=0]
        \draw [blue] (0,0) circle [radius=1];
        \draw [blue] (-1,0) -- (1,0);
    \end{tikzpicture} \hspace{4cm}
    \begin{tikzpicture}[baseline=0]
    \draw [blue, decoration={markings, mark=at position 0.5 with {\arrow{>}}}, postaction={decorate}] (-1,-1) -- (-0.5,0);
    \draw [blue, decoration={markings, mark=at position 0.5 with {\arrow{>}}}, postaction={decorate}](-1,1) -- (-0.5,0);
    \draw [blue, decoration={markings, mark=at position 0.5 with {\arrow{>}}}, postaction={decorate}] (-0.5,0) -- (0.5,0);
    \draw [red, decoration={markings, mark=at position 0.5 with {\arrow{>}}}, postaction={decorate}] (0,1) -- (0.5,0);
    \draw [red, decoration={markings, mark=at position 0.5 with {\arrow{>}}}, postaction={decorate}] (0,-1) -- (0.5,0);
    \draw [blue, decoration={markings, mark=at position 0.5 with {\arrow{>}}}, postaction={decorate}] (0.5,0) -- (1,-1);
    \draw [blue, decoration={markings, mark=at position 0.5 with {\arrow{>}}}, postaction={decorate}] (0.5,0) -- (1,1);
    \draw [red, decoration={markings, mark=at position 0.5 with {\arrow{>}}}, postaction={decorate}] (0.5,0) -- (1,0);
    \end{tikzpicture} \quad ${\longleftrightarrow}$ \begin{tikzpicture}[baseline=0]
    \draw [blue, decoration={markings, mark=at position 0.5 with {\arrow{>}}}, postaction={decorate}] (-1,-1) -- (0,-0.5);
    \draw[blue, decoration={markings, mark=at position 0.5 with {\arrow{>}}}, postaction={decorate}] (0,-0.5) -- (1,-1);
    \draw [blue, decoration={markings, mark=at position 0.5 with {\arrow{>}}}, postaction={decorate}] (-1,1) -- (0,0.5);
    \draw [blue, decoration={markings, mark=at position 0.5 with {\arrow{>}}}, postaction={decorate}] (0,0.5) -- (1,1);
    \draw [red, decoration={markings, mark=at position 0.5 with {\arrow{>}}}, postaction={decorate}] (0,1) -- (0,0.5);
    \draw [red] (0,0.5) to [out=-150,in=90] (-0.5,0) to [out=-90,in=150] (0,-0.5);
    \draw [red, decoration={markings, mark=at position 0.5 with {\arrow{>}}}, postaction={decorate}] (0,-1) -- (0,-0.5);
    \draw [red, decoration={markings, mark=at position 0.5 with {\arrow{>}}}, postaction={decorate}] (0,0.5) -- (0.5,0);
    \draw [red, decoration={markings, mark=at position 0.5 with {\arrow{>}}}, postaction={decorate}] (0,-0.5) -- (0.5,0);
    \draw [red, decoration={markings, mark=at position 0.5 with {\arrow{>}}}, postaction={decorate}] (0.5,0) -- (1,0);
    \draw [blue] (0,0.5) -- (0,-0.5);
    \end{tikzpicture}
    \caption{Examples of non-existence of compatible orientations.}
    \label{fig: no compatible orientations}
\end{figure}
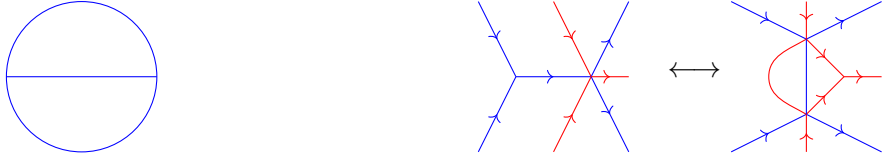
\noindent By introducing bivalent weave vertices, we can give these weaves compatible orientations.
\begin{figure}[H]
    \centering
    \begin{tikzpicture}[baseline=0]
        \draw [blue, decoration={markings, mark=at position 0.5 with {\arrow{>}}}, postaction={decorate}] (-1,0) arc (180:0:1);
        \draw [blue, decoration={markings, mark=at position 0.5 with {\arrow{>}}}, postaction={decorate}] (1,0) arc (0:-180:1);
        \draw [blue, decoration={markings, mark=at position 0.5 with {\arrow{>}}}, postaction={decorate}] (0,0) -- (1,0);
         \draw [blue, decoration={markings, mark=at position 0.5 with {\arrow{>}}}, postaction={decorate}] (0,0) -- (-1,0);
         \node [blue] at (0,0) [] {$\bullet$};
    \end{tikzpicture} \hspace{4cm}
    \begin{tikzpicture}[baseline=0]
    \draw [blue, decoration={markings, mark=at position 0.5 with {\arrow{>}}}, postaction={decorate}] (-1,-1) -- (-0.5,0);
    \draw [blue, decoration={markings, mark=at position 0.5 with {\arrow{>}}}, postaction={decorate}](-1,1) -- (-0.5,0);
    \draw [blue, decoration={markings, mark=at position 0.5 with {\arrow{>}}}, postaction={decorate}] (-0.5,0) -- (0.5,0);
    \draw [red, decoration={markings, mark=at position 0.5 with {\arrow{>}}}, postaction={decorate}] (0,1) -- (0.5,0);
    \draw [red, decoration={markings, mark=at position 0.5 with {\arrow{>}}}, postaction={decorate}] (0,-1) -- (0.5,0);
    \draw [blue, decoration={markings, mark=at position 0.5 with {\arrow{>}}}, postaction={decorate}] (0.5,0) -- (1,-1);
    \draw [blue, decoration={markings, mark=at position 0.5 with {\arrow{>}}}, postaction={decorate}] (0.5,0) -- (1,1);
    \draw [red, decoration={markings, mark=at position 0.5 with {\arrow{>}}}, postaction={decorate}] (0.5,0) -- (1,0);
    \end{tikzpicture} \quad ${\longleftrightarrow}$ \begin{tikzpicture}[baseline=0]
    \draw [blue, decoration={markings, mark=at position 0.5 with {\arrow{>}}}, postaction={decorate}] (-1,-1) -- (0,-0.5);
    \draw[blue, decoration={markings, mark=at position 0.5 with {\arrow{>}}}, postaction={decorate}] (0,-0.5) -- (1,-1);
    \draw [blue, decoration={markings, mark=at position 0.5 with {\arrow{>}}}, postaction={decorate}] (-1,1) -- (0,0.5);
    \draw [blue, decoration={markings, mark=at position 0.5 with {\arrow{>}}}, postaction={decorate}] (0,0.5) -- (1,1);
    \draw [red, decoration={markings, mark=at position 0.5 with {\arrow{>}}}, postaction={decorate}] (0,1) -- (0,0.5);
    \draw [red, decoration={markings, mark=at position 0.5 with {\arrow{>}}}, postaction={decorate}] (-0.5,0) to [out=-90,in=150] (0,-0.5);
    \draw [red, decoration={markings, mark=at position 0.5 with {\arrow{>}}}, postaction={decorate}] (-0.5,0) to [out=90,in=-150] (0,0.5);
    \node [red] at (-0.5,0) [] {$\bullet$};
    \draw [red, decoration={markings, mark=at position 0.5 with {\arrow{>}}}, postaction={decorate}] (0,-1) -- (0,-0.5);
    \draw [red, decoration={markings, mark=at position 0.5 with {\arrow{>}}}, postaction={decorate}] (0,0.5) -- (0.5,0);
    \draw [red, decoration={markings, mark=at position 0.5 with {\arrow{>}}}, postaction={decorate}] (0,-0.5) -- (0.5,0);
    \draw [red, decoration={markings, mark=at position 0.5 with {\arrow{>}}}, postaction={decorate}] (0.5,0) -- (1,0);
    \draw [blue, decoration={markings, mark=at position 0.5 with {\arrow{>}}}, postaction={decorate}] (0,0.5) -- (0,0);
    \draw [blue, decoration={markings, mark=at position 0.5 with {\arrow{>}}}, postaction={decorate}] (0,-0.5) -- (0,0);
    \node [blue] at (0,0) [] {$\bullet$};
    \end{tikzpicture}
    \caption{Compatible orientations after adding bivalent weave vertices.}
\end{figure}
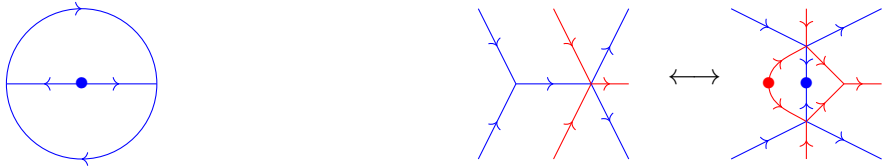
\noindent Note that bivalent weave vertices are not needed if we are not computing merodromies. Also, since bivalent weave vertices only introduce a possible sign to merodromies, they are not needed if the base field is of characteristic $2$.
\end{rmk}

\subsection{Mutations} In this subsection we describe how a family of weighted cycles change under a weave mutation at a short I-cycle.

\begin{defn}\label{defn: mutation of weighted cycles} Suppose $\ww$ and $\ww'$ are two weaves differing by a mutation at a short I-cycle at the weave edge $e$. If a weighted chain $\eta$ on $\ww$ has an intersection number $\inprod{e}{\eta}=0$, then we can use homotopy move (4) to move it out of the local mutating region and keep it as a weighted chain on $\ww'$. In contrast, if $\eta$ is a weighted chain with an intersection number $\inprod{e}{\eta}=1$, then we define the \emph{mutation} of $\eta$ at $e$ to be $\eta_1+\eta_2$ in $\cW(\ww')$, where $\eta_1$ and $\eta_2$ are depicted in Figure \ref{fig:merodromy mutation}. 
\end{defn}
\begin{figure}[H]
    \centering
    \begin{tikzpicture}[baseline=0]
        \draw [blue] (-1,-1) -- (-0.5,0);
    \draw [blue] (-1,1) -- (-0.5,0);
        \draw [blue] (1,-1) -- (0.5,0);
        \draw [blue] (0.5,0) -- (1,1);
        \draw [blue] (-0.5,0) -- (0.5,0);
        \draw [dashed,->] (0,-1) node [below] {$s_i.\mu$} -- (0,1) node [above] {$\mu$};
        \node at (0,-1.8) [] {$\eta$};
    \end{tikzpicture} \hspace{1cm} $\longrightarrow$ \hspace{1cm}
        \begin{tikzpicture}[baseline=0]
        \draw [blue] (-1,-1) -- (0,-0.5);
        \draw [blue] (1,-1) -- (0,-0.5);
        \draw [blue] (-1,1) -- (0,0.5);
        \draw [blue] (0,0.5) -- (1,1);
        \draw [blue] (0,-0.5) -- (0,0.5);
        \draw [dashed,->] (0,-1) node [below] {$s_i.\mu$} -- (0.5,0) node [right] {$\mu$} -- (-0.5,0) node [left] {$s_i.\mu$} -- (0,1) node [above] {$\mu$};
        \node at (0,-1.8) [] {$\eta_1$};
    \end{tikzpicture}
    \quad $+$ \quad
    \begin{tikzpicture}[baseline=0]
        \draw [blue] (-1,-1) -- (0,-0.5);
        \draw [blue] (1,-1) -- (0,-0.5);
        \draw [blue] (-1,1) -- (0,0.5);
        \draw [blue] (0,0.5) -- (1,1);
        \draw [blue] (0,-0.5) -- (0,0.5);
        \draw [dashed,->] (0,-1) node [below] {$s_i.\mu$} -- (-0.5,0) node [left] {$\mu$} -- (0.5,0) node [right] {$s_i.\mu$} -- (0,1) node [above] {$\mu$};
        \node at (0,-1.8) [] {$\eta_2$};
    \end{tikzpicture}
    \caption{Mutation of a weighted chain whose intersection number with the short I-cycle is $1$.}
    \label{fig:merodromy mutation}
\end{figure}

\begin{defn} Let $e$ be a short I-cycle on a weave $\ww$. Within the weighted cycle algebra $\cW(\ww)$, we define the subalgebra $\cW_e^+(\ww)$ of \emph{non-negative weighted cycles} at $e$ to be the subalgebra generated by weighted cycles with representatives such that all of whose weighted chains have non-negative intersection numbers with $e$. Then the mutation operation above induces an algebra homomorphism 
\[
\mu_e:\cW_e^+(\ww)\longrightarrow  \cW(\ww'),
\]
where $\ww'$ is the weave obtained from $\ww$ by a mutation at $e$.
\end{defn}

\begin{thm}\label{thm: mutation} The mutation homomorphism $\mu_e:\cW_e^+(\ww)\longrightarrow  \cW(\ww')$ commutes with the merodromy homomorphisms.
\end{thm}
\begin{proof} It suffices to prove that when $\inprod{e}{\eta}=1$, $M^\eta=M^{\eta_1}+M^{\eta_2}$ where $\eta_1+\eta_2$ is the mutation of $\eta$ at $e$. Let us use the weights in Figure \ref{fig:merodromy mutation} for reference. By Definition \ref{defn: intersection number}, we know that $\inprod{\alpha_i^\vee}{\mu}=1$. Without loss of generality, let us assume that the orientations on weave edges and the surrounding decorated flags are chosen as follows (with $p\neq q$ and $p,q\neq 0$). It follows that $M^\eta=h_2^\mu$.
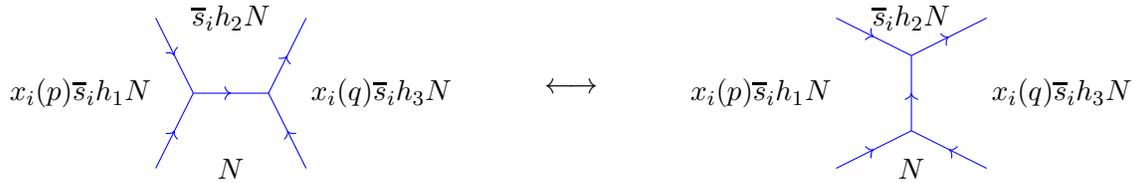
\[
\begin{tikzpicture}[baseline=0]
        \draw [blue, decoration={markings, mark=at position 0.5 with {\arrow{>}}}, postaction={decorate}] (-1,-1) -- (-0.5,0);
         \draw [blue, decoration={markings, mark=at position 0.5 with {\arrow{>}}}, postaction={decorate}] (-1,1) -- (-0.5,0);
        \draw [blue, decoration={markings, mark=at position 0.5 with {\arrow{>}}}, postaction={decorate}] (1,-1) -- (0.5,0);
        \draw  [blue, decoration={markings, mark=at position 0.5 with {\arrow{>}}}, postaction={decorate}] (0.5,0) -- (1,1);
        \draw [blue, decoration={markings, mark=at position 0.5 with {\arrow{>}}}, postaction={decorate}] (-0.5,0) -- (0.5,0);
        \node at (0,-1) [] {$N$};
        \node at (0,1) [] {$\overline{s}_ih_2N$};
        \node at (-2,0) [] {$x_i(p)\overline{s}_ih_1N$};
        \node at (2,0) [] {$x_i(q)\overline{s}_ih_3N$};
    \end{tikzpicture} \hspace{1cm} \longleftrightarrow \hspace{1cm}
        \begin{tikzpicture}[baseline=0]
        \draw [blue, decoration={markings, mark=at position 0.5 with {\arrow{>}}}, postaction={decorate}] (-1,-1) -- (0,-0.5);
        \draw [blue, decoration={markings, mark=at position 0.5 with {\arrow{>}}}, postaction={decorate}] (1,-1) -- (0,-0.5);
        \draw [blue, decoration={markings, mark=at position 0.5 with {\arrow{>}}}, postaction={decorate}] (-1,1) -- (0,0.5);
         \draw [blue, decoration={markings, mark=at position 0.5 with {\arrow{>}}}, postaction={decorate}] (0,0.5) -- (1,1);
        \draw [blue, decoration={markings, mark=at position 0.5 with {\arrow{>}}}, postaction={decorate}] (0,-0.5) -- (0,0.5);
        \node at (0,-1) [] {$N$};
        \node at (0,1) [] {$\overline{s}_ih_2N$};
        \node at (-2,0) [] {$x_i(p)\overline{s}_ih_1N$};
        \node at (2,0) [] {$x_i(q)\overline{s}_ih_3N$};
    \end{tikzpicture}
\]
For the right pictures, on the one hand, the signs of the crossings of $\eta_1$ with the weave edges are $-, +, +$; the $h_-$ distance between the bottom and the right decorated flags is $(-1)^{\alpha_i^\vee}h_3$; the $h_+$ distance between the right and the left decorated flags is $s_i(h_3^{-1})(q-p)^{-\alpha_i^\vee}h_1$; the $h_+$ distance between the left and the top decorated flags is $h_1^{-1}p^{\alpha_i^\vee}h_2$. Thus $M^{\eta_1}$ is equal to
\[
M^{\eta_1}=((-1)^{\alpha_i^\vee}h_3)^{\mu}\cdot (s_i(h_3^{-1})(q-p)^{-\alpha_i^\vee}h_1)^{s_i.\mu}\cdot (s_i(h_1^{-1})p^{\alpha_i^\vee}h_2)^\mu=\left(\frac{p}{p-q}\right)h_2^\mu.
\]
On the other hand, the signs of the crossings of $\eta_2$ with the weave edges are $+,-,+$; the $h_+$ distance between the bottom and the left decorated flags is $h_1$; the $h_-$ distance between the left and the right decorated flags is $s(h_3^{-1})(q-p)^{-\alpha_i^\vee}h_1$; the $h_+$ distance between the right and the top decorated flags is $s_i(h_1^{-1})q^{\alpha_i^\vee}h_2$. Thus $M^{\eta_2}$ is equal to 
\[
M^{\eta_2}=h_1^\mu\cdot (s(h_1^{-1})(q-p)^{-\alpha_i^\vee}h_3)^{s_i.\mu}\cdot (s_i(h_1^{-1})q^{\alpha_i^\vee}h_2)^{\mu}=\left(\frac{q}{q-p}\right)h_2^\mu.
\]
Thus $M^{\eta_1}+M^{\eta_2}=\left(\frac{p}{p-q}+\frac{q}{q-p}\right)h_2^\mu=h_2^\mu=M^\eta$.
\end{proof}

\section{Applications}\label{sec4}

\subsection{Cluster Theory} In the case where a cluster seed can be described by a weave (e.g., Demazure weaves \cite{CGGLSS}), we can use weighted cycles to give a pseudo-topological description of the cluster structure.

\begin{defn}\label{defn: X-variables} Let $\eta$ be the weighted cycle representative of a Y-cycle $\gamma$. We define the \emph{cluster $\cX$-variable} $X_\gamma:=M^\eta$.
\end{defn}

\begin{defn}\label{defn: A-variables} Let $r$ be the rank of $G$, let $\tau$ be the number of trivalent weave vertices, and let $\beta$ be the number of boundary base points. Let $m:=\tau+r(\beta-1)$. Let $\{\gamma_1,\dots,\gamma_m\}$ be a set of linearly independent Y-cycles. Let $\{\eta_1,\dots, \eta_m\}$ be a set of weighted cycles dual to $\{\gamma_1,\dots, \gamma_m\}$ in the sense that $\inprod{\eta_a}{\gamma_b}=\delta_{ab}$. We define the \emph{cluster $\cA$-variables} $A_1,\dots, A_m$ by setting $A_a:=M^{\eta_a}$.
\end{defn}

\begin{rmk} We will prove in the next section that in the case of Demazure weaves, Definitions \ref{defn: X-variables} and \ref{defn: A-variables} give the same cluster $\cX$- and $\cA$-variables as in \cite{CGGLSS}.
\end{rmk}

\begin{rmk} By Theorem \ref{thm: weighted cycle algebra is a torus}, the rank of the lattice of weighted cycles is precisely $m=\tau+r(\beta-1)$; thus, the linearly independent weighted cycles $\eta_1,\dots, \eta_m$ form a sublattice of finite index. If this sublattice is saturated, then we can express the weighted cycle representatives of Y-cycles as a linear combination of $\eta_1,\dots, \eta_m$, and thus giving a way to express any cluster $\cX$-variable $X_a=X_{\gamma_a}$ as a product of the cluster $\cA$-variables $A_1,\dots, A_m$, which must necessarily be $X_a=\prod_b A_b^{\epsilon_{ab}}$ due to the intersection pairing. This is equivalent to the cluster theoretical $p$-map\cite{FGensemble}. 
\end{rmk}

\begin{rmk} The mutations of cluster $\mathcal{A}$-variables and cluster $\mathcal{X}$-variables are compatible with the local mutation rule for weighted cycles described in Definition \ref{defn: mutation of weighted cycles}. For example, note that in Figure \ref{fig:merodromy mutation}, the intersection number between the weighted cycle and the weave edge in the left picture is $1$, and the intersection numbers in both of the right pictures are $1,-1,1$. This coincides with the mutation formula of cluster $\mathcal{A}$-variables:
\[
A'_c=\frac{\prod_{\epsilon_{ca}>0}A_a^{\epsilon_{ca}}+\prod_{\epsilon_{cb}<0}A_b^{-\epsilon_{cb}}}{A_c}
\]
A similar analysis reveals the compatibility with cluster $\mathcal{X}$-variables mutations as well.
\end{rmk}

\subsection{Demazure Weaves}

In \cite{CGGS,CGGLSS}, a special family of weaves called Demazure weaves were introduced, and each Demazure weave was equipped with a collection of special Y-cycles called Lusztig cycles. Let us first recall their definitions below.

\begin{defn} A \emph{Demazure weave} is a weave drawn on the rectangle $[0,1]\times [0,1]$ such that:
\begin{itemize}
    \item All external weave edges are incident to the top or the bottom boundary of the rectangle.
    \item No point along any weave edge admits a horizontal tangent line.
    \item Each trivalent weave vertex is incident to two weave edges above it and one below it.
    \item Each tetravalent weave vertex is incident to two weave edges above it and two below it.
    \item Each hexavalent weave vertex is incident to three weave edges above it and three below it.
\end{itemize}
We put one boundary base point at the lower left-hand corner of the rectangle and another one at the lower right-hand corner of the rectangle. Also we assume without loss of generality that all weave vertices in a Demazure weave are at different heights.
\end{defn}

Because of this definition, all edges in a Demazure weave can be equipped with a downward pointing orientation, and such a choice of orientations on weave edges is an example of the compatible orientations in the sense of Definition \ref{defn: compatible orientations}.

\begin{defn} Let $\ww$ be a Demazure weave and let $v$ be a trivalent weave vertex in $\ww$. The \emph{Lusztig cycle} associated with $v$ is constructed as follows as we scan $\ww$ from top to bottom.
\begin{itemize}
    \item Any weave edge $e$ that begins above $v$ is assigned with $\gamma(e)=0$.
    \item The unique weave edge $e$ that begins at $v$ is assigned with $\gamma(e)=1$.
    \item For any trivalent weave vertex $w$ below $v$, the assignments must satisfy $\gamma(c)=\min\{\gamma(a),\gamma(b)\}$, where $c$ is the unique weave edge that begins at $w$.
    \item For any tetravalent weave vertex $w$, the assignments must satisfy $\gamma(a)=\gamma(c)$ and $\gamma(b)=\gamma(d)$, where $a,b,c,d$ are weave edges incident to $w$ in a cyclic order, with $a,b$ above $w$ and $c,d$ below $w$.
    \item For any hexavalent weave vertex $w$, the assignments must satisfy $\gamma(d)=\gamma(a)+\gamma(b)-\min\{\gamma(a),\gamma(c)\}$, $\gamma(e)=\min\{\gamma(a),\gamma(c)\}$, and $\gamma(f)=\gamma(b)+\gamma(c)-\min\{\gamma(a),\gamma(c)\}$, where $a,b,c,d,e,f$ are weave edges incident to $w$ in a cyclic order, with $a,b,c$ above $w$ and $d,e,f$ below $w$.
\end{itemize}
\end{defn}

Recall that in \cite{CGGLSS}, each weave edge $e$ is equipped with a labeling $(z_e,u_e)$ such that the two decorated flags on the left and on the right of $e$ are $xN$ and $xB_i(z_e)u_e^{\alpha_i^\vee}N$ where $s_i$ is the color of $e$.

\begin{lem}\label{lem: merodromy across the edge below a trivalent} Suppose $\eta$ is a rightward-oriented weighted chain that only intersects the weave edge $e$ with color $s_i$ and labeling $(z_e,u_e)$ and suppose the intersection number is $n$. Then $M^\eta=u_e^n$.
\end{lem}
\begin{proof} Note that the intersection is positive. Suppose the chamber weight associated with the right side of $\eta$ is $\mu$. Then from the intersection number, we know that $\inprod{\mu}{\alpha_i^\vee}=n$. Note that $NB_i(z_e)u_e^{\alpha_i^\vee}N=N\overline{s}_iu_e^{\alpha_i^\vee}N$. Thus, 
\[
M^\eta=\left(h_+(xN,xB_i(z_e)u_e^{\alpha_i^\vee})\right)^{\mu}=u_e^{\inprod{\mu}{\alpha_i^\vee}}=u_e^n.\qedhere
\]
\end{proof}

Let us also recall from \cite[Theorem 5.12]{CGGLSS} that the cluster variables $\{A_a\}$ associated with the Demazure weave $\ww$ are uniquely determined by the condition
\begin{equation}\label{eq: equation defining cluster variables}
u_e=\prod_aA_a^{\gamma_a(e)},
\end{equation}
where $e$ is any weave edge.

\begin{thm}\label{thm: cluster variables in Demazure weaves} Let $\ww$ be a Demazure weave and let $\{\gamma_a\}$ be the collection of Lusztig cycles on $\ww$. Then there exist a collection of dual weighted relative cycles $\{\eta_a\}$ on $\ww$ such that $\inprod{\gamma_a}{\eta_b}=\delta_{ab}$ and $M^{\eta_a}=A_a$ for all $a$, where $A_a$ is the cluster variable associated with the Y-cycle $\gamma_a$.
\end{thm}
\begin{proof} Let $p_a$ be the trivalent weave vertex that defines a Y-cycle $\gamma_a$ and suppose the color of the weave edges incident to $p_a$ is $s_i$. We define the weighted relative cycle $\xi_a$ by drawing a short weighted chain going from left to right across the weave edge directly below $p_a$, with chamber weights changing from $s_i.\omega_i$ to $\omega_i$; we then extend this weighted chain horizontally at both ends until they reach the boundary of the disk. Note that by construction, $\inprod{\gamma_a}{\xi_a}=1$. Moreover, if we define $a<b$ to mean that the trivalent weave vertex $p_a$ is higher than $p_b$ in the vertical direction, then it is also clear that $\inprod{\gamma_a}{\xi_b}=0$ for all $b>a$. Now we can go from top to bottom along the weave and define
\[
\eta_a:=\xi_a-\sum_{b<a}\inprod{\gamma_b}{\xi_a}\eta_b
\]

We will do an induction to prove that the collection of weighted relative cycles $\{\eta_a\}$ satisfies the statement. For the base case with the smallest index $a$, we have $\eta_a=\xi_a$ and hence $\inprod{\gamma_b}{\eta_a}=\inprod{\gamma_b}{\xi_a}=\delta_{ba}$. Also, since $p_a$ is the highest trivalent weave vertex, all labelings at the same horizontal level as $p_a$ have $u_e=1$. Thus, by Lemma \ref{lem: merodromy across the edge below a trivalent} and Equation \eqref{eq: equation defining cluster variables}, we have $M^{\eta_a}=A_a$.

Now inductively, we observe that for a general $a$,
\[
\inprod{\gamma_b}{\eta_a}=\left\{\begin{array}{ll}
   \inprod{\gamma_b}{\xi_a}-\sum_{c<a}\inprod{\gamma_c}{\xi_a}\inprod{\gamma_b}{\eta_c}=\inprod{\gamma_b}{\xi_a}-\inprod{\gamma_b}{\xi_a}=0  & \text{if $b<a$}, \\
    \inprod{\gamma_a}{\xi_a}-\sum_{c<a}\inprod{\gamma_c}{\xi_a}{\gamma_a}{\eta_c} =1&\text{if $b=a$},\\
    -\sum_{c<a}\inprod{\gamma_c}{\xi_a}\inprod{\gamma_b}{\eta_c}=0 & \text{if $b>a$}.
\end{array}\right.
\]
As for the merodromy claim, it follows from Lemma \ref{lem: merodromy across the edge below a trivalent} that 
\[
M^{\xi_a}=\prod_{e\cap \xi_a\neq \emptyset} u_e^{\inprod{e}{\xi_a}}=A_a\prod_{e\cap \xi_a\neq \emptyset} \prod_{b<a}A_b^{\gamma_b(e)\inprod{e}{\xi_a}}=A_a\prod_{b<a}A_b^{\inprod{\gamma_b}{\xi_a}},
\]
which implies inductively that
\[
A_a=M^{\xi_a}\Big/\prod_{b<a}A_b^{\inprod{\gamma_b}{\xi_a}}=M^{\xi_a}\Big/\prod_{b<a}(M^{\eta_b})^{\inprod{\gamma_b}{\xi_a}}=M^{\xi_a-\sum_{b<a}\inprod{\gamma_b}{\xi_a}\eta_b}=M^{\eta_a}.\qedhere
\]
\end{proof}

\begin{rmk} The construction presented in the proof of Theorem \ref{thm: cluster variables in Demazure weaves} can be applied to a more general choice of weighted cycles $\{\xi_a\}$: instead of extending the short weighted chains horizontally, we can extend them upward, as long as we ensure that they always intersect the weave edges positively.
\end{rmk}

\subsection{Generalized Minors as Merodromies}

Let $G$ be a simply-connected semisimple Lie group. The Peter-Weyl theorem states that, with respect to the two-sided action of $G$, the coordinate ring of $G$ can be decomposed as
\[
\mathcal{O}(G)\cong\bigoplus_{\lambda\in P_+}  V(\lambda)\otimes V(\lambda)^*,
\]
where $P_+$ denotes the set of dominant weights and $V(\lambda)$ is the unique irreducible representation with highest weight $\lambda$. Correspondingly, $V(\lambda)^*$ has a unique lowest weight $-\lambda$. The Peter-Weyl isomorphism can be constructed as follows: with an element $v\otimes \xi\in V(\lambda)\otimes V(\lambda)^*$, we associate the function
\[
f_{v\otimes \xi}(g):=\inprod{\xi}{g.v}. 
\]

Suppose $w_1.\omega_i$ and $w_2.\omega_i$ are two chamber weights in the same Weyl group orbit. Fix a highest weight vector $v_{\omega_i}$ in $V(\omega_i)$ and a lowest weight vector $\xi_{\omega_i}$ in $V(\omega_i)^*$ such that $\inprod{\xi_{\omega_i}}{v_{\omega_i}}=1$ (note that $\xi_{\omega_i}$ should be of weight $-\omega_i$). For a chamber weight $\mu=w.\omega_i$, we define
\[
v_\mu:=\overline{w}.v_{\omega_i} \quad \text{and} \quad \xi_\mu:=\overline{w}.\xi_{\omega_i}.
\]

\begin{defn}[{\cite[Section 1.4]{FZ}}] The \emph{generalized minor} associated with the pair $(\mu_1,\mu_2)$ of chamber weights in the same Weyl group orbit is the following regular function on $G$:
\[
\Delta_{\mu_1,\mu_2}(g):=f_{\mu_2\otimes \mu_1}(g)=\inprod{\xi_{\mu_1}}{g.v_{\mu_2}}.
\]
\end{defn}

\begin{prop}\label{prop: generalized minors} Suppose $\eta$ is a weighted chain that passes through a collection of weave edges that form a reduced word for a Weyl group element $w$ and suppose all crossings are positive. Suppose the weights at the beginning and the end of $\eta$ are $w.\omega_i$ and $\omega_i$, respectively, and suppose the decorated flags at the beginning and the end of $\eta$ are $xN$ and $yN$, respectively. Then $M^\eta=\Delta_{w.\omega_i,\omega_i}(x^{-1}y)$. 
\end{prop}
\begin{proof} On the one hand, we can find $u_1\in N\cap \overline{w}N_-\overline{w}^{-1}$ and $u_2\in N$ such that $x^{-1}y=u_1\overline{w}hu_2$ and $M^\eta=h^{\omega_i}$. On the other hand,
\[
\Delta_{w.\omega_i,\omega_i}(x^{-1}y)=\inprod{\xi_{w.\omega_i}}{x^{-1}y.v_{\omega_i}}
=\inprod{\overline{w}.\xi_{\omega_i}}{u_1\overline{w}hu_2.v_{\omega_i}}
=\inprod{\overline{w}^{-1}u_1\overline{w}.\xi_{\omega_i}}{hu_2.v_{\omega_i}}.
\]
Since $\overline{w}^{-1}u_1\overline{w}\in N_-$ and $u_2\in N$, they fix $\xi_{\omega_i}$ and $v_{\omega_i}$, respectively. Thus, we can conclude that $\Delta_{w.\omega_i,\omega_i}(x^{-1}y)=\inprod{\xi_{\omega_i}}{h.v_{\omega_i}}=h^{\omega_i}$.
\end{proof}

Proposition \ref{prop: generalized minors} allows us to describe some cluster $\mathcal{A}$-variables on double Bruhat cells \cite{BFZ} using merodromies along weighted cycles. Below is an example for a double Bruhat cell in $\SL_3$.

\begin{exmp}\label{exmp: unipotent group} Consider the double Bruhat cell $B_-\cap Bw_0B$ in $\SL_3$. One of its cluster seeds can be described by the following weave, with two boundary base points at the two upper corners. We only draw the orientations on a few external weave edges; the orientations on the remaining weave edges can be uniquely determined by using the compatibility conditions. The yellow weave cycle corresponds to the unique mutable vertex. The merodromy along each of the five weighted cycles below is of the form
\[
\Delta_{w_0.\omega_i,\omega_i}(\doverline{u}^{-1}g)=\inprod{\overline{w}_0.\xi_{\omega_i}}{\doverline{u}^{-1}g.v_{\omega_i}}=\inprod{\overline{uw_0}.\xi_{\omega_i}}{g.v_{\omega_i}}=\Delta_{uw_0.\omega_i,\omega_i}(g).
\]
Thus, from left to right, the merodromies are $\Delta_{23,12}(g)$, $\Delta_{3,1}(g)$, $\Delta_{13,12}(g)$, $\Delta_{1,1}(g)$, and $\Delta_{12,12}(g)$, where $\Delta_{I,J}(g)$ denotes the determinant of the submatrix formed by the $I$th rows and the $J$th columns. These are precisely the cluster $\mathcal{A}$-variables of the seed.
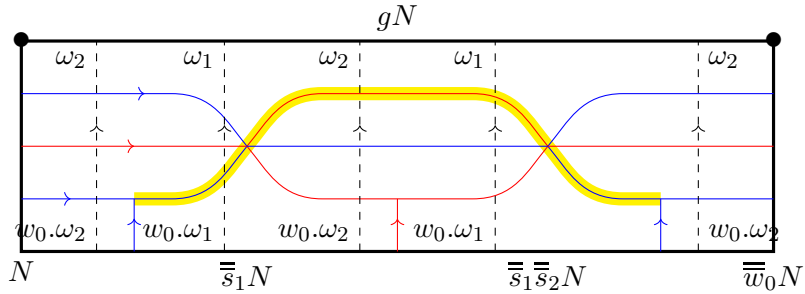
\begin{figure}[H]
    \centering
    \begin{tikzpicture}[yscale=0.7]
    \draw [very thick] (0,-1) rectangle (10,3);
    \node at (0,3) [] {\Large{$\bullet$}}; 
    \node at (10,3) [] {\Large{$\bullet$}};
    \draw [yellow, line width = 5pt] (1.5,0) -- (2,0) to [out=0,in=-120] (3,1) to [out=60,in=180] (4,2) -- (6,2) to [out=0,in=120] (7,1) to [out=-60,in=180] (8,0) -- (8.5,0);
        \draw [blue, decoration={markings, mark=at position 0.2 with {\arrow{>}}}, postaction={decorate}] (0,0) -- (2,0) to [out=0,in=-120] (3,1);
        \draw [blue, decoration={markings, mark=at position 0.5 with {\arrow{>}}}, postaction={decorate}] (0,2) -- (2,2) to [out=0,in=120] (3,1);
        \draw [red, decoration={markings, mark=at position 0.5 with {\arrow{>}}}, postaction={decorate}] (0,1) -- (3,1);
        \draw [red] (3,1) to [out=-60,in=180] (4,0) -- (6,0) to [out=0,in=-120] (7,1);
        \draw [red] (3,1) to [out=60,in=180] (4,2) -- (6,2) to [out=0,in=120] (7,1);
        \draw [blue] (3,1) -- (7,1);
        \draw [blue] (7,1) to [out=-60,in=180] (8,0) -- (10,0);
        \draw [blue] (7,1) to [out=60,in=180] (8,2) -- (10,2);
        \draw [red] (7,1) -- (10,1);
        \draw [blue, decoration={markings, mark=at position 0.5 with {\arrow{<}}}, postaction={decorate}] (1.5,0) -- (1.5,-1);
        \draw [red, decoration={markings, mark=at position 0.5 with {\arrow{<}}}, postaction={decorate}] (5,0) -- (5,-1);
        \draw [blue, decoration={markings, mark=at position 0.5 with {\arrow{<}}}, postaction={decorate}] (8.5,0) -- (8.5,-1);
        \node at (0,-1) [below] {$N$};
        \node at (3,-1) [below] {$\doverline{s}_1N$};
        \node at (7,-1) [below] {$\doverline{s}_1\doverline{s}_2N$};
        \node at (10,-1) [below] {$\doverline{w}_0N$};
        \node at (5,3) [above] {$gN$};
        \draw [dashed, decoration={markings, mark=at position 0.6 with {\arrow{>}}}, postaction={decorate}] (1,-1) node [above left] {$w_0.\omega_2$} -- (1,3) node [below left] {$\omega_2$};
        \draw [dashed, decoration={markings, mark=at position 0.6 with {\arrow{>}}}, postaction={decorate}] (2.7,-1) node [above left] {$w_0.\omega_1$} -- (2.7,3) node [below left] {$\omega_1$};
        \draw [dashed, decoration={markings, mark=at position 0.6 with {\arrow{>}}}, postaction={decorate}] (4.5,-1) node [above left] {$w_0.\omega_2$} -- (4.5,3) node [below left] {$\omega_2$};
        \draw [dashed, decoration={markings, mark=at position 0.6 with {\arrow{>}}}, postaction={decorate}] (6.3,-1) node [above left] {$w_0.\omega_1$} -- (6.3,3) node [below left] {$\omega_1$};
        \draw [dashed, decoration={markings, mark=at position 0.6 with {\arrow{>}}}, postaction={decorate}] (9,-1) node [above right] {$w_0.\omega_2$} -- (9,3) node [below right] {$\omega_2$};
    \end{tikzpicture}
    \caption{Merodromies as cluster variables for a double Bruhat cell.}
    \label{fig: Bruhat cell}
\end{figure}
\end{exmp}

\subsection{Cross-Ratios and Triple-Ratios as Merodromies}

It is well known that in many cases, certain cluster $\mathcal{X}$-variables can recover geometric invariants. For example, the cross-ratio between four distinct points in $\mathbb{P}^1$, and the triple-ratio among three generally positioned flags in a 3-dimensional vector space \cite{FGteich}. In this subsection, we will describe how to represent these geometric invariants using merodromies. 

\subsubsection{Cross-Ratios} The cluster structures on cyclic compactifications of $\mathcal{M}_{0,n}$ are captured by triangulations of $n$-gons with vertices labeled $1,2,\dots, n$. Given any triangulation of such a labeled $n$-gon, its dual graph is an $\mathrm{A}_1$-weave. We can give this $\mathrm{A}_1$-weave a compatible orientation by requiring the only outgoing external weave edge to be the edge between the vertices $1$ and $n$. Note that each diagonal in the triangulation and each boundary edge of the $n$-gon now cuts through exactly one weave line. 

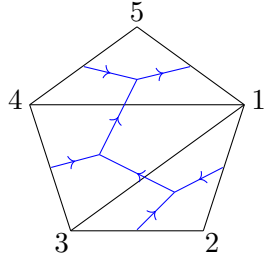
\begin{figure}[H]
    \centering
    \begin{tikzpicture}
    \foreach \i in {1,...,5}
    {
    \draw (-72*\i+90:1.5) -- (72+90-72*\i:1.5);
    \node at (90-72*\i:1.7) [] {$\i$};
    }
    \draw (18:1.5) -- (162:1.5);
    \draw (18:1.5) -- (234:1.5);
    \draw [blue, decoration={markings, mark=at position 0.5 with {\arrow{>}}}, postaction={decorate}] (-0.5,-0.2) -- (0,0.8);
    \draw [blue, decoration={markings, mark=at position 0.5 with {\arrow{>}}}, postaction={decorate}] (0,0.8) -- (54:1.2);
    \draw [blue, decoration={markings, mark=at position 0.5 with {\arrow{>}}}, postaction={decorate}] (126:1.2)-- (0,0.8);
    \draw [blue, decoration={markings, mark=at position 0.5 with {\arrow{>}}}, postaction={decorate}] (198:1.2) -- (-0.5,-0.2);
    \draw [blue, decoration={markings, mark=at position 0.5 with {\arrow{>}}}, postaction={decorate}] (0.5,-0.7) -- (-0.5,-0.2);
    \draw [blue, decoration={markings, mark=at position 0.5 with {\arrow{>}}}, postaction={decorate}] (-18:1.2) -- (0.5,-0.7);
    \draw [blue, decoration={markings, mark=at position 0.5 with {\arrow{>}}}, postaction={decorate}] (0,-1.2) -- (0.5,-0.7);
    \end{tikzpicture}
    \caption{A triangulation and its dual graph as a weave with a choice of compatible orientations.}
\end{figure}

Let us denote the diagonal/boundary edge connecting vertices $i<j$ by $\eta_{ij}$. We can then promote $\eta_{ij}$ to a weighted cycle by orienting it from $i$ to $j$ and labeling the source side with $-\omega_1$ and the target side with $\omega_1$. Then by construction, each $\eta_{ij}$ intersects a unique weave edge positively with an intersection number $1$.

Suppose we associate a line (a 1-dimensional linear subspace) $l_i\subset \mathbb{C}^2$ with the vertex $i$ such that any two lines connected by a diagonal in the triangulation or a boundary edge of the $n$-gon are transverse. Let $N$ be the maximal unipotent subgroup of upper unipotent triangular matrices in $\SL_2$. Then a decorated flag over $l_i$ can be represented by $x_iN$ for some $x_i\in \SL_2$. In particular, with this representative, the first column vector $v_i$ of $x_i$ is a non-zero vector in $l_i$. Note that for any $\SL_2$ matrix $\begin{pmatrix} a & b\\ c & d\end{pmatrix}$, we have $\begin{pmatrix} a & b\\ c & d\end{pmatrix}^{-1}=\begin{pmatrix} d & -b\\ -c & a\end{pmatrix}$. Thus,
\[
x_i^{-1}x_j=\begin{pmatrix} d_ia_j-b_ic_j & d_ib_j-b_id_j \\ -c_ia_j+a_ic_j & -c_ib_j+a_id_j\end{pmatrix}
\]
and hence by Proposition \ref{prop: generalized minors},
\[
M_{ij}:=M^{\eta_{ij}}=\Delta_{2,1}(x_i^{-1}x_j)=-c_ia_j+a_ic_j=\det(v_i\wedge v_j).
\]
Note that these are precisely the Pl\"{u}cker coordinates in the corresponding cluster seed in $\Gr_{2,n}$.

Meanwhile, each internal weave edge is a short I-cycle. Suppose the four adjacent faces of a short I-cycle $\gamma$ are associated with lines $l_i,l_j, l_k$, and $l_l$ for $i<j<k<l$, as depicted in Figure \ref{fig: crossratio}.
\begin{figure}[H]
    \centering
    \begin{tikzpicture}[baseline=0]
    \draw[line width=5pt, yellow] (-0.5,0) -- (0.5,0);
        \draw [blue, decoration={markings, mark=at position 0.5 with {\arrow{>}}}, postaction={decorate}] (-1.5,-1.5) -- (-0.5,0);
         \draw [blue, decoration={markings, mark=at position 0.5 with {\arrow{>}}}, postaction={decorate}] (-1.5,1.5) -- (-0.5,0);
        \draw [blue, decoration={markings, mark=at position 0.5 with {\arrow{>}}}, postaction={decorate}] (1.5,-1.5) -- (0.5,0);
        \draw  [blue, decoration={markings, mark=at position 0.5 with {\arrow{>}}}, postaction={decorate}] (0.5,0) -- (1.5,1.5);
        \draw [blue, decoration={markings, mark=at position 0.5 with {\arrow{>}}}, postaction={decorate}] (-0.5,0) -- (0.5,0);
        \node at (0,-1.5) [] {$l_j$};
        \node at (0,1.5) [] {$l_l$};
        \node at (-2.5,0) [] {$l_k$};
        \node at (2.5,0) [] {$l_i$};
        \draw [dashed, ->, decoration={markings, mark=at position 0.5 with {\arrow{>}}}, postaction={decorate}] (0.5,-0.5) -- (-0.5,-0.5) arc (270:90:0.5) -- (0.5,0.5) arc (90:-90:0.5);
        \node at (1,0) [right] {$-\omega_1$};
        \node at (0,-0.5) [below] {$\omega_1$};
        \node at (-1,0) [left] {$-\omega_1$};
        \node at (0,0.5) [above] {$\omega_1$};
    \end{tikzpicture}
    \caption{Local picture of a Y-cycle.}
    \label{fig: crossratio}
\end{figure}
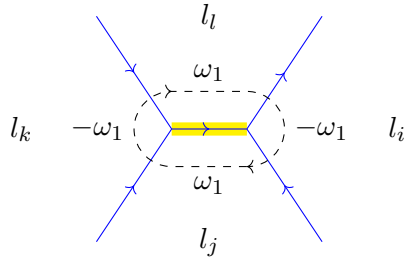
Let us fix decoration $v_i, v_j, v_k$, and $v_l$ above the four lines. Then on the one hand, the merodromy from the right face to the bottom face and that from the left face to the top face are $M_{ij}=\det(v_i\wedge v_j)$ and $M_{kl}=\det(v_k\wedge v_l)$, respectively, and the merodromy from the bottom face to the left face is
\[
(-\omega_1)(h_+(x_jN,x_kN)=(\omega_1(h_+(x_jN,x_kN))^{-1}=M_{jk}^{-1}=\det(v_j\wedge v_k)^{-1}.
\]
On the other hand, note that the intersection between the weighted cycle and the northeast weave edge is negative. By using homotopy move (1) (see Figure \ref{fig: homotopies}), the merodromy from the top face to the right face is equal to the merodromy of a weighted chain starting with $\omega_1$ on the right face, going across the northeast weave edge, and ending with $-\omega_i$ on the top face. Note that this is analogous to the merodromy from the bottom face to the left face and therefore we can conclude that this last piece of contribution is $M_{il}^{-1}=\det(v_i\wedge v_l)^{-1}$. Therefore in total, we have
\[
X_\gamma=\frac{M_{ij}M_{kl}}{M_{jk}M_{il}},
\]
which is precisely the cross-ratio of $l_i$, $l_j$, $l_k$, and $l_l$.

\subsubsection{Triple-Ratios} The cluster structure on the configuration of three generally positioned flags in a 3-dimensional vector space can be captured by the following plabic graph \cite{Gon}. By performing the T-shift construction \cite{CLSBW}, we obtain the following $\mathrm{A}_2$-weave in the middle. The highlighted part is the Y-cycle whose associated cluster $\cX$-variable is the triple-ratio of the three flags. The right picture is the weighted cycle representative for this Y-cycle; we also orient the weave edges in a compatible fashion in the right picture.
\begin{figure}[H]
    \centering
    \begin{tikzpicture}
    \foreach \i in {0,1,2}
    {
    \draw [very thick] (90+120*\i:2) -- (-30+120*\i:2);
    \draw (90+120*\i:2) -- (90+120*\i:1);
    \draw (90+120*\i:1) -- (120*\i+30:1);
    \draw (90+120*\i:1) -- (120*\i+150:1);
    }
    \foreach \i in {0,1,2}
    {
    \draw [fill=white] (90+120*\i:2) circle [radius=0.2];
    \draw [fill=white] (30+120*\i:1) circle [radius=0.2];
    \draw [fill=black] (90+120*\i:1) circle [radius=0.2];
    }
    \end{tikzpicture}\hspace{2cm}
    \begin{tikzpicture}
    \foreach \i in {0,1,2}
    {
    \draw [very thick] (90+120*\i:2) -- (-30+120*\i:2);
    \draw [yellow, line width = 5pt] (0,0) -- (90+120*\i:1);
    \draw [blue] (30+120*\i:1) -- (0,0);
    \draw [red] (90+120*\i:1) -- (0,0);
    \draw [red] (110+120*\i:1.34) -- (90+120*\i:1) -- (70+120*\i:1.34);
    }
    \end{tikzpicture}\hspace{2cm}
    \begin{tikzpicture}
    \draw [red, decoration={markings, mark=at position 0.5 with {\arrow{>}}}, postaction={decorate}] (120:1.732)-- (0,1);
    \draw [red, decoration={markings, mark=at position 0.5 with {\arrow{>}}}, postaction={decorate}] (60:1.732) -- (0,1);
    \draw [red, decoration={markings, mark=at position 0.5 with {\arrow{>}}}, postaction={decorate}] (0,1) -- (0,0);
    \draw [blue, decoration={markings, mark=at position 0.5 with {\arrow{>}}}, postaction={decorate}] (150:1.5) -- (0,0);
    \draw [blue, decoration={markings, mark=at position 0.5 with {\arrow{>}}}, postaction={decorate}] (30:1.5) -- (0,0);
    \draw [blue, decoration={markings, mark=at position 0.5 with {\arrow{>}}}, postaction={decorate}] (0,0) -- (0,-1.5);
    \draw [red, decoration={markings, mark=at position 0.5 with {\arrow{>}}}, postaction={decorate}]
    (0,0) -- (-150:1);
    \draw [red, decoration={markings, mark=at position 0.5 with {\arrow{>}}}, postaction={decorate}] (0,0) -- (-30:1);
    \draw [red, decoration={markings, mark=at position 0.5 with {\arrow{>}}}, postaction={decorate}]
    (-1.732,0) -- (-150:1);
    \draw [red, decoration={markings, mark=at position 0.5 with {\arrow{>}}}, postaction={decorate}] (1.732,0) -- (-30:1);
    \draw [red, decoration={markings, mark=at position 0.5 with {\arrow{>}}}, postaction={decorate}] (-150:1) -- (-120:1.732);
    \draw [red, decoration={markings, mark=at position 0.5 with {\arrow{>}}}, postaction={decorate}] (-30:1) -- (-60:1.732);
    \foreach \i in {0,1,2}
    {
    \draw [dashed, decoration={markings, mark=at position 0.5 with {\arrow{>}}}, postaction={decorate}] (90+120*\i:1.4) arc (90+120*\i:120*\i:0.4)--(30+120*\i:0.4); 
    \draw [dashed, decoration={markings, mark=at position 0.5 with {\arrow{<}}}, postaction={decorate}] (90+120*\i:1.4) arc (90+120*\i:180+120*\i:0.4)--(150+120*\i:0.4); 
    \node at (90+120*\i:1.7) [] {$\omega_1-\omega_2$};
    }
    \end{tikzpicture}
    \caption{Ideal web and T-shifted weave for the configuration space of three flags.}
\end{figure}
By applying the homotopy moves (1), (8), (9), and (11), we can turn the weighted cycle in the right picture above into the following. Note that all crossings are positive.
\begin{figure}[H]
    \centering
    \begin{tikzpicture}
    \draw [red, decoration={markings, mark=at position 0.5 with {\arrow{>}}}, postaction={decorate}] (120:1.732)-- (0,1);
    \draw [red, decoration={markings, mark=at position 0.5 with {\arrow{>}}}, postaction={decorate}] (60:1.732) -- (0,1);
    \draw [red, decoration={markings, mark=at position 0.5 with {\arrow{>}}}, postaction={decorate}] (0,1) -- (0,0);
    \draw [blue, decoration={markings, mark=at position 0.5 with {\arrow{>}}}, postaction={decorate}] (150:1.5) -- (0,0);
    \draw [blue, decoration={markings, mark=at position 0.5 with {\arrow{>}}}, postaction={decorate}] (30:1.5) -- (0,0);
    \draw [blue, decoration={markings, mark=at position 0.5 with {\arrow{>}}}, postaction={decorate}] (0,0) -- (0,-1.5);
    \draw [red, decoration={markings, mark=at position 0.5 with {\arrow{>}}}, postaction={decorate}]
    (0,0) -- (-150:1);
    \draw [red, decoration={markings, mark=at position 0.5 with {\arrow{>}}}, postaction={decorate}] (0,0) -- (-30:1);
    \draw [red, decoration={markings, mark=at position 0.5 with {\arrow{>}}}, postaction={decorate}]
    (-1.732,0) -- (-150:1);
    \draw [red, decoration={markings, mark=at position 0.5 with {\arrow{>}}}, postaction={decorate}] (1.732,0) -- (-30:1);
    \draw [red, decoration={markings, mark=at position 0.5 with {\arrow{>}}}, postaction={decorate}] (-150:1) -- (-120:1.732);
    \draw [red, decoration={markings, mark=at position 0.5 with {\arrow{>}}}, postaction={decorate}] (-30:1) -- (-60:1.732);
    \draw [dashed, decoration={markings, mark=at position 0.6 with {\arrow{>}}}, postaction={decorate}] (80:2) node [above right] {$-\omega_2$} -- (-20:2) node [right] {$\omega_1$};
    \draw [dashed, decoration={markings, mark=at position 0.6 with {\arrow{>}}}, postaction={decorate}] (75:1.5) node [above] {$\omega_1$} -- (-15:1.5)node [below] {$-\omega_2$};
    \draw [dashed, decoration={markings, mark=at position 0.6 with {\arrow{<}}}, postaction={decorate}] (100:2) node [above left] {$\omega_1$} -- (-160:2) node [left] {$-\omega_2$};
    \draw [dashed, decoration={markings, mark=at position 0.6 with {\arrow{<}}}, postaction={decorate}] (105:1.5) node [above] {$-\omega_2$} -- (-165:1.5)node [below] {$\omega_1$};
    \draw [dashed, decoration={markings, mark=at position 0.6 with {\arrow{>}}}, postaction={decorate}] (-1.5,-1.1) node [left] {$\omega_2$} -- (1.5,-1.1) node [right] {$-\omega_1$};
    \draw [dashed, decoration={markings, mark=at position 0.6 with {\arrow{>}}}, postaction={decorate}] (-1.5,-1.4) node [left] {$-\omega_1$} -- (1.5,-1.4) node [right] {$\omega_2$};
    \node at (-30:3) [] {$zN$};
    \node at (-150:3) [] {$xN$};
    \node at (90:2.5) [] {$yN$};
    \end{tikzpicture}
    \caption{Breaking down the weighted cycle into six weighted chains.}
\end{figure}

Next we need to compute the merodromy along each weighted chain. To do that, we make use of Proposition \ref{prop: generalized minors} again. Suppose we have generally positioned decorated flags $xN$ and $yN$ with $x,y\in \SL_3$, then the merodromy of the weighted chain $\begin{tikzpicture}[baseline=5]
\node at (-1,0.5) [] {$xN$};
\draw [red, ->] (0,1) -- (0,0);
\draw [blue,->] (1,1) -- (1,0);
\draw [red,->] (2,1) -- (2,0);
\node at (3,0.5) [ ] {$yN$};
\draw [dashed, decoration={markings, mark=at position 0.6 with {\arrow{>}}}, postaction={decorate}] (-0.5,0.5) node [above] {$-\omega_2$} -- (2.5,0.5) node [above] {$\omega_1$};
\end{tikzpicture}$ is
\[
\Delta_{3,1}(x^{-1}y)=y_{11}\det\begin{pmatrix}x_{21} & x_{22} \\ x_{31} & x_{32}\end{pmatrix}-y_{21}\det\begin{pmatrix}x_{11} & x_{12} \\ x_{31} & x_{32} \end{pmatrix}+y_{31}\det\begin{pmatrix} x_{11} & x_{12} \\ x_{21} & x_{22}
\end{pmatrix}=\det(y_1\wedge x_1 \wedge x_2),
\]
where $x_i$ and $y_j$ are the $i$th and $j$th column vectors of $x$ and $y$, respectively. Similarly, the merodromy of the weighted chain $\begin{tikzpicture}[baseline=5]
\node at (-1,0.5) [] {$xN$};
\draw [red, ->] (0,1) -- (0,0);
\draw [blue,->] (1,1) -- (1,0);
\draw [red,->] (2,1) -- (2,0);
\node at (3,0.5) [ ] {$yN$};
\draw [dashed, decoration={markings, mark=at position 0.6 with {\arrow{>}}}, postaction={decorate}] (-0.5,0.5) node [above] {$-\omega_1$} -- (2.5,0.5) node [above] {$\omega_2$};
\end{tikzpicture}$ is
\[
\Delta_{23,12}(x^{-1}y)=\Delta_{3,1}(y^{-1}x)=\det(x_1\wedge y_1\wedge y_2).
\]
Thus, the total merodromy of the weighted cycle representative of the Y-cycle is
\[
X_\gamma=\frac{\det(y_1\wedge x_1\wedge x_2)\det(z_1\wedge y_1\wedge y_2)\det(x_1\wedge z_1\wedge z_2)}{\det(x_1\wedge y_1\wedge y_2)\det (y_1\wedge z_1\wedge z_2)\det (z_1\wedge x_1\wedge x_2)},
\]
which is precisely the triple-ratio of the triple of flags $(xN,yN,zN)$.

\subsection{Type A Weaves} Recall that a weave $\ww$ of Dynkin type $\mathrm{A}_n$ in fact describes a Legendrian surface $\Lambda_\ww$ \cite{CZ}. The Legendrian surface $\Lambda_\ww$ is a ramified $n+1$-fold cover of the disk, and the weave $\ww$ is precisely the projection of the singular locus of the front projection of $\Lambda_\ww$. 

In the case of a type $\mathrm{A}_n$ weave $\ww$, the weighted cycles on $\ww$ can be realized as actual relative cycle in $H_1(\Lambda_\ww,\partial \Lambda_\ww-B; \mathbb{Q})$, where $B$ is the set of boundary base points. To see this, note that each weight $\mu$ in the weight lattice of $\mathrm{A}_n$ is a vector $(\mu_1,\dots, \mu_{n+1})$ in a $(n+1)$-dimensional Euclidean space, all of whose entries are in $\frac{1}{n+1}\mathbb{Z}$. Thus, we can lift a weighted chain $\eta$ with weights $\mu$ to a relative chain $\tilde{\eta}$ on $\Lambda_\ww$ whose support is the preimage of $\eta$, and whose signed multiplicity on level $i$ is precisely $\mu_i$. It is not hard to verify that the intersection pairing between weighted cycles in the Dynkin type A case recovers the intersection pairing between relative $1$-cycles on the Legendrian surface $\Lambda_\ww$. Also, the merodromies of weighted cycles recover the merodromies along relative cycles introduced in \cite{CW} in Dynkin type A.

\begin{exmp} Figure \ref{fig: fundamental weights in A_2} is a cross-sectional depiction of the realization of a weighted chain on an $\mathrm{A}_2$ weave $\ww$ as an actual relative cycle on $\Lambda_\ww$.
\end{exmp}
\begin{figure}[H]
    \centering
    \begin{tikzpicture}[baseline=20]
    \draw [red] (-0.3,1) -- (-0.3,0) node [below] {$s_2$};
    \draw [blue] (1.3,1) -- (1.3,0) node [below] {$s_1$};
    \draw [dashed, decoration={markings, mark=at position 0.5 with {\arrow{>}}}, postaction={decorate}] (-1,0.5) node [above] {$\omega_2$} -- node [above] {$\omega_1-\omega_2$}(2,0.5) node [above] {$-\omega_1$};
    \end{tikzpicture} \hspace{2cm}
    \begin{tikzpicture}[baseline=0, yscale=0.7]
        \draw [decoration={markings, mark=at position 0.5 with {\arrow{>}}}, postaction={decorate}] (0,-1) node [left] {$\frac{1}{3}$} -- (2,-1) to [out=0,in=180] (4,0);
        \draw [decoration={markings, mark=at position 0.5 with {\arrow{>}}}, postaction={decorate}] (0,0) node [left] {$\frac{1}{3}$} to [out=0,in=180] (2,1) -- (4,1);
        \draw [decoration={markings, mark=at position 0.5 with {\arrow{<}}}, postaction={decorate}] (0,1) node [left] {$\frac{2}{3}$} to [out=0,in=180] (2,0) to [out=0,in=180] (4,-1);
        \node [red] at (1,0.5) [] {$\bullet$};
        \node [blue] at (3,-0.5) [] {$\bullet$};
    \end{tikzpicture}
    \caption{Lifting a weighted cycle into a relative cycle}
    \label{fig: fundamental weights in A_2}
\end{figure}
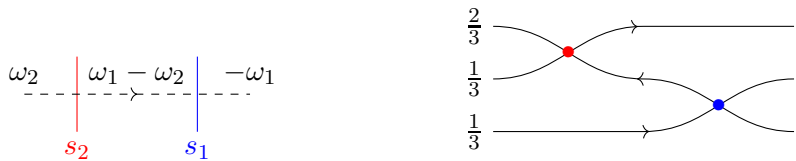

\subsection{Quantum Group \texorpdfstring{$U_q(\mathfrak{sl}_2)$}{}}\label{subsec: quantum group} Through a series of work \cite{SS,IP,GS3,Shen}, it is known that quantum groups are closely related to cluster algebras. In this subsection, we would like to take the quantum Drinfeld double $D_q(\mathfrak{sl}_2)$ as an example and exhibit its generators as weighted cycles on a weave $\ww$ and show that the relations among the generators can be recovered from the intersection pairings between weighted cycles, and thus mapping $D_q(\mathfrak{sl}_2)$ into the quantum weighted cycle algebra $\mathbb{W}(\ww)$. 

Although the base surface here is no longer a disk but a punctured disk, all previous constructions of weighted cycles on a disk can be easily generalized to a punctured disk. We also expect that this mapping of Drinfeld doubles into quantum weighted cycle algebra can be generalized to all Dynkin types.

Recall that the quantum Drinfeld double has four generators $E$, $F$, $K$, and $K'$, and they satisfy the following relations:
\[
KE=q^2EK, \quad K'E=q^{-2}EK', \quad KF=q^{-2}FK, \quad K'F=q^2FK', \quad [E,F]=(q-q^{-1})(K'-K).
\]

We map these four generators to the following weighted cycles, which we denote by $\eta_E$, $\eta_F$, $\eta_K$, and $\eta_{K'}$, respectively. Note that each of the top weaves differ from the weave in the bottom row by a single mutation, and each of the weighted cycles $\eta_E$ and $\eta_F$ can be mutated into a sum of two weighted cycles on the weave at the bottom. We abuse notation and still denote the sums that are mutated weighted cycles by $\eta_E$ and $\eta_F$, respectively. We also choose the weave $\ww$ at the bottom row to define our quantum weighted cycle algebra $\mathbb{W}(\ww)$.
\begin{figure}[H]
$E\longmapsto$\begin{tikzpicture}[baseline=0]
\draw [very thick] (0,0)  circle [radius=1.5];
\draw [very thick] (0,0)  circle [radius=0.1];
\node at (-170:1.5) [] {\Large{$\bullet$}};
\node at (10:1.5) [] {\Large{$\bullet$}};
\draw [blue] (0,0) circle [radius=0.5];
\draw [blue] (-1.5,0) to [out=0,in=180] (0,-1) to [out=0,in=180] (1.5,0);
\draw [blue] (0,-0.5) -- (0,-1);
\draw [dashed, decoration={markings, mark=at position 0.6 with {\arrow{>}}}, postaction={decorate}] (-102:1.5) -- node [above left] {\footnotesize{$\omega_1$}} (-0.3,-1) arc (180:0:0.3) --  node [above right] {\footnotesize{$-\omega_1$}} (-78:1.5);
\node at (0,-1.5) [below] {$\eta_E$};
\end{tikzpicture}, \hspace{1cm} $F\longmapsto$\begin{tikzpicture}[baseline=0]
\draw [very thick] (0,0)  circle [radius=1.5];
\draw [very thick] (0,0)  circle [radius=0.1];
\node at (-170:1.5) [] {\Large{$\bullet$}};
\node at (10:1.5) [] {\Large{$\bullet$}};
\draw [blue] (0,0) circle [radius=0.5];
\draw [blue] (-1.5,0) to [out=0,in=180] (0,1) to [out=0,in=180] (1.5,0);
\draw [blue] (0,0.5) -- (0,1);
\draw [dashed, decoration={markings, mark=at position 0.6 with {\arrow{>}}}, postaction={decorate}] (78:1.5) -- node [below right] {\footnotesize{$\omega_1$}} (0.3,1) arc (0:-180:0.3) --  node [below left] {\footnotesize{$-\omega_1$}} (102:1.5);
\node at (0,1.5)[above] {$\eta_F$};
\end{tikzpicture},\\\vspace{0.3cm}
$K\longmapsto$  \begin{tikzpicture}[baseline=0]
\draw [very thick] (0,0)  circle [radius=1.5];
\draw [very thick] (0,0)  circle [radius=0.1];
\node at (-170:1.5) [] {\Large{$\bullet$}};
\node at (10:1.5) [] {\Large{$\bullet$}};
\draw [blue] (0,0) circle [radius=0.5];
\draw [blue] (-1.5,0) -- (-0.5,0);
\draw [blue] (1.5,0) -- (0.5,0);
\draw [dashed, decoration={markings, mark=at position 0.6 with {\arrow{>}}}, postaction={decorate}] (-135:1.5) node [above right] {\footnotesize{$\alpha_1$}} -- (135:1.5) node [below right] {\footnotesize{$-\alpha_1$}};
\node at (-1.1,0) [above right] {$\eta_K$};
\end{tikzpicture}, \hspace{1cm}
$K'\longmapsto$ \begin{tikzpicture}[baseline=0]
\draw [very thick] (0,0)  circle [radius=1.5];
\draw [very thick] (0,0)  circle [radius=0.1];
\node at (-170:1.5) [] {\Large{$\bullet$}};
\node at (10:1.5) [] {\Large{$\bullet$}};
\draw [blue] (0,0) circle [radius=0.5];
\draw [blue] (-1.5,0) -- (-0.5,0);
\draw [blue] (1.5,0) -- (0.5,0);
\draw [dashed, decoration={markings, mark=at position 0.6 with {\arrow{>}}}, postaction={decorate}] (45:1.5) node [below left] {\footnotesize{$\alpha_1$}} -- (-45:1.5) node [above left] {\footnotesize{$-\alpha_1$}};
\node at (1.1,0) [above left] {$\eta_{K'}$};
\end{tikzpicture}.
\caption{Mapping of generators of the quantum Drinfeld double.}
\end{figure}
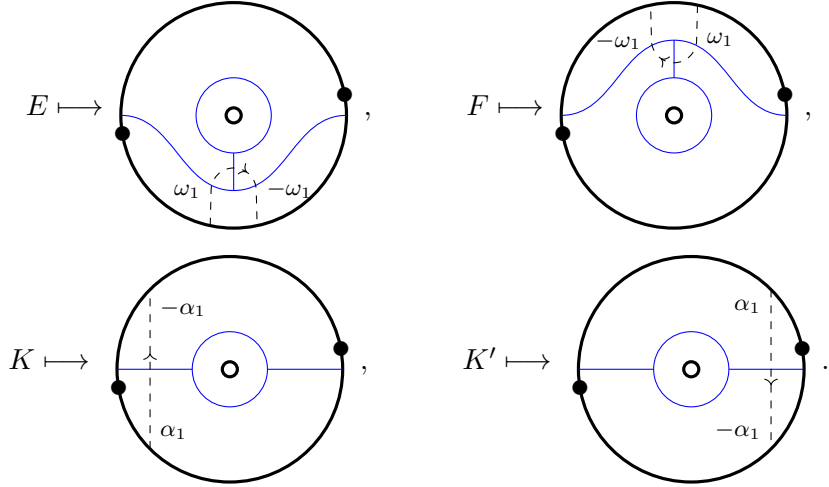

We observe that since $\eta_K$ can be homotoped into a small neighborhood near the boundary base point the left, the only contribution to the intersection pairing $\{\eta_K,\mu(\eta_E)\}$ comes from the lower boundary interval; through computation, one can find that $\{\eta_K,\mu(\eta_E)\}=1$ and therefore 
\[
\eta_K\eta_E=q^{2\{\eta_K,\eta_E\}}\eta_E\eta_K=q^2\eta_E\eta_K.
\]
Similar computations prove all remaining relations except the very last one.

For the last relation, let us draw both $\eta_E$ and $\eta_F$ on $\ww$. Note that the first term in $\eta_E$ commutes with the second term in $\eta_F$, and the second term in $\eta_E$ commutes with the first term in $\eta_F$. Thus, the commutator $[E,F]$ comes only from commuting the first terms between $\eta_E$ and $\eta_F$ and commuting the second terms between $\eta_E$ and $\eta_F$.
\begin{figure}[H]
    \centering
    $\eta_E=$\quad \begin{tikzpicture}[baseline=-5]
    \draw [blue] (0,0) circle [radius=1];
    \draw [blue] (-2,0) -- (-1,0) ;
    \draw [blue] (1,0) -- (2,0);
    \draw [very thick] (0,0) circle [radius=0.1];
    \draw [dashed, decoration={markings, mark=at position 0.6 with {\arrow{>}}}, postaction={decorate}] (0.5,-1.3) node [below] {\footnotesize{$\omega_1$}} -- (0.5,0) node [above left] {\footnotesize{$-\omega_1$}} arc (180:0:0.5) node [above right] {\footnotesize{$\omega_1$}} -- (1.5,-1.3) node [below] {\footnotesize{$-\omega_1$}};
    \end{tikzpicture}\quad $+$ \quad \begin{tikzpicture}[baseline=-5]
    \draw [blue] (0,0) circle [radius=1];
    \draw [blue] (-2,0) -- (-1,0) ;
    \draw [blue] (1,0) -- (2,0);
    \draw [very thick] (0,0) circle [radius=0.1];
    \draw [dashed, decoration={markings, mark=at position 0.6 with {\arrow{>}}}, postaction={decorate}] (-1.5,-1.3) node [below] {\footnotesize{$\omega_1$}} -- (-1.5,0) node [above left] {\footnotesize{$-\omega_1$}} arc (180:0:0.5) node [above right] {\footnotesize{$\omega_1$}} -- (-0.5,-1.3) node [below] {\footnotesize{$-\omega_1$}};
    \end{tikzpicture}, \\
    $\eta_F=$\quad \begin{tikzpicture}[baseline=-5]
    \draw [blue] (0,0) circle [radius=1];
    \draw [blue] (-2,0) -- (-1,0) ;
    \draw [blue] (1,0) -- (2,0);
    \draw [very thick] (0,0) circle [radius=0.1];
    \draw [dashed, decoration={markings, mark=at position 0.6 with {\arrow{>}}}, postaction={decorate}] (1.5,1.3) node [above] {\footnotesize{$\omega_1$}} -- (1.5,0) node [below right] {\footnotesize{$-\omega_1$}} arc (0:-180:0.5) node [below left] {\footnotesize{$\omega_1$}} -- (0.5,1.3) node [above] {\footnotesize{$-\omega_1$}};
    \end{tikzpicture}\quad $+$ \quad \begin{tikzpicture}[baseline=-5]
    \draw [blue] (0,0) circle [radius=1];
    \draw [blue] (-2,0) -- (-1,0) ;
    \draw [blue] (1,0) -- (2,0);
    \draw [very thick] (0,0) circle [radius=0.1];
    \draw [dashed, decoration={markings, mark=at position 0.6 with {\arrow{>}}}, postaction={decorate}] (-0.5,1.3) node [above] {\footnotesize{$\omega_1$}} -- (-0.5,0) node [below right] {\footnotesize{$-\omega_1$}} arc (0:-180:0.5) node [below left] {\footnotesize{$\omega_1$}} -- (-1.5,1.3) node [above] {\footnotesize{$-\omega_1$}};
    \end{tikzpicture}, 
    \caption{Mutated weighted cycles $\eta_E$ and $\eta_F$ drawn on $\ww$.}
\end{figure}
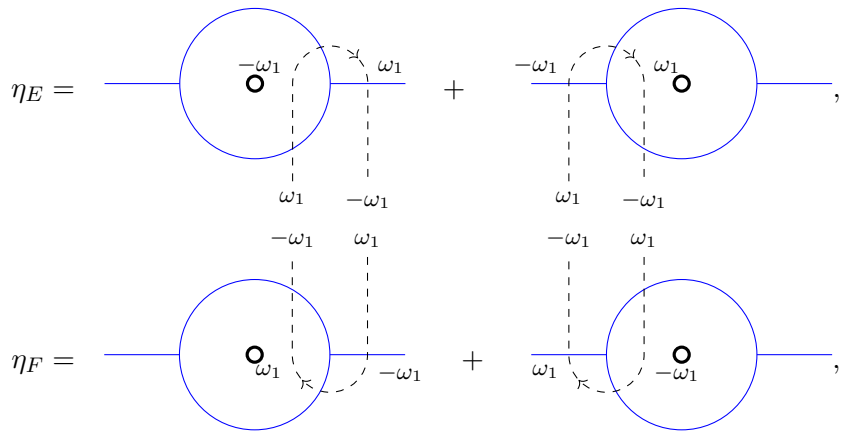
Let us multiply the first term in $\eta_E$ and the first term in $\eta_F$. There are two intersection points (Figure \ref{fig: intersection point between E and F}), yielding an intersectino number of $(\omega_1,\omega_1)-(-\omega_1,\omega_1)=\frac{1}{2}+\frac{1}{2}=1$. By applying the homotopy moves (9), (10), and (11), we can turn this classical product into $\eta_{K'}$. This shows that the commutator of the first terms is $(q-q^{-1})K'$. By a similar computation, one can show that the commutator of the second terms is $(q^{-1}-q)K$. Combining these two commutators, we get precisely the remaining relation $[E,F]=(q-q^{-1})(K'-K)$.
\begin{figure}[H]
    \centering$q\cdot$
    \begin{tikzpicture}[baseline=-5]
    \draw [blue] (0,0) circle [radius=1];
    \draw [blue] (-2,0) -- (-1,0) ;
    \draw [blue] (1,0) -- (2,0);
    \draw [very thick] (0,0) circle [radius=0.1];
    \draw [dashed, decoration={markings, mark=at position 0.6 with {\arrow{>}}}, postaction={decorate}] (0.5,-1.3) node [below] {\footnotesize{$\omega_1$}} -- (0.5,0) node [above left] {\footnotesize{$-\omega_1$}} arc (180:0:0.5) node [above right] {\footnotesize{$\omega_1$}} -- (1.5,-1.3) node [below] {\footnotesize{$-\omega_1$}};
    \draw [dashed, decoration={markings, mark=at position 0.6 with {\arrow{>}}}, postaction={decorate}] (1.5,1.3) node [above] {\footnotesize{$\omega_1$}} -- (1.5,0) node [below right] {\footnotesize{$-\omega_1$}} arc (0:-180:0.5) node [below left] {\footnotesize{$\omega_1$}} -- (0.5,1.3) node [above] {\footnotesize{$-\omega_1$}};
    \node at (0.5,0) [] {$\bullet$};
    \node at (1.5,0.1) [] {$\bullet$};
    \end{tikzpicture}  $=q\cdot$ 
    \begin{tikzpicture}[baseline=-5]
    \draw [blue] (0,0) circle [radius=1];
    \draw [blue] (-2,0) -- (-1,0) ;
    \draw [blue] (1,0) -- (2,0);
    \draw [very thick] (0,0) circle [radius=0.1];
    \draw [dashed, decoration={markings, mark=at position 0.6 with {\arrow{>}}}, postaction={decorate}] (0.5,-1.3) node [below] {\footnotesize{$\omega_1$}} to [in=-90] (1.2,0) to [out=90] (0.5,1.3) node [above] {\footnotesize{$-\omega_1$}};
    \draw [dashed, decoration={markings, mark=at position 0.4 with {\arrow{>}}}, postaction={decorate}] (1.5,1.3) node [above] {\footnotesize{$\omega_1$}} -- (1.5,-1.3) node [below] {\footnotesize{$-\omega_1$}};
    \end{tikzpicture} $=q\cdot$
    \begin{tikzpicture}[baseline=-5]
    \draw [blue] (0,0) circle [radius=1];
    \draw [blue] (-2,0) -- (-1,0) ;
    \draw [blue] (1,0) -- (2,0);
    \draw [very thick] (0,0) circle [radius=0.1];
    \draw [dashed, decoration={markings, mark=at position 0.4 with {\arrow{>}}}, postaction={decorate}] (1.5,1.3) node [above] {\footnotesize{$\alpha_1$}} -- (1.5,-1.3) node [below] {\footnotesize{$-\alpha_1$}};
    \end{tikzpicture}
    \caption{Product of the first terms in $\eta_E$ and $\eta_F$.}
    \label{fig: intersection point between E and F}
\end{figure}
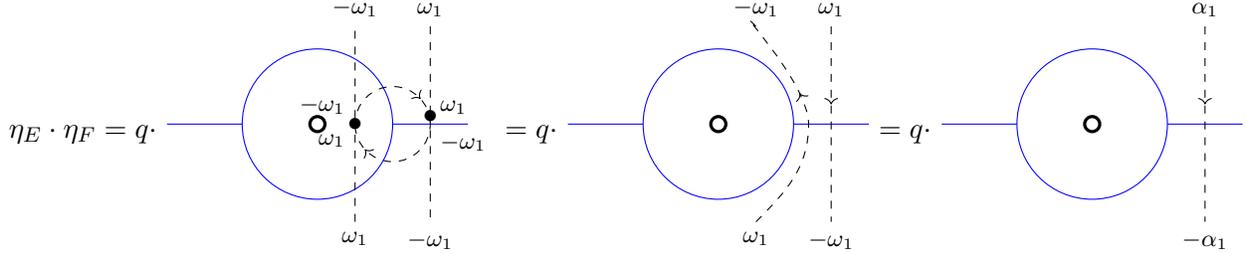

\section{The Non-Simply-Laced Dynkin Types}\label{sec5}

In general, a Dynkin diagram can be encoded by a symmetrizable Cartan matrix $C_{ij}=\inprod{\alpha_i^\vee}{\alpha_j}$. The lattice $Q$ spanned by $\{\alpha_i\}$ is called the \emph{root lattice} and the lattice $Q^\vee$ spanned by $\{\alpha_i^\vee\}$ is called the \emph{coroot lattice}. The \emph{weight lattice} $P$ is defined to be the dual lattice of the coroot lattice $Q^\vee$ and the coroot lattice and the \emph{coweight lattice} $P^\vee$ is defined to be the dual lattice of the root lattice $Q$. Since the Cartan matrix has integer entries, it follows that $Q\subset P$ and $Q^\vee\subset P^\vee$.

There is a diagonal matrix $D$ with relatively prime positive integer entries such that $D^{-1}C$ is symmetric. The entries $d_i$'s of $D$ are called \emph{multipliers}.

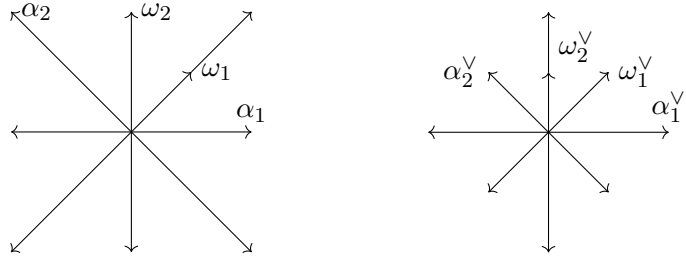
\begin{figure}[H]
    \centering
    \begin{tikzpicture}[scale=0.8]
        \draw[<->] (-2,0) -- (2,0) node [above] {$\alpha_1$};
        \draw [<->] (0,-2) -- (0,2) node [right] {$\omega_2$};
        \draw [<->] (-2,-2) -- (1,1) node [right] {$\omega_1$};
        \draw [->] (1,1) -- (2,2);
        \draw [<->] (-2,2) node [right] {$\alpha_2$} -- (2,-2);
    \end{tikzpicture}\hspace{2cm}\begin{tikzpicture}[scale=0.8]
        \draw[<->] (-2,0) -- (2,0) node [above] {$\alpha_1^\vee$};
        \draw [<->] (0,-2) -- (0,1) node [above right] {$\omega_2^\vee$};
        \draw [->] (0,1) -- (0,2);
        \draw [<->] (-1,-1) -- (1,1) node [right] {$\omega_1^\vee$};
        \draw [<->] (-1,1) node [left] {$\alpha_2^\vee$} -- (1,-1);
    \end{tikzpicture}
    \caption{Left: the root lattice and weight lattice of B$_2$. Right: the coroot lattice and the coweight lattice of B$_2$. Note that $d_1=2$ and $d_2=1$ for the multipliers in this case.}
    \label{fig:root and coroot of B2}
\end{figure}

To define the Weyl group, we first define a positive integer $m_{ij}$ for any pair of simple roots $\alpha_i\neq \alpha_j$ such that $\cos\left(\frac{\pi}{m_{ij}}\right)=\left(\frac{\sqrt{C_{ij}C_{ji}}}{2}\right)$. Then the Weyl group $W$ is a Coxeter group with one generator $s_i$ for each simple root $\alpha_i$, subject to the relations:
\begin{itemize}
    \item $s_i^2=e$,
    \item $(s_is_j)^{m_{ij}}=e$ for $i\neq j$.
\end{itemize}
The action of the Coxeter generators $s_i$ on the weight lattice $P$ are given by $s_i.\mu=\mu-\inprod{\alpha_i^\vee}{\mu}\alpha_i$, and dually, the action of the Coxeter generators $s_i$ on the coweight lattice $P^\vee$ are given by $s_i.\mu^\vee=\mu^\vee-\inprod{\alpha_i}{\mu^\vee}\alpha_i^\vee$. 

Corresponding to the $m_{ij}=4$ (Dynkin type B$_2$) and the $m_{ij}=6$ (Dynkin type G$_2$) cases, we add two more types of weave vertices. These weave vertices should be viewed as the foldings of the respective weave patterns of Dynkin type A$_3$ and D$_4$ below them in Figure \ref{fig:new weave vertices}. In particular, each weave edge incident to either weave vertex corresponds to a collection of external edges in the unfolded pattern. We call these collections of external edges the \emph{families of lifts}.
\begin{figure}[H]
    \centering
    \begin{tikzpicture}
    \foreach \i in {0,...,3}
    {
    \draw [red] (\i*90:1) -- (0,0);
    \draw [blue] (45+\i*90:1) -- (0,0);
    \node [red] at (\i*90:1.3) [] {$s_j$};
    \node [blue] at (45+\i*90:1.3) [] {$s_i$};
    }
\end{tikzpicture}\hspace{3cm}
\begin{tikzpicture}
    \foreach \i in {0,...,5}
    {
    \draw [blue] (\i*60:1) -- (0,0);
    \draw [red] (30+\i*60:1) -- (0,0);
    \node [blue] at (\i*60:1.3) [] {$s_i$};
    \node [red] at (30+\i*60:1.3) [] {$s_j$};
    }
\end{tikzpicture}\\ \vspace{0.5cm}
\begin{tikzpicture}[baseline=0,scale=0.8]
    \foreach \i in {0,1}
    {
    \draw [red] (90+180*\i:1) -- (180*\i:1) -- (180*\i-90:1)--(180*\i-90:2);
    \draw [red] (180*\i:1) -- (180*\i:2);
    \draw [blue] (1-\i*2,2) -- (0,1);
    \draw [blue] (1-\i*2,-2) -- (0,-1);
    \draw [teal] (2,1-\i*2)--(1,0);
    \draw [teal] (-2,1-\i*2) -- (-1,0);
    }
    \draw [blue] (0,1) -- (0,-1);
    \draw [teal] (1,0) -- (-1,0);
\end{tikzpicture}\hspace{3cm}
\begin{tikzpicture}[baseline=0, scale=1.2]
    \foreach \i in {0,1,2}
    {
    \draw [red] (0,0) -- (30+120*\i:0.5);
    \draw [blue] (30+120*\i:0.5)--(150+120*\i:0.5);
    \draw [red] (30+120*\i:0.5) -- (60+120*\i:0.86);
    \draw [red] (30+120*\i:0.5) -- (120*\i:0.86);
    \draw [red] (60+120*\i:0.86) -- (90+120*\i:1) -- (120+120*\i:0.86);
    \draw [red] (60+120*\i:0.86) -- (30+120*\i:1) -- (120*\i:0.86);
    \draw [brown] (120*\i:0.86) --(60+120*\i:0.86) -- (120+120*\i:0.86);
    \draw [brown] (60+120*\i:0.86) -- (60+120*\i:1.5);
    \draw [brown] (120+120*\i:0.86) -- (120+120*\i:1.5);
    \draw [teal] (0,0) -- (90+120*\i:1);
    \draw [blue] (30+120*\i:0.5) -- (30+120*\i:1);
    \draw [red] (90+120*\i:1) -- (90+120*\i:1.5);
    \draw [red] (30+120*\i:1) -- (30+120*\i:1.5);
    \draw [blue] (40+120*\i:1.5)--(30+120*\i:1)--(20+120*\i:1.5);
    \draw [teal] (80+120*\i:1.5) -- (90+120*\i:1) -- (100+120*\i:1.5);
    }
\end{tikzpicture}
    \caption{Left: octavalent weave vertex when $m_{ij}=4$ and its unfolding. Right: dodecavalent weave vertex when $m_{ij}=6$ and its unfolding.}
    \label{fig:new weave vertices}
\end{figure}
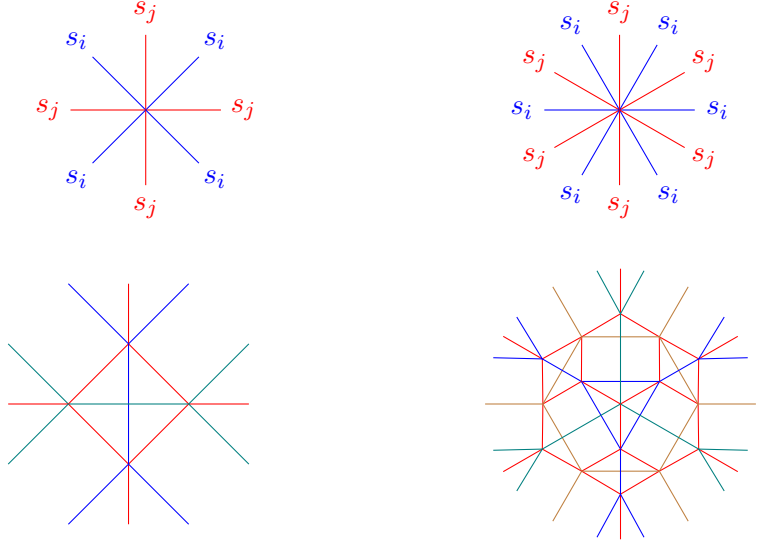

Following the idea of folding, Y-cycles that pass through these weave vertices with higher valence should be descendants of Y-cycles on the local unfolded weave pattern. To be more precise, let $\tilde{\gamma}$ be a Y-cycle on the unfolded weave pattern such that in the family of lifts $\{\tilde{a}_1,\dots, \tilde{a}_s\}$ of each weave edge $a$, $\tilde{\gamma}(\tilde{a}_1)=\cdots =\tilde{\gamma}(\tilde{a}_s)$; then the \emph{descendant} of $\tilde{\gamma}$ is a Y-cycle $\gamma$ with values $\gamma(a)=\tilde{\gamma}(\tilde{a}_i)$.

\begin{prop} \begin{enumerate}
    \item Let $v$ be a octavalent weave vertex. Denote the weave edges incident to $v$ by $a_1,a_2,\dots, a_8$ in the counterclockwise direction. For each $k=1,\dots, 8$, we define an $7$-tuple
\[
c_k=\left\{\begin{array}{ll}
    (1,1,1,0,-1,-1,-1) & \text{if $a_k$ has a color with multiplier $1$,} \\
    (1,2,1,0,-1,-2,-1) & \text{if $a_k$ has a color with multiplier $2$.}
\end{array}\right.
\]
We denote the $l$th term of the septuple $c_k$ as $c_k(l)$. Then $\gamma$ is a Y-cycle at $v$ if and only if $\gamma$ satisfies the following equality for each edge $a_k$ (indices taken modulo $8$):
\begin{equation}\label{eq: Y-cycle}
\sum_{l=1}^7c_k(l)\gamma(a_{k+l})=0.
\end{equation}
\item Let $v$ be a dodecavalent weave vertex. Denote the weave edges incident to $v$ by $a_1,a_2,\dots, a_{12}$ in the counterclockwise direction. For each $k=1,...,12$, we define an $11$-tuple
\[
c_k=\left\{\begin{array}{ll}
    (1,1,2,1,1,0,-1,-1,-2,-1,-1) & \text{if $a_k$ has a color with multiplier $1$}, \\
    (1,3,2,3,1,0,-1,-3,-2,-3,-1) & \text{if $a_k$ has a color with multiplier $3$}.
\end{array}\right.
\]
We denote the $l$th term in the undecuple $c_k$ as $c_k(l)$. Then $\gamma$ is a Y-cycle at $v$ if and only if $\gamma$ satisfies the following equality for each edge $a_k$ (indices taken modulo $12$):
\begin{equation}
    \sum_{l=1}^{11} c_k(l)\gamma(a_{k+l})=0.
\end{equation}
\end{enumerate}

\end{prop}
\begin{proof} We will prove part (1) here and leave part (2) as an exercise for the readers. 

Suppose $m=4$ and $\gamma$ is a descendant from a Y-cycle $\tilde{\gamma}$ on the local unfolded weave pattern. Let us assume that the values of $\tilde{\gamma}$ along the remaining internal weave edges are as in Figure \ref{fig: unfolded B2 weave}. Then it follows that
\[
\gamma(a_1)+\gamma(a_2)+\gamma(a_3)=\xi_2+\xi_6+\gamma(a_3)=\gamma(a_5)+\xi_3+\xi_6=\gamma(a_5)+\gamma(a_6)+\gamma(a_7),
\]
and
\[
\gamma(a_2)+2\gamma(a_3)+\gamma(a_4)=\xi_1+\xi_6+\xi_3+\xi_5=\gamma(a_6)+2\gamma(a_7)+\gamma(a_8).
\]
Equation \eqref{eq: Y-cycle} for the rest of the weave edges can be obtained by similar computations.
\begin{figure}[H]
    \centering
    \begin{tikzpicture}[baseline=0,scale=1]
    \foreach \i in {0,1}
    {
    \draw [red] (90+180*\i:1) -- (180*\i:1) -- (180*\i-90:1)--(180*\i-90:2);
    \draw [red] (180*\i:1) -- (180*\i:2);
    \draw [blue] (1-\i*2,2) -- (0,1);
    \draw [blue] (1-\i*2,-2) -- (0,-1);
    \draw [teal] (2,1-\i*2)--(1,0);
    \draw [teal] (-2,1-\i*2) -- (-1,0);
    }
    \draw [blue] (0,1) -- (0,-1);
    \draw [teal] (1,0) -- (-1,0);
    \node at (-1,2) [above left] {$\gamma(a_1)$};
    \node at (-2,1) [above left] {$\gamma(a_1)$};
    \node at (-2,0) [left] {$\gamma(a_2)$};
    \node at (-2,-1) [below left] {$\gamma(a_3)$};
    \node at (-1,-2) [below left] {$\gamma(a_3)$};
    \node at (0,-2) [below] {$\gamma(a_4)$};
    \node at (1,-2) [below right] {$\gamma(a_5)$};
    \node at (2,-1) [below right] {$\gamma(a_5)$};
    \node at (2,0) [right] {$\gamma(a_6)$};
    \node at (2,1) [above right] {$\gamma(a_7)$};
    \node at (1,2) [above right] {$\gamma(a_7)$};
    \node at (0,2) [above] {$\gamma(a_8)$};
    \foreach \i in {1,...,4}
    {
    \node at (45+90*\i:0.9) [] {$\xi_\i$};
    }
    \node at (0.2,0.4) [] {$\xi_5$};
    \node at (0.4,-0.2) [] {$\xi_6$};
    \end{tikzpicture}
    \caption{Multiplicities in a lift of a Y-cycle.}
    \label{fig: unfolded B2 weave}
\end{figure}
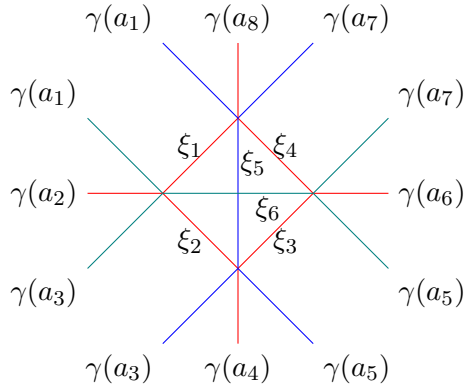

Conversely, suppose we have a candidate Y-cycle $\gamma$ that satisfies Equation \eqref{eq: Y-cycle} for all incident edges $a_k$'s; we need to find a lift $\tilde{\gamma}$ that is a Y-cycle on the local weave pattern. Without loss of generality, let us assume that $\gamma(a_5)$ is the smallest among $\{\gamma(a_1),\gamma(a_3),\gamma(a_5),\gamma(a_7)\}$. Then we set
\[
\xi_1:=\min\{\gamma(a_1)+\gamma(a_3), \gamma(a_2)+\gamma(a_3), \gamma(a_1)+\gamma(a_7), \gamma(a_7)+\gamma(a_8)\}.
\]
It then follows that 
\[
\xi_2=\gamma(a_1)+\gamma(a_3)-\xi_1, \quad \xi_3=\gamma(a_1)+2\gamma(a_7)-\gamma(a_5)-\xi_1, \quad \xi_4=\gamma(a_1)+\gamma(a_7)-\xi_1,
\]
\[
\xi_5=\gamma(a_7)+\gamma(a_8)-\xi_1, \quad \xi_6=\gamma(a_2)+\gamma(a_3)-\xi_1
\]
gives a Y-cycle $\tilde{\gamma}$ on the local unfolded weave pattern.
\end{proof}

\begin{rmk} For a Y-cycle $\gamma$ at a higher valence weave vertex, its lift $\tilde{\gamma}$ to a local unfolded weave pattern is not necessarily unique. However, since there are no trivalent weave vertices in the local unfolded pattern, any two weighted cycle lifts are homotopic to each other.
\end{rmk}

Similar to the roots in a non-simply-laced root system, each Y-cycle on a non-simply-laced weave also has a multiplier. This multiplier plays an important role in defining a \textbf{skew-symmetrizable pairing} between Y-cycles.

\begin{defn} Let $d$ be the non-trivial multiplier of the Dynkin type ($d=2$ for Dynkin type B, C, and F$_4$, and $d=3$ for Dynkin type G$_2$). We set the \emph{multiplier} $d_\gamma$ of a Y-cycle $\gamma$ to be $d$ if for all weave edges $e$ of color $s_i$ with $d_i=1$, $\gamma(e)$ is a multiple of $d$. Otherwise we set $d_\gamma=1$.
\end{defn}

\begin{defn} Let $\ww$ be a weave of general Dynkin type (not necessarily simply-laced). Let $\tilde{\ww}$ be an unfolded weave of $\ww$ that is of simply-laced Dynkin type. Let $\gamma_1$ and $\gamma_2$ be two Y-cycles on $\ww$ and let $\tilde{\gamma}_i$ be a lift of $\gamma_i$ to $\tilde{\ww}$ for $i=1,2$. Then we define the \emph{skew-symmetrizable pairing} between $\gamma_1$ and $\gamma_2$ to be $\inprod{\gamma_1}{\gamma_2}:=\{\tilde{\gamma_1},\tilde{\gamma_2}\}d_{\gamma_1}^{-1}$.
\end{defn}

\begin{rmk} One can find that the local skew-symmetrizable pairing between Y-cycles at an octavalent weave vertex can be given by the following closed form formula:
\[
\inprod{\gamma}{\gamma'}=\frac{1}{d_{\gamma}}\sum_{i=0}^3\det\begin{pmatrix} 1 & 1 & 1\\ \gamma(a_{2i+1}) & \gamma(a_{2i}) & \gamma(a_{2i-1}) \\ \gamma'(a_{2i+1}) & \gamma'(a_{2i}) & \gamma'(a_{2i-1})\end{pmatrix},
\]
where $a_i$'s are the labeling of incident weave edges according to Figure \ref{fig: unfolded B2 weave}, and all indices are taken modulo $8$. We expect that a closed form formula also exists for the local skew-symmetrizable pairing between Y-cycles at a dodecavalent weave vertex, but we are not able to find it yet.
\end{rmk}

Weighted cycles on weaves of general Dynkin type are defined in the same way as the simply-laced case: the only things we need to add are the homotopy moves across octavalent and dodecavalent weave vertices, similar to the homotopy move (6) near hexavalent weave vertices. The weighted cycle algebra and merodromies are also defined in a similar fashion. 

We can also lift Y-cycles to weighted cycles in the non-simply-laced case. However, unlike the simply-laced case, each Y-cycle $\gamma$ in the non-simply-laced case admits two lifts to weighted cycles, one labeled by weights, which we denote by $\eta$, and the other labeled by coweights, which we denote by $\eta^\vee$. The construction recipes for both are analogous to that of weighted cycle representatives in Subsection \ref{subsec: 3.4}, and the patterns we use at a octavalent and dodecavalent weave vertex are similar to that at the hexavalent weave vertex.

\begin{figure}[H]
    \centering
    \begin{tikzpicture}[baseline=0]
    \foreach \i in {0,...,3}
    {
    \draw [red] (\i*90:2) -- (0,0);
    \draw [blue] (45+\i*90:2) -- (0,0);
    }
    \foreach \i in {1,...,8}
    {
    \draw [dashed, decoration={markings, mark=at position 0.6 with {\arrow{<}}}, postaction={decorate}] (90+\i*45+11:2) -- (90+\i*45+22.5:1);
    \draw [dashed, decoration={markings, mark=at position 0.6 with {\arrow{>}}}, postaction={decorate}] (90+\i*45+34:2) -- (90+\i*45+22.5:1);
    \draw [dashed, decoration={markings, mark=at position 0.5 with {\arrow{>}}}, postaction={decorate}] (90+\i*45+22.5:1) -- (0,0);
    \node at (90+\i*45+11:2.4) [] {$a_\i\omega_\i$};
    \node at (90+\i*45-11:2.4) [] {$a_\i\omega_\i$};
    \node at (90+\i*45+22.5:1) [] {$\bullet$};
    }
    \node at (0,0) [] {$\bullet$};
\end{tikzpicture}\hspace{2cm}
\begin{tikzpicture}[baseline=0]
    \foreach \i in {0,...,5}
    {
    \draw [blue] (\i*60:3) -- (0,0);
    \draw [red] (30+\i*60:3) -- (0,0);
    }
    \foreach \i in {1,...,12}
    {
    \draw [dashed, decoration={markings, mark=at position 0.6 with {\arrow{<}}}, postaction={decorate}] (90+\i*30+10:3) -- (90+\i*30+15:1.8);
    \draw [dashed, decoration={markings, mark=at position 0.6 with {\arrow{>}}}, postaction={decorate}] (90+\i*30+20:3) -- (90+\i*30+15:1.8);
    \draw [dashed, decoration={markings, mark=at position 0.5 with {\arrow{>}}}, postaction={decorate}] (90+\i*30+15:1.8) -- (0,0);
    \node at (90+\i*30+8:3.4) [] {\footnotesize{$a_{\i}\omega_{\i}$}};
    \node at (90+\i*30-8:3.4) [] {\footnotesize{$a_{\i}\omega_{\i}$}};
    \node at (90+\i*30+15:1.8) [] {$\bullet$};
    }
    \node at (0,0) [] {$\bullet$};
\end{tikzpicture}
    \caption{Local models for weighted cycle representatives near a octavalent (left) and a dodecavalent (right) vertex.}
\end{figure}

\begin{rmk} One can view \emph{coweighted cycles} as weighted cycles for the Langlands dual group $G^\vee$.
\end{rmk}

We can define a pairing between coweighted cycles with weighted cycles, just with the inner product $(\cdot, \cdot)$ replaced by the pairing $\inprod{\cdot}{\cdot}$ between weights and coweights.

\begin{thm} For any Y-cycles $\gamma_1$ and $\gamma_2$ on a weave of general Dynkin type, we have $\inprod{\gamma_1}{\gamma_2}=\inprod{\eta_1}{\eta^\vee_2}$, where $\eta_1$ is the coweighted cycle representative of $\gamma_1$ and $\eta^\vee_2$ is the weighted cycle representative of $\gamma_2$.
\end{thm}
\begin{proof} The proof is done by first covering up the weave vertex by a simple closed curve $C$ with a generic base point $p$. Then we reduce the problem to a computation along the interval $C\setminus\{p\}$, similar to that in the case of hexavalent vertices in the proof of Theorem \ref{thm: recovering pairing between Y-cycles}.
\end{proof}

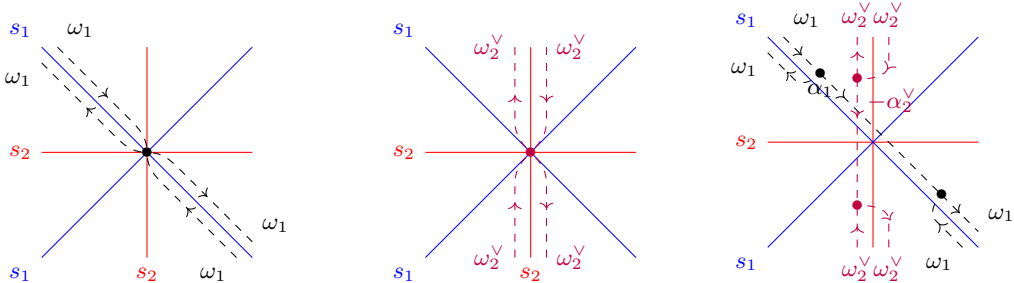
\begin{figure}[H]
    \centering
    \begin{tikzpicture}[scale=0.7]
        \draw [red] (0,-2) node [below] {\footnotesize{$s_2$}} -- (0,2);
        \draw [red] (2,0) -- (-2,0)node [left] {\footnotesize{$s_2$}};
        \draw [blue] (-2,2) node [above left] {\footnotesize{$s_1$}}-- (2,-2);
        \draw [blue] (-2,-2) node [below left] {\footnotesize{$s_1$}} -- (2,2);
        \draw [dashed, decoration={markings, mark=at position 0.5 with {\arrow{>}}}, postaction={decorate}] (-1.7,2) node [above right] {\footnotesize{$\omega_1$}} -- (-0.3,0.6) to [out=-45,in=90] (0,0);
        \draw [dashed, decoration={markings, mark=at position 0.5 with {\arrow{<}}}, postaction={decorate}] (-2,1.7) node [below left] {\footnotesize{$\omega_1$}} -- (-0.6,0.3) to [out=-45,in=180] (0,0);
        \draw [dashed, decoration={markings, mark=at position 0.5 with {\arrow{>}}}, postaction={decorate}] (1.7,-2) node [below left] {\footnotesize{$\omega_1$}} -- (0.3,-0.6) to [out=135,in=-90] (0,0);
        \draw [dashed, decoration={markings, mark=at position 0.5 with {\arrow{<}}}, postaction={decorate}] (2,-1.7) node [above right] {\footnotesize{$\omega_1$}} -- (0.6,-0.3) to [out=135,in=0] (0,0);
        \node at (0,0) [] {\footnotesize{$\bullet$}};
    \end{tikzpicture}\hspace{1cm}
    \begin{tikzpicture}[scale=0.7]
        \draw [red] (0,-2) node [below] {\footnotesize{$s_2$}} -- (0,2);
        \draw [red] (2,0) -- (-2,0)node [left] {\footnotesize{$s_2$}};
        \draw [blue] (-2,2) node [above left] {\footnotesize{$s_1$}}-- (2,-2);
        \draw [blue] (-2,-2) node [below left] {\footnotesize{$s_1$}} -- (2,2);
        \draw [purple, dashed, decoration={markings, mark=at position 0.5 with {\arrow{>}}}, postaction={decorate}] (0.3,2) node [right] {\footnotesize{$\omega_2^\vee$}} -- (0.3,0.6) to [out=-90,in=45] (0,0);
        \draw [purple, dashed, decoration={markings, mark=at position 0.5 with {\arrow{<}}}, postaction={decorate}] (-0.3,2) node [left] {\footnotesize{$\omega_2^\vee$}} -- (-0.3,0.6) to [out=-90,in=135] (0,0);
        \draw [purple, dashed, decoration={markings, mark=at position 0.5 with {\arrow{>}}}, postaction={decorate}] (-0.3,-2) node [left] {\footnotesize{$\omega_2^\vee$}} -- (-0.3,-0.6) to [out=90,in=-135] (0,0);
        \draw [purple, dashed, decoration={markings, mark=at position 0.5 with {\arrow{<}}}, postaction={decorate}] (0.3,-2) node [right] {\footnotesize{$\omega_2^
        \vee$}} -- (0.3,-0.6) to [out=90,in=-45] (0,0);
        \node [purple] at (0,0) [] {\footnotesize{$\bullet$}};
    \end{tikzpicture}\hspace{1cm}
    \begin{tikzpicture}[scale=0.7]
        \draw [red] (0,-2) -- (0,2);
        \draw [red] (2,0) -- (-2,0)node [left] {\footnotesize{$s_2$}};
        \draw [blue] (-2,2) node [above left] {\footnotesize{$s_1$}}-- (2,-2);
        \draw [blue] (-2,-2) node [below left] {\footnotesize{$s_1$}} -- (2,2);
        \draw [dashed, decoration={markings, mark=at position 0.5 with {\arrow{>}}}, postaction={decorate}] (-1.7,2) node [above right] {\footnotesize{$\omega_1$}} -- (-1,1.3);
        \draw [dashed, decoration={markings, mark=at position 0.2 with {\arrow{>}}}, postaction={decorate}] (-1,1.3) node [below] {\footnotesize{$\alpha_1$}} -- (1.3,-1);
        \draw [dashed, decoration={markings, mark=at position 0.5 with {\arrow{>}}}, postaction={decorate}] (1.3,-1) -- (2,-1.7) node [above right] {\footnotesize{$\omega_1$}};
        \draw [dashed, decoration={markings, mark=at position 0.5 with {\arrow{<}}}, postaction={decorate}] (-2,1.7) node [below left] {\footnotesize{$\omega_1$}} -- (-1.5,1.2) to [out=-45,in=-135] (-1,1.3);
        \draw [dashed, decoration={markings, mark=at position 0.5 with {\arrow{>}}}, postaction={decorate}] (1.7,-2) node [below left] {\footnotesize{$\omega_1$}} -- (1.2,-1.5) to [out=135,in=-135] (1.3,-1);
        \node at (-1,1.3) [] {\footnotesize{$\bullet$}};
        \node at (1.3,-1) [] {\footnotesize{$\bullet$}};
        \draw [purple, dashed, decoration={markings, mark=at position 0.5 with {\arrow{>}}}, postaction={decorate}] (0.3,2) node [above] {\footnotesize{$\omega_2^\vee$}} -- (0.3,1.5) to [out=-90,in=0] (-0.3,1.2);
        \draw [purple, dashed, decoration={markings, mark=at position 0.5 with {\arrow{<}}}, postaction={decorate}] (-0.3,2) node [above] {\footnotesize{$\omega_2^\vee$}} -- (-0.3,1.2);
        \draw  [purple, dashed, decoration={markings, mark=at position 0.3 with {\arrow{>}}}, postaction={decorate}] (-0.3,1.2) node [below right] {\footnotesize{$-\alpha_2^\vee$}}-- (-0.3,-1.2);
        \draw [purple, dashed, decoration={markings, mark=at position 0.5 with {\arrow{>}}}, postaction={decorate}] (-0.3,-2) node [below] {\footnotesize{$\omega_2^\vee$}} -- (-0.3,-1.2);
        \draw [purple, dashed, decoration={markings, mark=at position 0.5 with {\arrow{<}}}, postaction={decorate}] (0.3,-2) node [below] {\footnotesize{$\omega_2^
        \vee$}} -- (0.3,-1.5) to [out=90,in=0] (-0.3,-1.2);
        \node [purple] at (-0.3,1.2) [] {\footnotesize{$\bullet$}};
        \node [purple] at (-0.3,-1.2) [] {\footnotesize{$\bullet$}};
    \end{tikzpicture}
    \caption{Local pictures at an octavalent vertex of a B$_2$ weave. Left: weight cycle representative $\eta_1$ of a Y-cycle $\gamma_1$. Middle: coweight cycle representative $\eta^\vee_2$ of a Y-cycle $\gamma_2$. Right: homotopic images of the two cycles, from which we can see that the skew-symmetrizable intersection pairing gives $\inprod{\eta_1}{\eta^\vee_2}=\inprod{\alpha_1}{-\alpha_2^\vee}=1=\inprod{\gamma_1}{\gamma_2}=\frac{1}{2}\{\tilde{\gamma}_1,\tilde{\gamma}_2\}$ (c.f. Figures \ref{fig:root and coroot of B2} and \ref{fig: unfolded B2 weave}).}
\end{figure}

\bibliographystyle{biblio}

\bibliography{biblio}

\end{document}